\documentclass[12pt]{amsart}

\numberwithin{equation}{section}
\usepackage{kpfonts}

\usepackage{amsthm}

\usepackage{amsmath}

\usepackage{amssymb}

\usepackage{dsfont} %for blackboard 1

\usepackage{thmtools}

\usepackage{tikz}
\usetikzlibrary{decorations.markings, arrows.meta,positioning}

\usepackage{cancel}
\usepackage{verbatim}

\usepackage[shortlabels]{enumitem}
\setlist[enumerate]{leftmargin=*,font=\upshape,align=parleft,label=(\alph*),labelsep=0pt}
\setlist[itemize]{leftmargin=*,labelwidth=*}
\usepackage{graphicx}
\usepackage{float}

\usepackage{color}

\usepackage{xspace}

\usepackage{calrsfs}

\DeclareSymbolFont{defaultmathcal}{OMS}{zplm}{m}{n}
\DeclareSymbolFontAlphabet{\mathcal}{defaultmathcal}

\DeclareSymbolFont{handwritten}{OMS}{rsfs}{m}{n}
\DeclareSymbolFontAlphabet{\handcal}{handwritten}
\definecolor{darkgreen}{RGB}{54,124,50}
\usepackage{hyperref}
\hypersetup{
	pdfstartview={XYZ null null 1.00}, 
	pdfpagemode=UseNone, 
	colorlinks,
	breaklinks,
	linkcolor=blue,
	urlcolor=blue, 
	anchorcolor=blue,
	citecolor=darkgreen,
}

\usepackage[capitalise]{cleveref}

\usepackage{xparse} %for defining Roman numeral command \RN
\usepackage{geometry}
\usepackage{microtype}

\newcommand{\N}{\mathbb{N}}
\newcommand{\Z}{\mathbb{Z}}

\newcommand{\R}{\mathbb{R}}

\newcommand{\0}{\mathbb{\emptyset}}
\newcommand{\w}{\infty}

\newcommand{\F}{\mathbb{F}}

\newcommand{\e}{\varepsilon}
\newcommand\ka{\kappa}

\renewcommand{\phi}{\varphi}

\DeclareRobustCommand{\rchi}{{\mathpalette\irchi\relax}}
\newcommand{\irchi}[2]{\raisebox{\depth}{$#1\cchi$}}
\let\cchi\chi
\let\chi\rchi

\newcommand{\EC}{\mathcal{E}}
\newcommand{\FC}{\mathcal{F}}

\newcommand{\HC}{\mathcal{H}}
\newcommand{\IC}{\mathcal{I}}

\newcommand{\dom}{\operatorname{dom}}

\newcommand{\im}{\operatorname{im}}

\newcommand{\Pow}{\handcal{P}}

\newcommand{\del}{\partial}

\makeatletter
\def\moverlay{\mathpalette\mov@rlay}

\def\mov@rlay#1#2{\leavevmode\vtop{%
		\baselineskip\z@skip \lineskiplimit-\maxdimen
		\ialign{\hfil$\m@th#1##$\hfil\cr#2\crcr}}}

\newcommand{\charfusion}[3][\mathord]{
	#1{\ifx#1\mathop\vphantom{#2}\fi
		\mathpalette\mov@rlay{#2\cr#3}
	}
	\ifx#1\mathop\expandafter\displaylimits\fi}
\makeatother

\renewcommand{\leq}{\leqslant}
\renewcommand{\le}{\leqslant}
\newcommand{\nle}{\nleqslant}
\renewcommand{\geq}{\geqslant}
\renewcommand{\ge}{\geqslant}

\newcommand*{\defeq}{\mathrel{\vcenter{\baselineskip0.5ex \lineskiplimit0pt \hbox{\scriptsize.}\hbox{\scriptsize.}}}=}
\newcommand*{\eqdef}{=\mathrel{\vcenter{\baselineskip0.5ex \lineskiplimit0pt \hbox{\scriptsize.}\hbox{\scriptsize.}}}}

\DeclareMathSymbol{\lqm}{\mathord}{operators}{``}
\DeclareMathSymbol{\rqm}{\mathord}{operators}{`'}

\newcommand{\set}[1]{\left\{ #1 \right\}}

\newcommand{\clint}[1]{\left[ #1 \right]}

\newcommand{\opint}[1]{\left( #1 \right)}

\newcommand{\rest}[1]{\mathord{|_{#1}}}

\newcommand{\eqcomment}[1]{\Big[\text{#1}\Big] \hspace{6pt}}

\let\defterm\textbf
\theoremstyle{plain}

\newtheorem{theorem}[equation]{Theorem}
\newtheorem*{theorem*}{Theorem}

\makeatletter
\def\@empty{}
\def\ifemptycredit#1{%
	\def\tmp{#1}%
	\ifx\tmp\@empty%
	\else%
	{~(#1)}%
	\fi%
}
\makeatother

\newenvironment{namedthm*}[2][]{
	\par\medskip\noindent \textbf{#2}\ifemptycredit{#1}\textbf{.}\itshape\xspace
}{}
\crefname{prop}{Proposition}{Propositions}
\newtheorem{prop}[equation]{Proposition}
\newtheorem*{propo*}{Proposition}

\crefname{property}{Property}{Properties}

\newtheorem*{property*}{Property}

\newtheorem{lemma}[equation]{Lemma}
\newtheorem*{lemma*}{Lemma}

\crefname{claimlemma}{Claim}{Claims}

\crefname{cor}{Corollary}{Corollaries}
\newtheorem{cor}[equation]{Corollary}
\newtheorem*{cor*}{Corollary}

\crefname{obs}{Observation}{Observations}
\newtheorem{obs}[equation]{Observation}
\newtheorem*{obs*}{Observation}

\crefname{obss}{Observations}{Observations}

\newtheorem{obss*}{Observations}

\crefname{fact}{Fact}{Facts}

\newtheorem*{fact*}{Fact}

\theoremstyle{definition}

\crefname{defn}{Definition}{Definitions}
\newtheorem{defn}[equation]{Definition}
\newtheorem*{defn*}{Definition}

\newenvironment{defn**}[1][]{\par\medskip\noindent \textbf{Definition\xspace#1.}\xspace}{}

\crefname{question}{Question}{Questions}

\newtheorem*{question*}{Question}

\crefname{conj}{Conjecture}{Conjectures}

\newtheorem*{conj*}{Conjecture}

\crefname{example}{Example}{Examples}
\newtheorem{example}[equation]{Example}
\newtheorem*{example*}{Example}

\crefname{examples.plain}{Examples}{Examples}
\newtheorem{examples.plain}[equation]{Examples}
\newtheorem*{examples.plain*}{Examples}

\crefname{workinghyp}{Working Hypothesis}{Working Hypotheses}
\newtheorem{workinghyp}[equation]{Working Hypothesis}

\theoremstyle{remark}

\crefname{remark}{Remark}{Remarks}
\newtheorem{remark}[equation]{Remark}
\newtheorem*{remark*}{Remark}

\newenvironment{remarklike*}[2][]{\par\medskip\noindent \textit{#2}#1\textbf{.}\rmfamily\xspace}{\smallskip}

\crefname{claim+}{Claim}{Claims}
\newtheorem{claim+}[equation]{Claim}

\crefname{claim}{Claim}{Claims}

\newtheorem*{claim*}{Claim}

\crefname{subclaim}{Subclaim}{Subclaims}

\newtheorem*{subclaim*}{Subclaim}

\newenvironment{case*}[1]{\smallskip\par\noindent \textit{Case}:~#1.\rmfamily}{}
\newenvironment{case}[2]{\smallskip\par\noindent \textit{Case}~#1: \rmfamily #2.}{}

\crefname{notation}{Notation}{Notations}
\newtheorem{notation}[equation]{Notation}
\newtheorem*{notation*}{Notation}

\newtheorem*{terminology*}{Terminology}

\crefname{convention}{Convention}{Conventions}
\newtheorem{convention}[equation]{Convention}
\newtheorem*{convention*}{Convention}
\newtheorem*{conventions*}{Conventions}

\crefname{spec}{Speculation}{Speculations}

\newtheorem*{spec*}{Speculation}

\crefname{caution}{Caution}{Cautions}

\newtheorem*{caution*}{Caution}

\crefname{hypothesis}{Hypothesis}{Hypotheses}

\newtheorem*{hypothesis*}{Hypothesis}

\crefname{assumption}{Assumption}{Assumptions}

\newtheorem*{assumption*}{Assumption}

\newcommand{\fntsz}[1][11]{\fontsize{#1}{#1}\selectfont}

\crefname{examples}{Examples}{Examples}

\newenvironment{examples*}[1][\alph*]
{
	\refstepcounter{equation}
	\medskip
	\noindent\textbf{Examples.}
	\medskip
	\begin{enumerate}[\bfseries(\theequation.#1),ref=(\theequation.#1),itemsep=5pt]
}
{
	\end{enumerate}
	\smallskip
}

\theoremstyle{remark}

\declaretheoremstyle[
spaceabove=\topsep, 
spacebelow=6pt,
headfont=\normalfont\itshape,
notefont=\normalfont, notebraces={(}{)},
bodyfont=\normalfont,
postheadspace=4pt,
qed=\mbox{\smaller[4]$\boxtimes$}
]{claimproofstyle}
\declaretheorem[name={Proof of Claim}, style=claimproofstyle, unnumbered]{pf}

\crefname{subsection}{Subsection}{Subsections}

\crefformat{footnote}{#2\footnotemark[#1]#3}

\crefrangeformat{enumi}{#3#1#4--#5#2#6}

\theoremstyle{plain}

\usepackage{subfiles}

\usepackage{mdframed}
\newmdenv[
leftmargin = 1cm,
rightmargin = 0pt,
skipabove = 8pt,
skipbelow = 3pt,
innerleftmargin = 8pt,
innertopmargin = 0pt,
innerbottommargin = 0pt,
innerrightmargin = 0pt,
linewidth = 3pt,
topline = false,
rightline = false,
bottomline = false
]{leftbar}

\newcommand{\red}[1]{{\color{red}#1}}

\definecolor{gris}{RGB}{90,90,90}

\definecolor{vert}{RGB}{7,126,26}

\definecolor{purple}{RGB}{116,0,159}

\makeatletter
\def\@settitle{\begin{center}%
		\baselineskip14\p@\relax
		\bfseries
		\uppercasenonmath\@title
		\@title
		\ifx\@subtitle\@empty\else
		\\[1ex]\uppercasenonmath\@subtitle
		\footnotesize\mdseries\@subtitle
		\fi
	\end{center}%
}
\def\subtitle#1{\gdef\@subtitle{#1}}
\def\@subtitle{}
\makeatother
\input{local_commands.tex}

\usepackage{times}

\title[]{The Radon--Nikodym topography of acyclic measured graphs}

\author[]{Anush Tserunyan}
\address[Anush Tserunyan]{Department of Mathematics and Statistics, McGill University, Montreal, QC, Canada}
\email{anush.tserunyan@mcgill.ca}

\author[]{Robin Tucker-Drob}
\address[Robin Tucker-Drob]{Department of Mathematics, University of Florida, Gainesville, FL, USA}
\email{r.tuckerdrob@ufl.edu}



\date{\today}

\begin{document}

\begin{abstract}
The classical Adams dichotomy \cite{Adams90} characterizes amenability of probability-measure-preserving treed equivalence relations in terms of the ends of a given treeing.
We extend this dichotomy to the measure-class-preserving (mcp) setting by analyzing the \textit{Radon--Nikodym topography} of an acyclic mcp Borel graph, i.e., the interaction between the graph geometry and the associated Radon--Nikodym cocycle.  
In the general mcp setting one must count not all ends, but only the \textit{nonvanishing ends}: those ends which have positive upper cocycle-weight.

More precisely, we prove that a locally countable acyclic mcp graph is amenable if and only if a.e.\ component has at most two nonvanishing ends, while it is nowhere amenable exactly when a.e.\ component has a nonempty perfect closed set of nonvanishing ends.

At the heart of our analysis lies the study of acyclic countable-to-one Borel functions.
Our critical result is that, outside of the essentially two-ended setting, all back ends in a.e.\ orbit are vanishing and have cocycle-finite geodesics.
We also give several characterizations, including topographical and cocycle-free, of essential one-endedness of such functions, and show that the number of \textit{barytropic ends} controls the essential number of ends.

Our proofs utilize mass transport, end selection, and the notion of the \textit{Radon--Nikodym core} for acyclic mcp graphs, a new concept that serves as a guiding framework for our topographical analysis.
\end{abstract}

\maketitle

\tableofcontents

\section{Introduction}

A classical theorem of Adams \cite{Adams90} says that, for probability-measure-preserving (\dfn{pmp}) treed equivalence relations, amenability is geometrically visible at infinity.

More precisely, if $T$ is a treeing of an ergodic pmp countable Borel equivalence relation $E$, then $E$ is amenable if and only if a.e.\ $T$-component has at most two ends.  
A routine argument shows that if $E$ is nonamenable then a.e.\ $T$-component in fact has perfectly many ends.
Thus the measured property of amenability is detected by the large-scale geometry of any treeing, with nonamenability corresponding to the robust presence of infinitely many ends.
We refer to the combined statement as the \textit{classical Adams dichotomy}, see \cref{subsec:classical_Adams} for a concise self-contained proof.

Adams's theorem became a starting point for structural results on amenable subequivalence relations of treeable equivalence relations, as well as non-treeability results for equivalence relations generated by certain groups \cite{Adams-Lyons,Kechris:classification_torsion-free_abelian,Hjorth:non-treeability,Bowen:eq_rel_act_hyperbolic}. %See \cref{subsec:classical_Adams} for a concise proof of the classical Adams dichotomy.  

Following the work of Gaboriau \cite{Gaboriau:cost,Gaboriau:ell-two-Betti}, many of these pmp phenomena could also be understood through the lens of cost and $\ell^2$-Betti numbers.  The notion of cost was introduced by Levitt \cite{Levitt} and developed by Gaboriau \cite{Gaboriau:cost}, and has become a fundamental invariant in the structure and rigidity theory of pmp equivalence relations.  
Gaboriau \cite{Gaboriau:ell-two-Betti} also extended the definition of $\ell^2$-Betti numbers to pmp equivalence relations and showed that if $E$ is a treeable pmp equivalence relation, then $\beta^{(2)}_1(E)=\mathrm{cost}(E)-1$.  In particular, for an ergodic treeable pmp equivalence relation $E$, amenability is equivalent to $\mathrm{cost}(E)=1$, and also to $\beta^{(2)}_1(E)=0$.

It is widely believed, however, that there is no satisfactory extension of the theory of cost to the measure-class-preserving (\dfn{mcp}) setting, as evidenced by \cite[Corollary 3.11]{JKL} and \cite{Inselmann:MSc_thesis,Poulin:elasticity}. 
Moreover, the naive generalization of the classical Adams dichotomy fails spectacularly in the mcp setting.
Indeed, the orbit equivalence relation associated to the one-sided Bernoulli shift $h:2^\N\to 2^\N$ is amenable, yet admits (on a co-countable set) a $3$-regular treeing $T_h$ obtained by connecting each point to its $h$-image (see \cref{example:one_sided_shift}).  
A similar counterexample comes from the (free part of the) action of a nonabelian free group on its boundary equipped with the hitting measure of the simple random walk (see \cref{example:free_group_boundary}).

Nevertheless, in this article we are able to extend the classical Adams dichotomy to the mcp setting 
by accounting for the interplay between the tree geometry and the associated Radon--Nikodym cocycle. 
Informally, our results can be summarized by the motto that, in the mcp setting, ``not all ends have equal weight.''

The Radon--Nikodym cocycle supplies a system of relative weights, or heights, along the classes of the equivalence relation.
Some directions remain visible to the measured large-scale geometry; others are present in the underlying tree but vanish once the tree is viewed through the cocycle.  
We refer to the interaction between the geometry of a treeing and the behavior of its Radon--Nikodym cocycle as the \dfn{Radon--Nikodym topography} of the treeing; it is this topography that underlies all of the results in this article.

Our main theorem is a generalized Adams dichotomy in the mcp setting (\cref{intro:trichotomy}). 
In order to detect amenability in the general mcp setting, the correct reconceptualization is to count not all ends, but only the \dfn{nonvanishing ends}:
those ends which have positive upper cocycle-weight (see \cref{sec:intro:nonvanishing_trichotomy} or \cref{def:barytropic-vanishing-finite_geod} for the precise definition).
In the pmp case the cocycle is identically $1$, so every end is nonvanishing and Adams's original dichotomy is recovered.  
In the general mcp case, however, a treeing may have many geometrically visible ends along which the cocycle tends to $0$; such ends vanish topographically.

In our companion article \cite{TTD:subrel-antitree}, we apply our main theorem to obtain a general structure theory for amenable subrelations of treeable equivalence relations.
We show that in the Adams end selection theorem \cite{Adams:indec_treed_eq_rel_end-selection}, ``the Radon--Nikodym topography selects the ends.''
As a consequence, every nowhere smooth amenable subrelation of a treeable mcp equivalence relation is contained in an essentially unique maximal amenable subrelation.
This yields an anti-treeability criterion, which we apply to deduce that \textit{(i) all (possibly nonunimodular) nonamenable locally compact second countable groups with noncompact amenable radical are not treeable, and (ii) every nonamenable, nonunimodular, treeable locally compact second countable groups contains a compact open subgroup}.
Further applications of the generalized Adams dichotomy appear in other works and are discussed in \cref{sec:intro:upcoming_results}.

This paper has circulated in various preliminary forms since 2021, and the framework and results developed here have already been used in several subsequent works \cite{Chen-Terlov-Tserunyan:nonamenable_subforests,kids,Poulin:elasticity,Terlov-Timar:weighted_amenability,Bell-et-al:heavy_repulsion}, which we discuss in \cref{sec:intro:related_work}.

\subsection{Nonvanishing ends and the generalized Adams dichotomy}\label{sec:intro:nonvanishing_trichotomy}

We work throughout with mcp countable Borel equivalence relations on standard probability spaces. 
These are precisely the orbit equivalence relations arising from mcp Borel actions of countable groups \cite{Feldman-Moore}.
An mcp countable Borel equivalence relation $E$ on a standard probability space $(X,\mu)$ is said to be \dfn{treeable} if there exists an acyclic Borel graph $T$ on $X$ whose connected components are exactly the equivalence classes of $E$; such a graph $T$ is called a \dfn{treeing} of $E$ (see \cite{Adams90,Gaboriau:cost,JKL}).
Treeable equivalence relations play a role in measured equivalence relation theory analogous to that played by free groups in group theory.

A central notion in the setting of mcp Borel equivalence relations is that of $\mu$-amenability (see \cref{subsec:amenability}), which is equivalent to $\mu$-hyperfiniteness by a celebrated theorem of Connes, Feldman, and Weiss \cite{Connes-Feldman-Weiss}.
Throughout the article, we adopt the convention that ``amenable'' always means ``$\mu$-amenable,'' where $\mu$ refers to the underlying probability measure.

The Radon--Nikodym cocycle $\rho :E \to \Rpos$ associated to an mcp countable Borel equivalence relation $E$ on $(X,\mu)$ is an essentially unique Borel function which quantifies the extent to which $\mu$ is not $E$-invariant. 
For $(x,y) \in E$, we write $\rho ^x(y)$ for $\rho(y,x)$, and we view $\rho ^x$ as a ``relative assignment of weights or heights from the perspective of $x$,'' induced by $\rho$ on the $E$-class $[x]_E$; see \cref{subsec:mcp} for precise definitions and further discussion.

The critical definition which initiated the work in this article is that of a \emph{vanishing end}, which was developed in parallel with the paper \cite{Chen-Terlov-Tserunyan:nonamenable_subforests}.
Given $(X, \mu)$ and $(T,\rho)$ as above, we say that an end $\xi$ of a $T$-component $[x]_T$ is \textbf{$\rho$-vanishing} if $\lim_{y \to \xi} \rho^x(y) = 0$, where the limit is taken in either of the natural topologies on the end completion of $[x]_T$ with respect to $T$ (and not just along a geodesic to $\xi$), see \cref{subsec:topologies}.
Thus, in the pmp setting, where $\rho \equiv 1$, all ends are $\rho$-nonvanishing.
We show in \cref{measure_class_invariant} that this notion actually does not depend on the particular measure (hence the particular Radon--Nikodym cocycle $\rho$), in the sense that any two equivalent measures agree, within a.e. T-component, on which ends are vanishing.
We therefore drop $\rho$ from the terminology and just use \dfn{vanishing} when there is no cause for confusion.

We also give a new criterion of a.e.\ smoothness of locally countable acyclic mcp graphs, which is closely related to Miller's theorem \cite[Theorem~A]{Miller:ends_of_graphs_I} asserting that Borel graphs with no ends are smooth.
Using this concept, we generalize the classical Adams dichotomy by replacing the bare number of ends with the number of nonvanishing ends. 
We combine these two results into the following trichotomy, which is one of the main results of our paper.

\begin{theorem}[Trichotomy for acyclic mcp graphs]\label{intro:trichotomy}
Let $T$ be a locally countable acyclic mcp Borel graph on a standard probability
space $(X,\mu)$.
\begin{enumerate}[(a)]
\item\label{item:intro:trich:0} $E_T$ is smooth on a conull set if and only if all ends of almost every
$T$-component are vanishing.
\item\label{item:intro:trich:1-2} $E_T$ is amenable and $\mu$-nowhere smooth if and only if almost every $T$-component has exactly one or two nonvanishing ends.
\item\label{item:intro:trich:>2} $E_T$ is $\mu$-nowhere amenable if and only if for almost every $T$-component,
the set of all nonvanishing ends of that component is nonempty, closed, and has no
isolated points.
\end{enumerate}
\end{theorem}

We also investigate two other natural notions of topographical significance for ends (see \cref{def:barytropic-vanishing-finite_geod}), which turn out to be less robust than the notion of nonvanishing.
We say that an end is \dfn{barytropic} if each half-space containing it is $\rho$-infinite, and we say that an end \dfn{has $\rho$-infinite geodesics} if each geodesic to it is $\rho$-infinite.
See the paragraphs following \cref{def:barytropic-vanishing-finite_geod} for a discussion of the relationships between these notions.
We highlight here that in the case where each $T$-component has at least two nonvanishing ends, the notions of barytropic and nonvanishing coincide on a conull set, as proved in \cref{at_least_2_nonvanishing_core}.

We prove in \cref{measure_class_invariant}\labelcref{item:barytropic_mcp-invariance} that the notion of a barytropic end is also a measure class invariant in the same sense as with vanishing.
We also show that replacing nonvanishing with barytropic, part \labelcref{item:intro:trich:0} of \cref{intro:trichotomy} holds, but the other parts of this trichotomy fail, even for graphs of acyclic Borel functions, e.g., the one-sided Bernoulli shift on $2^\N$.

In \cref{dichotomy:coc-finite_geodesics}, we prove the geodesic analogue of the forward directions of \labelcref{item:intro:trich:1-2,item:intro:trich:>2} of \cref{intro:trichotomy}, where ``nonvanishing'' is replaced with ``having $\rho$-infinite geodesics'' and ``closed'' is replaced with ``$G_\delta$''.
However, as pointed out in \cref{infinite_geodesics_not_robust}, there are nowhere smooth acyclic graphs in which all ends have $\rho$-finite geodesics; see \cref{example:least_deletion}.
Furthermore, there are examples of acyclic mcp graphs such that in each connected component, the set of ends having $\rho$-infinite geodesics is $G_\delta$ but not closed.

\subsection{Topography of back ends for acyclic countable-to-one Borel functions}

A large part of \cref{intro:trichotomy} hinges on a detailed topographical analysis of ends of mcp treeings generated by countable-to-one Borel functions which are \dfn{acyclic}, i.e., $f^n(x)\neq x$ for all $n\geq 1$ and $x\in X$. 
Given such a function $f:X \to X$ on a  standard probability
space $(X,\mu)$ such that $E_f$ is mcp, we let $T_f$ denote the associated treeing of $E_f$, i.e., where each $x$ is adjacent to $f(x)$.
For each $x \in X$, the sequence $(f^n(x))_{n\in\N}$ converges to the unique \dfn{forward $f$-end} of the $T_f$-component of $x$. 
We call each other end of $T_f\rest{[x]_{E_f}}$ a \dfn{back end} of $f$.

In \cref{summable_back-rays,back_ends_converge_to_0,two-ended_function} we characterize exactly when $f$ has nonvanishing back ends in terms of the essential number of ends of $T_f$ (defined in \cref{subsubsec:finite_ends}).

\begin{theorem}[Nonvanishing and essential number of ends]\label{intro:cocycle-back-ends}
Let $f : X \to X$ be an acyclic countable-to-one Borel function on a standard probability space $(X,\mu)$ such that $E_f$ is mcp and $\mu$-nowhere smooth, with Radon--Nikodym
cocycle $\rho$. 
\begin{enumerate}[(a)]
\item If $f$ is essentially two-ended then, after discarding an $E_f$-invariant null set, each forward end of $f$ is $\rho$-nonvanishing, and each $T_f$-component has exactly one $\rho$-nonvanishing back end. 

\item\label{item:intro_back_ends_vanish} If $f$ is $\mu$-nowhere essentially two-ended then, after discarding an $E_f$-invariant null set, each forward end of $f$ is $\rho$-nonvanishing, and all back ends of $f$ are $\rho$-vanishing and have $\rho$-finite geodesics. 
\end{enumerate}
\end{theorem}

Part \labelcref{item:intro_back_ends_vanish} of \cref{intro:cocycle-back-ends} is the most difficult ingredient in our analysis of vanishing and nonvanishing ends throughout the paper, and it is particularly
important in the proof of \cref{intro:trichotomy}.

In \cref{essential_ends=barytropic} we also characterize barytropy in terms of the essential number of ends of $T_f$.
See \cref{def:mono-poly} for the terminology.

\begin{theorem}[Barytropy and essential number of ends]\label{intro:barytropy_and_essential_ends}
Let $f : X \to X$ be an acyclic countable-to-one Borel function on a standard
probability space $(X,\mu)$ such that $E_f$ is mcp and $\mu$-nowhere smooth, with Radon--Nikodym cocycle $\rho$. 
Then
\begin{enumerate}[(a)]
\item\label{item:intro_monobarytropic} $f$ is $\rho$-monobarytropic if and only if, after discarding an $E_f$-invariant Borel null set, $f$ is essentially one-ended.

\item $f$ is $\rho$-dibarytropic if and only if, after discarding an $E_f$-invariant Borel null set, $f$ is essentially two-ended.

\item $f$ is $\rho$-polybarytropic if and only if $f$ is $\mu$-nowhere essentially finitely-ended.
\end{enumerate}
\end{theorem}

\subsection{Characterization of essentially one-ended countable-to-one Borel functions}

Part \labelcref{item:intro_monobarytropic} of \cref{intro:barytropy_and_essential_ends} is one piece of the following broader characterization of essential one-endedness, proved in \cref{char-rho-finite-back-orbits}.

\begin{theorem}[Characterizations of essential one-endedness]\label{intro:char-one-ended}
Let $f : X \to X$ be an acyclic countable-to-one Borel function on a standard
probability space $(X,\mu)$, such that $E_f$ is measure-class-preserving with associated
Radon--Nikodym cocycle $\rho$. 
The following are equivalent:
\begin{enumerate}[(1)]
    \item\label{item:intro_back-orbits-rho-finite} Almost all back $f$-orbits are $\rho$-finite. 
    \item\label{item:intro_forward_summable} There exists a strictly positive Borel function $Q : X \rightarrow \Rpos$ with $\sum _{n\geq 0} Q(f^n(x))<\infty$ for almost every $x\in X$.  

    \item\label{item:intro_one-ended_next_ret} After discarding a null set there is a forward $f$-recurrent Borel $E_f$-complete section $Y\subseteq X$ for which the next return map $f_Y : Y\rightarrow Y$ is one-ended.
    
    \item\label{item:intro_ess_one_ended} After discarding a null set there is a forward $f$-invariant Borel $E_f$-complete section $Z\subseteq X$ on which the restriction $f:Z\rightarrow Z$ is one-ended.    
    \item\label{item:intro_ess_one_ended_loc_fin} After discarding a null set there is a set $Z$ as in \labelcref{item:intro_ess_one_ended} on which the restriction $f:Z\rightarrow Z$ is moreover finite-to-one.
\end{enumerate}
\end{theorem}

Perhaps the most surprising implication is \labelcref{item:intro_one-ended_next_ret}$\Rightarrow$\labelcref{item:intro_ess_one_ended_loc_fin}, whose proof involves analytic boundedness and the Borel--Cantelli lemma (see \cref{ess_one_end_loc_fin}).
As discussed in \cref{remark:essential_one-ended_characterization}, neither implication \labelcref{item:intro_one-ended_next_ret}$\Rightarrow$\labelcref{item:intro_ess_one_ended} nor \labelcref{item:intro_ess_one_ended}$\Rightarrow$\labelcref{item:intro_ess_one_ended_loc_fin} holds in the purely Borel (or even Baire category) setting.

\cref{intro:char-one-ended} easily gives a characterization of essential one-endedness for acyclic locally countable mcp graphs, as stated in \cref{ess_one_ended_trees}.

\subsection{Radon--Nikodym core}

To capture the interaction between the geometry of a treeing and the Radon--Nikodym cocycle, we introduce the \dfn{Radon--Nikodym core} of a treeing $T$, denoted $\Core_\mu(T)$, which captures the topographically essential part of $T$.
After discarding a null set, $\Core_\mu(T)$ coincides with the $T$-convex hull (in the $T$-end completion of $X$) of the set of all $\rho$-barytropic ends.

We call the set $\rnc_T \defeq \Core_\mu(T) \cap X$ the \dfn{Radon--Nikodym vertex core}.
When $T = T_f$ for some acyclic Borel function $f : X \to X$, we also define the \dfn{directed Radon--Nikodym vertex core} to be the set $\rnc_f$ of all points whose back $f$-orbit is $\rho$-infinite, and we show in \cref{core=X_infty} that, after discarding a null set, these two sets coincide, i.e., $\rnc_{T_f} = \rnc_f$.

Throughout the paper, the Radon--Nikodym vertex core and its basic properties provide the main organizational framework for the proofs of our results.
The following summarizes the main properties of the Radon--Nikodym vertex core, proved in \cref{core_minimality,Core-bi-inf-or-perfect,Core_idempotent,null_core=>ess_one-ended}.

\begin{theorem}\label{thm:intro_core}
Let $T$ be a locally countable acyclic mcp Borel graph on a standard probability space $(X,\mu)$ with associated Radon--Nikodym cocycle $\rho$.
Let $T_\infty \defeq T \rest{\rnc_T}$.
Then:
\begin{enumerate}[(a), leftmargin=*]
\item (Minimality) $\rnc_T$ is essentially contained in each $T$-convex Borel complete section.

\item (Geometry) Almost every $T_\infty$-component is either a bi-infinite line or a leafless tree with no isolated ends.

\item (Idempotence) After discarding a null set, $X^{\mu_\infty}_{T_\infty} = \rnc_T$, where $\mu_\infty \defeq \mu \rest{\rnc_T}$.

\item (Nullness of the core) $\rnc_T$ is null if and only if $T$ is essentially at most one-ended.
\end{enumerate}
\end{theorem}

\subsection{Overview of related work applying \cref{intro:trichotomy}}\label{sec:intro:related_work}

We now briefly survey several subsequent works \cite{Chen-Terlov-Tserunyan:nonamenable_subforests,kids,Poulin:elasticity,Terlov-Timar:weighted_amenability,Bell-et-al:heavy_repulsion} which have used the framework and results developed in this article.

In \cite{Chen-Terlov-Tserunyan:nonamenable_subforests}, Chen, Terlov, and Tserunyan extend a theorem of Gaboriau \cite[IV.24]{Gaboriau:cost} (which builds on a theorem of Ghys \cite{Ghys:Stallings}) from the pmp to the general mcp setting.
They show that any locally finite mcp Borel graph $G$ whose a.e.\ component has more than two nonvanishing ends admits a relatively ergodic Borel subforest with the same property.
Applying the nonamenability clause of our trichotomy (\cref{intro:trichotomy}), they conclude that this subforest is nowhere amenable and in particular $G$ is nowhere amenable.
As mentioned above, this work was developed in parallel with the present paper, including the definition of nonvanishing ends.

In \cite{kids}, Bell, Chu, and Rodgers answer two of our original questions by constructing two natural acyclic Borel functions with mcp orbit equivalence relations: in one, the Radon--Nikodym cocycle along forward geodesics oscillates, and in the other it converges nonsummably to $0$.
The functions constructed in \cite{kids}, discussed in \cref{example:least_deletion}, are one-ended, leaving open the existence of a nowhere essentially finitely-ended function with oscillatory behavior, which we recently constructed in \cref{example:any-ended_oscillation}.

In \cite{Poulin:elasticity}, Poulin gives a measure-theoretic strengthening of a result of Gaboriau and Jackson–Kechris–Louveau \cite[Theorem 3.17]{JKL}, showing that every treeable ergodic type III mcp countable Borel equivalence relation is generated by a free action of $\F_n$, for any $2 \le n \le \infty$, whose Schreier graph has only nonvanishing ends. 
In the original construction from \cite[Theorem 3.17]{JKL}, the Schreier graph has many vanishing ends, so by the minimality of the Radon--Nikodym vertex core (\cref{core_minimality}), Poulin’s action is more essential: every convex Borel complete section is conull.

The work in \cite{Terlov-Timar:weighted_amenability} is in the setting of percolation theory, where Terlov and Timár revisit weighted amenability for quasi-transitive graphs and prove its equivalence with several properties, including hyperfiniteness. 
Their proofs rely on the amenable case of our trichotomy (\cref{trichotomy}), transported to the percolation setting via Gaboriau’s cluster graphing construction \cite[Sections 2.2--2.3]{Gaboriau:invariant_percolation}.

Finally, in \cite{Bell-et-al:heavy_repulsion}, the authors show that, in Bernoulli percolation on quasi-transitive graphs, distinct clusters have a light set of neighboring vertices (a set of finite total weight), extending a theorem of Tim\'ar from the unimodular setting \cite{Timar:neighb_clusters} and advancing a longstanding question of H\"aggstr\"om, Peres, and Schonmann \cite{Haggstrom-Peres-Schonmann}.
Their argument relies on the construction of subgraphs with many nonvanishing ends and on a relative (subrelation) version of the nonamenability clause in our trichotomy (\cref{intro:trichotomy}) for locally finite mcp graphs, formulated explicitly in \cite[Proposition 3.24]{Chen-Terlov-Tserunyan:nonamenable_subforests}.

\subsection{Upcoming results}\label{sec:intro:upcoming_results}

In addition to the companion paper \cite{TTD:subrel-antitree} discussed earlier, the machinery of the current paper is applied in several upcoming works, which we now briefly describe.

In \cite{TTD:paddleball}, we revisit and generalize results of Carri\`ere and Ghys from \cite{Carriere-Ghys} on nonamenability and free splittings of mcp equivalence relations.
We then combine this generalization with a combinatorial ``paddle-ball lemma'' to give a new proof of the critical \cref{having_bigger_points_behind} from the present article, and hence of \cref{back_ends_converge_to_0}, which does not rely on \cref{summable_back-rays}.

In ongoing joint work with Berlow and Chen \cite{BCTTD}, we generalize many of the results of this paper to the setting of discrete mcp groupoids acting on tree bundles. 

Lastly, in an upcoming work with Chen and Terlov \cite{CTTTD}, we develop a general theory of ends for locally countable Borel graphs, which we then use to extend the results of the present paper and \cite{Chen-Terlov-Tserunyan:nonamenable_subforests}, as well as the end selection theorem \cite{Adams:indec_treed_eq_rel_end-selection}, \cite[3.19--3.21]{JKL} and the Epstein--Hjorth theorem \cite{Epstein-Hjorth}, to all locally countable mcp Borel graphs.

\subsubsection*{Acknowledgments}

The authors thank Ruiyuan (Ronnie) Chen and Grigory (Greg) Terlov for their feedback and help in shaping this paper, in particular for their parallel development of the notion of nonvanishing ends and for coining this term.
We also thank Katalin Berlow for reading through a draft of this paper and suggesting corrections, and Nachi Avraham-Re'em for pointing out \cite[Proposition 1.3.1]{Aaronson:book} and \cite[Theorem 23]{Kaimanovich:Hopf_decomp}, which relate to some of the statements used or proved here.
We thank the organizers and participants of the ``\textit{Descriptive Set Theory and Dynamics}'' workshop at the Banach Center in Warsaw, Poland, in August 2023, where this work was presented in a three-lecture tutorial, for encouraging us to sharpen the essential Borel complexity calculation of sets of the set of ends with cocycle-infinite geodesics.
Finally, the authors thank Benjamin Miller for his inspirational work which influenced the present development.

ATs was supported by NSERC Discovery Grant RGPIN-2020-07120.
RTD was supported by NSF grant DMS-2246684.

\medskip

\subsubsection*{AI statement.} The research for this article was completed without any use of AI. 
The first draft of what would eventually become this article was written between May 2018 and August 2022; since then we split that draft into three articles, with this being the first (the other two being \cite{TTD:subrel-antitree,TTD:paddleball}).

We used ChatGPT 5.1 for stylistic feedback when revising the introduction, and to check for typos and terminological inconsistencies in the complete manuscript.
We used Gemini Pro 3.1 for help with figures.

\section{Preliminaries}

Throughout, we use standard notation $\N \defeq \set{0, 1, \dots}$ and $\Rpos \defeq \set{r \in \R : r > 0}$. 
A \dfn{standard Borel space} is a set $X$ equipped with a $\sigma$-algebra that can be obtained as the Borel $\sigma$-algebra of a Polish (i.e., separable and completely metrizable) topology on $X$.  
The $\sigma$-algebra on $X$ will be suppressed from our notation.
We refer to \cite{bible} for terminology and background on standard Borel spaces.   
We will make frequent use of the Luzin--Novikov uniformization theorem \cite[18.10]{bible} to ensure that various sets we deal with are Borel.
Given standard Borel spaces $X$ and $Y$, we always equip $X\times Y$ with the product $\sigma$-algebra, making it a standard Borel space as well. 

By a \dfn{standard $\sigma$-finite measure space} we mean a pair $(\Omega ,\nu )$, where $\Omega$ is a standard Borel space and $\nu$ is a $\sigma$-finite Borel measure on $\Omega$.
A \dfn{standard probability space} is a pair $(X,\mu )$, where $X$ is a standard Borel space and $\mu$ is a Borel probability measure on $X$.

\begin{convention}\label{convention:standard_probability}
    Throughout the article, unless indicated otherwise, $X$ denotes a standard Borel space, and $(X,\mu)$ denotes a standard probability space.
\end{convention}

\begin{comment}
Given a locally compact second countable group $G$ and a Borel action of $G$ on a standard Borel space $X$, we write $E_G^X$ for the associated orbit equivalence relation, i.e., where two points are equivalent if and only if they are in the same $G$-orbit. 
\end{comment}

\subsection{Countable Borel equivalence relations}
Let $E$ be a countable Borel equivalence relation on a standard Borel space $X$.
That is, $E$ is an equivalence relation on $X$ that is Borel as a subset of $X^2$, and each $E$-class is countable.
Throughout the article, by a \dfn{subrelation} of $E$ we mean an equivalence subrelation of $E$, i.e., a subset of $E$ that is also an equivalence relation on $X$.

We use standard terminology regarding countable Borel equivalence relations, such as \dfn{smooth}, \dfn{hyperfinite}, and \dfn{treeable}; we refer the reader to \cite{JKL} for background on these notions.

We write $[x]_E$ for the $E$-class of a point $x\in X$.
A subset $A$ of $X$ is called an \dfn{$E$-complete section} if it meets every $E$-class in a nonempty set; equivalently, the saturation $[A]_E \defeq \bigcup _{y\in A}[y]_E$, of $A$,  is all of $X$. 
A \dfn{transversal} for $E$ is a subset of $X$ that meets every $E$-class in exactly one point.

Let $\rho : E \to \Rpos$ be a Borel \dfn{cocycle} into the positive reals, i.e., $\rho$ is a Borel map satisfying the \dfn{cocycle identity}: $\rho(z,y)\rho(y,x) = \rho(z,x)$ for all pairwise $E$-related points $x,y,z \in X$.
Given a point of reference $x \in X$, for  each $y \in [x]_E$ and $D\subseteq [x]_E$ we define $\rho^x(y)\defeq \rho(y,x)$ and $\rho ^x(D)\defeq \sum _{z\in D}\rho ^x(z)$. 
This notation is used to emphasize viewing $\rho ^x$ as a ``relative assignment of weights from the perspective of $x$,'' induced by $\rho$ on the $E$-class $[x]_E$; by the cocycle identity, switching to the perspective of some other $y\in [x]_E$ scales the corresponding weight assignment by a multiplicative constant, i.e., $\rho ^y = \rho ^y(x)\rho ^x$. 

It follows from the cocycle identity that $\rho ^x(x) =1$ and $\rho ^x(y)=\rho ^y(x)^{-1}$ for all $(y,x)\in E$. 
Thus, the inequality $\rho ^x(y)>1$ expresses that the $\rho$-weight of $y$ is strictly greater than that of $x$.
To emphasize the homogeneous, perspective-independent nature of such statements, we sometimes %drop the superscript; 
write $\rho^\bullet$; for example, for $(y,x) \in E$, we may write $\rho^\bullet (x)<  \rho^\bullet (y)$ instead of $\rho ^x(y)>1$.
A set $D$ of pairwise $E$-related points is called \dfn{$\rho$-finite} or $\dfn{$\rho$-infinite}$ according to whether $\rho^\bullet (D)$ is finite or infinite.

A Borel cocycle $\rho :E\to \Rpos$ is called a \dfn{coboundary} if there is a Borel function $m : X \to \Rpos$ such that $\rho(y,x) = m(y) / m(x)$ for all $(y,x)\in E$.
Such a function $m$ may be viewed as an ``absolute assignment of weights'' that is consistent with the relative assignments induced by $\rho$. 

\subsection{Measure-class-preserving equivalence relations}\label{subsec:mcp}

Let $E$ be a countable Borel equivalence relation on a standard Borel space $\Omega$, and let $\nu$ be a $\sigma$-finite Borel measure on $\Omega$.

Given a property $P$ of Borel equivalence relations, we say that $E$ is \dfn{$\nu$-nowhere} $P$ if there is no positive measure $E$-invariant Borel subset $A$ of $\Omega$ such that the restriction $E\rest{A}$ has property $P$.
For example, $E$ is \dfn{$\nu$-nowhere smooth} if there is no positive measure $E$-invariant Borel subset $A$ of $\Omega$ such that $E\rest{A}$ is smooth.

The measure $\nu$ is called {\bf $E$-invariant} if the associated left and right fiber measures $\nu _E^{\ell}$ and $\nu _E^r$ on $E\subseteq \Omega^2$ coincide, where
\begin{equation}\label{eq:fiber_measures}
\int_E h \, d \nu_E^\ell \defeq \int_\Omega \sum_{y \in [x]_E} h(x,y) \; d\nu(x) 
\ \text{ and }\;
\int_E h \, d\nu_E^r \defeq \int_\Omega \sum_{x \in [y]_E} h(x,y) \;d\nu(y)
\end{equation}
for each Borel function $h : E \to [0, \infty]$.
In the case where $\nu$ is an $E$-invariant probability measure, we say that $E$ is
\dfn{probability-measure-preserving (pmp)} on $(\Omega ,\nu )$.

Assume now that $E$ is \dfn{measure-class-preserving} (\dfn{mcp}) on $(\Omega,\nu)$, i.e., the left and right fiber measures $\nu_E^\ell$ and $\nu_E^r$, defined by \cref{eq:fiber_measures}, are in the same measure class. 
Equivalently (see, for example, \cite[Section 8]{Kechris-Miller}), the $E$-saturation of every $\nu$-null set is $\nu$-null.
Then the Radon--Nikodym derivative $\rho_\nu \defeq d \nu_E^\ell /d \nu_E^r$ is a Borel cocycle, called the \dfn{Radon--Nikodym cocycle of $E$ with respect to $\nu$}.
Strictly speaking, $d \nu_E^\ell/d \nu_E^r$ only determines an equivalence class of functions, but we can--and shall--always choose a representative $\rho_\nu$ which is indeed a Borel cocycle.

By its definition, $\rho_\nu$ is uniquely characterized (up to null sets) by the following \dfn{mass transport principle}: for each Borel function $h : E \to [0, \infty]$,
\begin{equation}\label{eq:mass_transport}
\int_\Omega \sum_{y \in [x]_E} h(x,y) \;d\nu(x) = \int_\Omega \sum_{y \in [x]_E} h(y,x) \rho_\nu^x(y) \;d\nu(x).    
\end{equation}
In the context of applying this principle, we refer to the function $h$ as a \dfn{mass transport function}, and we think of the value $h(x,y)$ as the amount of mass that $x$ sends to $y$ from the perspective of $x$. 
The quantity $\sum_{y \in [x]_E} h(x,y)$ is then the total mass sent out by $x$ (from the perspective of $x$), while $\sum_{y \in [x]_E} h(y,x) \rho_\nu^x(y)$ is the total mass received by $x$ (from the perspective of $x$).
When $\nu$ is a probability measure, the mass transport principle may then be stated succinctly as: for each mass transport function, the expected mass sent out is equal to the expected mass received. 

\begin{remark}
Although we interpret $h(x,y)$ as the mass sent from $x$ to $y$ from the perspective of $x$, this choice is arbitrary. 
Using the opposite convention--where $h(x,y)$ is interpreted as the mass sent from $y$ to $x$ from the perspective of $y$--leads to the same mass transport principle. 
To verify this for a given $h$, apply \cref{eq:mass_transport} to the function $g(x,y) \defeq h(x,y) \rho_\nu^x(y)$ in place of $h$ to see that the expected mass sent out is still equal to the expected mass received under this alternative interpretation.
\end{remark}

The following proposition is a straightforward consequence of the defining property of Radon--Nikodym derivatives.

\begin{prop}\label{equivalent_measures}
Let $E$ be an mcp countable Borel equivalence relation on a standard $\sigma$-finite measure space $(\Omega,\nu)$.

\begin{enumerate}[(a)]
    \item If $\zeta$ is a $\sigma$-finite Borel measure on $\Omega$ in the same measure class as $\nu$ then, after discarding a null set,
    \[
    \rho_\zeta(x,y) = \tfrac{d\zeta}{d\nu}(x) \rho_\nu(x,y) \tfrac{d\zeta}{d\nu}(y)^{-1}
    \]
    for all $(x,y) \in E$.

    \item \label{coboundary=invariant} The measure $\nu$ is equivalent to an $E$-invariant $\sigma$-finite Borel measure if and only if $\rho_\nu$ is a coboundary on some $\nu$-conull subset of $\Omega$.
\end{enumerate}
\end{prop}

Since every $\sigma$-finite Borel measure is equivalent to a probability measure, we restrict our attention to Borel probability measures for the majority of the article.  
Furthermore, as typically only one measure will be relevant, we adopt the following convention.

\begin{convention}\label{convention:rho}
When there is no risk of confusion, we denote by $\rho$ (rather than $\rho_\mu$) the Radon--Nikodym cocycle associated to an mcp countable Borel equivalence relation $E$ on a standard probability space $(X,\mu )$.
\end{convention}

An immediate application of mass transport is a characterization of mcp equivalence relations that are smooth on a conull set, which is due to Miller \cite[Proposition 2.1]{Miller:meas_with_cocycle_II}, and independently Kaimanovich \cite[Theorem 23]{Kaimanovich:Hopf_decomp} (phrased in a different language).
For the convenience of the reader, we give a proof of this here as an illustration of mass transport technique:

\begin{lemma}\label{smooth=rho-finite}
Let $E$ be an mcp countable Borel equivalence relation on a standard probability space $(X,\mu)$.

\begin{enumerate}[(a)]
\item \label{part:coc-finite_preimages} Let $f : X \to X$ be a Borel function whose graph is contained in $E$.
Then 
\[
\int_X \rho^x (f^{-1}(x)) \,d\mu(x) = 1.
\]
In particular $f^{-1}(x)$ is $\rho$-finite for a.e.\ $x \in X$.

\item \label{part:smooth=rho-finite} \emph{(\cite{Miller:meas_with_cocycle_II,Kaimanovich:Hopf_decomp})} $E$ is smooth on a conull subset of $X$ if and only if a.e.\ $E$-class is $\rho$-finite.
\end{enumerate}
\end{lemma}
\begin{proof}
Part \labelcref{part:coc-finite_preimages} follows from the mass transport principle applied to the indicator function of the graph of $f$.
For part \labelcref{part:smooth=rho-finite}, if $E$ is smooth on a conull set then, after discarding a null set, there is a Borel selector $f : X \to X$ for $E$, so part \labelcref{part:coc-finite_preimages} implies that a.e.\ $E$-class is $\rho$-finite.
Conversely, if a.e.\ $E$-class is $\rho$-finite, then after discarding a null set, the set of all $x \in X$ for which $\rho^x(y) \le 1$ for all $y \in [x]_E$ meets every $E$-class in a finite set, witnessing the smoothness of $E$.
\end{proof}

\begin{lemma}\label{backward_mass_transport}
Let $f:X\to X$ be a countable-to-one Borel function such that $E_f$ is mcp on $(X,\mu )$.
If $\rho ^x(f^{-1}(x))\geq 1$ for a.e. $x\in X$ then $\rho ^x(f^{-1}(x))= 1$ for a.e.\ $x\in X$.
\end{lemma}

\begin{proof}
This is an immediate consequence of \cref{smooth=rho-finite}\labelcref{part:coc-finite_preimages}.
\end{proof}

\subsection{Amenability}\label{subsec:amenability}

Let $E$ be an mcp countable Borel equivalence relation on a standard probability space $(X,\mu)$. 
The equivalence relation $E$ is called \dfn{$\mu$-amenable} if there exists an $E$-invariant Borel conull set $X_0 \subseteq X$ and a sequence $(m^k)_{k \in \N}$ of Borel assignments $x \mapsto m^k_x$, of a probability vector $m^k_x\in \ell ^1([x]_E)$ to each $x \in X_0$, satisfying 
\[
\lim_{k\to\infty} \| m^k_x - m^k_y\|_1 =0
\]
for all $(x,y)\in E \rest{X_0}$.

By a celebrated theorem of Connes, Feldman, and Weiss \cite{Connes-Feldman-Weiss}, $\mu$-amenability is equivalent to $\mu$-hyperfiniteness, where $E$ is called \dfn{$\mu$-hyperfinite} if there exists an $E$-invariant Borel conull set $X_0\subseteq X$ such that $E\rest{X_0}$ is a countable increasing union of finite Borel equivalence relations.
Throughout this article, ``amenable'' always means ``$\mu$-amenable,'' where $\mu$ is the underlying probability measure.

\subsection{Graphs}

Let $G$ be a graph on a vertex set $V$, i.e., $G$ is an irreflexive symmetric subset of $V^2$. 
We call $G$ \dfn{locally countable} if every vertex of $G$ has countably-many neighbors in $G$.
We let $E_G$ denote the $G$-connectedness equivalence relation on $V$, i.e.,
\[
E_G \defeq \{ (x,y)\in V^2 : \text{$x$ and $y$ belong to the same $G$-component} \} .
\]
For a subset $Y$ of $V$, we denote the induced subgraph on $Y$ by $G \rest{Y}$, i.e., $G \rest{Y}\defeq G \cap Y^2$.
A subset $Y$ of $V$ is called \dfn{$G$-connected} if $G\rest{Y}$ is a connected graph on $Y$. 
An equivalence relation $F$ on $V$ is called $G$-connected if each $F$-class is $G$-connected.

We denote by $d_G$ the graph distance associated to $G$ (where $d_G(x,y)\defeq\infty$ if the vertices $x$ and $y$ lie in different $G$-components).
Given vertices $x,y\in V$ at graph distance $d_G(x,y)=\ell <\infty$, a \dfn{geodesic path} through $G$ from $x$ to $y$ is a sequence $(x_n)_{n = 0}^\ell$ of vertices, with $x_0=x$ and $x_{\ell}=y$ such that $d_G(x_n,x_m ) = |n-m|$ for all $0 \le n,m \le \ell$. 
Similarly, a \dfn{geodesic ray}  through $G$ is a sequence $(x_n)_{n \in \N}$ of vertices with $d_G(x_n,x_m )=|n-m|$ for all $n,m\in \N$. 
We define a \dfn{geodesic line} $(x_n)_{n \in \Z}$ through $G$ analogously.

\subsubsection{Borel graphs}  
A graph $G$ on a standard Borel space $X$ is called Borel if it is Borel as a subset of $X^2$.
Given a locally countable Borel graph $G$ on $X$, and a Borel probability measure $\mu$ on $X$, we say that $G$ is \dfn{measure-class-preserving (mcp)} on $(X,\mu )$ if $E_G$ is mcp on $(X,\mu )$.
In this context we also write $\rho$ for the Radon--Nikodym cocycle associated to $E_G$.

Given a countable Borel equivalence relation $E$ on $X$, a Borel graph $G$ on $X$ is called a \dfn{graphing} of $E$ if $E_G = E$, i.e., the connected components of $G$ are precisely the $E$-classes. 
If $T$ is a graphing of $E$ that is acyclic, then we call $T$ a \dfn{treeing} of $E$, and we say that $E$ is \dfn{treed} by $T$.

\subsubsection{Ends and geodesics in acyclic graphs}\label{subsubsec:geom_acyclic_graphs}

Let $T$ be an acyclic graph on a vertex set $V$. 

Let $\Rays(T) \subseteq V^\N$ denote the set of all geodesic rays through $T$.
Let $\sim_T$ be the \dfn{tail equivalence relation} on $\Rays(T)$ defined by setting $(x_n) \sim_T (y_n)$ exactly when $(x_{n+N})_{n \in \N} = (y_{n+M})_{n \in \N}$ for some $N,M \in \N$.
An equivalence class of the relation $\sim_T$ is called an \dfn{end} of $T$.
Let 
\[
\del_T V \defeq \Rays(T) / \sim_T
\]
denote the set of all ends of $T$.
For a subset $Y$ of $V$, we denote by $\del_T Y$ the set of all ends of $T$ that are represented by some geodesic ray contained in $Y$.
We call the set 
\[
\-V^T \defeq V \cup \del_T V
\]
the \dfn{end completion} of $V$ (with respect to $T$), which we simply denote by $\-V$ when $T$ is clear from the context.

For a directed edge $e$ of $T$, let $\orig (e)$ denote its origin, $\term (e)$ its terminus, and $e^{-1}$ its inverse, so that $\orig (e^{-1})=\term (e)$, $\term (e^{-1})=\orig (e)$, and $(e^{-1})^{-1}=e$.
The \dfn{edge-boundary} of a subset $Y$ of $V$ is the set of all undirected edges of $T$ incident to $Y$ and $V \setminus Y$, and the \dfn{outgoing edge-boundary} of $Y$ is the set of all directed edges $e$ of $T$ with $\orig (e)\in Y$ and $\term (e)\not \in Y$.
A \dfn{half-space} of $T$ is a $T$-connected subset $U$ of $V$ with exactly one directed edge in its outgoing edge-boundary.
For a directed edge $e$ of $T$, let $\Vo(e)$ be the unique half-space of $T$ with $e$ in its outgoing edge-boundary, and let $\Vt(e)\defeq \Vo(e^{-1})$. 
Thus, every half-space of $T$ is of the form $\Vo(e)$ for a unique directed edge $e$ of $T$.

The \defterm{edge-removal topology} on $\-V$ is the topology generated by the sets $U \cup \del_T U$, with $U$ ranging over all half-spaces of $T$.
The \defterm{vertex-removal topology} on $\-V$ is the topology generated by the edge-removal topology together with all subsets of $V$.
In both topologies, the closure of a half-space $U$ is $U \cup \del_T U$, and the collection of all such half-space closures that contain a given end $\xi \in \del_T V$ is a neighborhood basis of $\xi$.
We may therefore unambiguously speak of convergence of a net in $\-V$ to an end of $T$ without needing to specify the topology.
See \cref{subsec:topologies} for further discussion of these topologies.

For $x,y \in V$ in the same $T$-component, the \dfn{$T$-geodesic between $x$ and $y$}, denoted $[x,y]_T$, is the set of vertices on the unique geodesic path through $T$ from $x$ to $y$. 
We let $(x,y)_T \defeq [x,y]_T \setminus \set{x,y}$, and define $[x,y)_T$ and $(x,y]_T$ analogously.

For a $T$-component $C$, a vertex $x \in C$, and an end $\xi \in \del_T C$, the \dfn{$T$-geodesic from $x$ to $\xi$}, denoted $[x, \xi)_T$, is the set of all vertices on the unique geodesic ray through $T$ that starts at $x$ and represents $\xi$.
Naturally, we define $(x, \xi)_T \defeq [x,\xi)_T \setminus \set{x}$. 
For distinct ends $\xi, \eta \in \del_T C$ of the same $T$-component $C$, the \dfn{$T$-geodesic between $\xi$ and $\eta$} is the set of points on some/every geodesic line $(x_n)_{n \in \Z}$ through $T$ for which $(x_{-n})_{n \in \N}$ and $(x_n)_{n \in \N}$ represent the ends $\xi$ and $\eta$, respectively.

For a subset $Y$ of $\-V$, we denote by $\Conv_T Y$ the \dfn{convex hull} of $Y$ in $T$, i.e., the set of points of $\-V$ that lie on some $T$-geodesic between two points in $Y$.
A subset $Y$ of $V$ is called \dfn{$T$-convex} if $\Conv_T Y = Y$.

\subsubsection{Acyclic graphs with finitely many ends}\label{subsubsec:finite_ends}

Let $T$ be a locally countable acyclic Borel graph on a standard Borel space $X$. 
We say that $T$ is \dfn{$\ka$-ended} for $\ka \in \N \cup \set{\aleph_0, 2^{\aleph_0}}$ if the cardinality of the set of ends of each $T$-connected component is $\ka$. 
We say that $T$ is essentially $\ka$-ended if there is a Borel $E_T$-complete section $Y$ that is $T$-convex such that $T \rest{Y}$ is $\ka$-ended.

\begin{comment}    
\begin{lemma}\label{two-ended=>lined}
If $T$ is essentially two-ended, then $T$ is essentially $2$-regular.
\end{lemma}

\begin{proof}
Replacing $X$ with a witness to essential two-endedness of $T$, we may assume that $T$ is two-ended to begin with. 
Now let $Y$ be the set of points that lie in a $2$-regular connected subgraph of $T$. 
Because $T$ is two-ended, each component of $T$ contains exactly one such subgraph, so $Y$ is a Borel set by the Luzin--Souslin theorem \cite[15.2]{bible}.
\end{proof}
\end{comment}

The following lemma is a special case of \cite[Theorem 2.1]{Miller:ends_of_graphs_I}. 

\begin{lemma}[Miller \cite{Miller:ends_of_graphs_I}]\label{0-ended=>smooth}
If a locally countable acyclic Borel graph $T$ is essentially 0-ended, then $E_T$ is smooth.
\end{lemma}
\begin{proof}
In the proof of \cite[Theorem 2.1]{Miller:ends_of_graphs_I}, the general case is reduced to the case of an acyclic graph and the proof for the latter is the short argument right after \cite[Lemma 2.4]{Miller:ends_of_graphs_I}.
\end{proof}

The following lemma is a restatement of \cite[Lemma 3.19(ii)]{JKL}.

\begin{lemma}\label{3-ended=>smooth}
If a locally countable acyclic Borel graph $T$ is essentially $n$-ended for some finite $n \ge 3$, then $E_T$ is smooth.
\end{lemma}

\begin{lemma}\label{uniquely_ended}
If a locally countable acyclic Borel graph $T$ is essentially $n$-ended and essentially $m$-ended for distinct $m,n \in \N$, then $E_T$ is smooth.
\end{lemma}
\begin{proof}
Assume $m < n$ and let $A$ and $B$ be $T$-convex Borel $E_T$-complete sections such that $T \rest{A}$ is $m$-ended and $T \rest{B}$ is $n$-ended. Let $A'$ be the set of vertices of $A$ that are the closest to the ends of $T \rest{B}$ that are not ends of $T \rest{A}$. 
This set is Borel and meets each $E_T$-class in a nonempty set of at most $n$ points, witnessing the smoothness of $E_T$.
\end{proof}

\subsection{Countable-to-one functions}

Let $X$ be a standard Borel space and let $f : X \to X$ be a countable-to-one Borel function that is \dfn{acyclic}, i.e., there is no $n\geq 1$ and $x\in X$ with $f^n(x)=x$.
We denote by $T_f$ the (undirected) graph on $X$ generated by $f$, i.e., with an undirected edge between $x$ and $f(x)$ for every $x\in X$.
The function $f$ being acyclic implies that the graph $T_f$ is acyclic. 
We write $E_f$ for the equivalence relation $E_{T_f}$.
For a $T_f$-component $C$, the \dfn{forward $f$-end} of $T_f \rest{C}$ is the end represented by the geodesic ray $(f^n(x))_{n\in \N}$, for some/any $x \in C$.
Each other end of $T_f \rest{C}$ is called a \dfn{back $f$-end}.
We say that $f$ is \dfn{$n$-ended} (resp.\ \dfn{essentially $n$-ended}) if $T_f$ is $n$-ended (resp.\ essentially $n$-ended).

A subset $Y$ of $X$ is said to be \dfn{forward $f$-invariant} if $f(Y)\subseteq Y$, and $Y$ is called \dfn{forward $f$-recurrent} if for each $y\in Y$ there is some $n\geq 1$ with $f^n(y) \in Y$. 
For $x \in \dom(f)$, $D\subseteq \N$, and $A\subseteq X$ we define 
\[
f^{-D}(x)\defeq \bigcup_{n\in D} f^{-n}(x) 
\ \text{ and }\; f^{-D}(A) \defeq \bigcup_{y\in A}f^{-D}(y).
\]
The set $f^{-\N}(x)$ is called the \dfn{$f$-back-orbit} of $x$. 

Given a Borel subset $Y$ of $X$, the \dfn{retraction to $Y$ along $f$} is the function $\retract{Y,f} : f^{-\N}(Y) \to Y$ given by $\retract{Y,f}(x)\defeq f^n(x)$, where $n\geq 0$ is least such that $f^n(x) \in Y$.  
When $Y$ is forward $f$-recurrent we also define the \dfn{next return map} induced by $f$ on $Y$ to be the map $f_Y : Y \to Y$ given by $f_Y(y)\defeq f^n(y)$ where $n\geq 1$ is least such that $f^n(y)\in Y$. 

\subsection{End selection}

Let $T$ be a treeing of a countable Borel equivalence relation $E$ on $X$. 
Given a Borel subset $X_0$ of $X$, a \dfn{$T$-end selection} on $X_0$ is a map $X_0 \to \Pow _{\mathrm{fin}}(\del_T X)$, $x \mapsto \EC(x)$, assigning to each $x \in X_0$ a finite nonempty set $\EC(x)$ of ends of $T \rest{[x]_{E}}$. %When each $\EC(x)$ is a singleton we write $x \mapsto \xi_x : X \to \del_T X$.
Such a $T$-end selection $x \mapsto \EC(x)$ is \dfn{Borel} if it lifts to a Borel map $\dot\EC$ from $X_0$ to the standard Borel space of finite subsets of $\Rays(T)$ (where $\Rays(T)$ is naturally identified with a Borel subset of $X^{\N}$). 
Given a Borel subequivalence relation $F$ of $E$, we say that a Borel $T$-end selection $x\mapsto \EC (x)$ on $X_0$ is $\dfn{$F$-invariant}$ if $\EC (x)=\EC (y)$ for all $(x,y)\in F\rest{X_0}$. 
In this situation it is worth noting that, although the map $\EC$ is $F$-invariant and admits a Borel lift $\dot\EC$, in most cases $\EC$ does not admit an $F$-invariant Borel lift.

Given an mcp countable Borel equivalence relation $F$ on $(X,\mu)$, we define 
\[
\Xns_F \defeq \set{x \in X : \rho^x([x]_F) = \infty}.
\]
By \cref{smooth=rho-finite}, the set $\Xns_F$ is the essentially unique largest $F$-invariant Borel subset of $X$ on which $F$ is $\mu$-nowhere smooth.
We can now state the following theorem due to Adams \cite{Adams:indec_treed_eq_rel_end-selection} (see also \cite{JKL} for another version), which we state in a slightly sharper form, and which is the starting point of our analysis of amenable subrelations of a treed mcp equivalence relation. 

\begin{theorem}[End selection {\cite[Theorem 3.1]{Adams:indec_treed_eq_rel_end-selection}}]\label{end_selection}
Let $T$ be a treeing of an mcp countable Borel equivalence relation $E$ on $(X,\mu )$, and assume that $E$ is $\mu$-nowhere smooth. 
A Borel subequivalence relation $F$ of $E$ is amenable if and only if, after discarding a null set, there is an $F$-invariant Borel $T$-end selection, $X\rightarrow \Pow _{\mathrm{fin}}(\del_T X)$, $x \mapsto \EC (x)$. 

Moreover, if $F$ is an amenable Borel subequivalence relation of $E$ then
\begin{enumerate}[(a)]
\item every $F$-invariant Borel $T$-end selection on $\Xns_F$ satisfies $1\leq |\EC (x)|\leq 2$ for a.e.\ $x\in \Xns_F$;
\item after discarding a null set, there exists an essentially unique maximum $F$-invariant Borel $T$-end selection $x\mapsto \EC _F (x)$ on $\Xns_F$. 
That is, if $x\mapsto \EC (x)$ is any other $F$-invariant Borel $T$-end selection on a relatively conull subset of $\Xns_F$, then $\EC (x)\subseteq \EC _F(x)$ for a.e.\ $x\in \Xns_F$.
\end{enumerate} 
\end{theorem}

\begin{proof}
By \cref{0-ended=>smooth}, since $E$ is $\mu$-nowhere smooth, after discarding a null set we may assume that each $T$-component has at least one end. 
The theorem then follows from \cite[3.19, 3.21, and 3.26]{JKL}.
\end{proof}

%\section{Coherent sets of edges and the Helly property}

\subsection{Coherent sets of directed edges and pruning}

Let $T$ be an acyclic graph on a vertex set $V$. 
Let $\leq _T$ be the partial order on the set of directed edges of $T$ defined by $e_0 \leq_T e_1$ if and only if $\Vo(e_0) \subseteq \Vo(e_1)$. 

We say that a set $H$ of directed edges of $T$ is \dfn{coherent} if for any two edges $e_0,e_1\in H$ in the same $T$-component, the half-spaces $\Vt(e_0)$ and $\Vt(e_1)$ have nonempty intersection. 
Given a coherent set $H$ of directed edges of $T$, observe that for any two directed edges $e_0 ,e_1 \in H$, the half-spaces $\Vo(e_0)$ and $\Vo(e_1)$ are either disjoint or one contains the other.
Thus, we obtain an equivalence relation $F_H$ on the set $\Vo(H) \defeq \bigcup_{e\in H}\Vo (e)$, by declaring two vertices to be $F_H$-equivalent if and only if they are both contained in $\Vo (e)$ for some $e\in H$.

\begin{prop}\label{coherent_orientation_prop}
Let $T$ be an acyclic graph on a vertex set $V$ and let $H$ be a coherent set of directed edges of $T$.
\begin{enumerate}[(a)]

\item\label{item:coherent_dichotomy} If $T$ is connected then either every $e\in H$ is below some $\leq_T$-maximal element of $H$ (\cref{fig:coherent_max}), or else $(H,\leq_T)$ is a directed set and the net $(\orig (e)) _{e\in H}$ converges in both the edge-removal and vertex-removal topologies to a single end in $\partial _T V$ (\cref{fig:coherent_end}). 

\item\label{item:maximal_edges_transversal} Assume that every $e\in H$ is below some $\leq_T$-maximal element of $H$, and let $H_0 \subseteq H$ be the set of $\leq_T$-maximal elements of $H$. 
Then $\orig (H_0) \defeq \{ \orig (e): e\in H_0 \}$ is a transversal for the associated equivalence relation $F_H$.
\end{enumerate}
\begin{figure}[htbp]
\centering

\begin{tikzpicture}[
scale=0.75, % Scales down the overall size of the diagram
very thick,
line cap=round,
line join=round,
line width=1.8pt, % Adjusted for smaller scale
midarrow/.style={
postaction={
decorate,
decoration={
markings,
% 0.5 is the exact middle. 
% (0.25 would be a quarter way, 0.9 near the end)
mark=at position 0.65 with {\arrow{>}}
}
}
},
strblue/.style={blue!90!cyan},
strred/.style={color={rgb,255:red,210; green,0; blue,0}},
dot/.style={circle, fill=black, inner sep=2pt, outer sep=0pt},
>={Straight Barb[length=7pt, width=9pt]},
marr/.style={
postaction={
decorate,
decoration={
markings,
mark=at position 0.5 with {\arrow{>}} % Arrowheads exactly in the middle
}
}
}
]

% -------------------------
% 1. Define Coordinates
% -------------------------

% Main Trunk Nodes (y=0)
\coordinate (T1) at (0, 0);
\coordinate (T2) at (2, 0);
\coordinate (T3) at (4, 0);
\coordinate (T4) at (6, 0);
\coordinate (T5) at (8, 0);
\coordinate (T6) at (8, 0);
\coordinate (T7) at (10.5, 0);
\coordinate (T8) at (12.5, 0);
\coordinate (T9) at (14.5, 0);

% T1 - Far Left Branches
\coordinate (L1) at (-1.5, 0);
\coordinate (UL) at (-1.2, 1);
\coordinate (DL) at (-1.2, -1);

% T2 - Down Right Branch
\coordinate (D2_1) at (2.8, -1);
\coordinate (D2_2a) at (2.4, -2);
\coordinate (D2_2b) at (3.5, -1.6);

% T3 - Up Left Branch
\coordinate (U3_1) at (3.2, 1);
\coordinate (U3_2a) at (2.4, 1.5);
\coordinate (U3_2b) at (3.8, 1.8);

% T4 - Down Left Branch (Red Path)
\coordinate (D4_1) at (5.2, -1);
\coordinate (D4_2L) at (4.4, -2);
\coordinate (D4_2R) at (5.8, -1.8);

% T5 - Up Right Branch
\coordinate (U5_1) at (8.6, 0.6);

% T6 - Complex Top Left Tree (Red / Blue Path)
\coordinate (C1) at (7.5, 1.2);
\coordinate (C2) at (7, 2.2);
\coordinate (C3) at (8.3, 1.5);
\coordinate (C4) at (7.2, 3);
\coordinate (C5) at (8, 2.8);
\coordinate (C6) at (6, 3);
\coordinate (C7) at (5.6, 3.8);
\coordinate (C8) at (6.6, 3.8);

\coordinate (D1) at (8, -1);
\coordinate (D2) at (9.2, -1);
\coordinate (D3) at (7.6, -2);
\coordinate (D4) at (8.3, -2);
\coordinate (D5) at (9.4, -2);

% T7 - Vertical Top Right Tree (Red / Blue Path)
\coordinate (R1) at (11, 1.5);
\coordinate (R2) at (11.5, 2.5);
\coordinate (R3) at (9.5, 3);
\coordinate (R4) at (8.8, 3.8);
\coordinate (R5) at (10.2, 3.8);

% T8 - End Branches
\coordinate (D8_1) at (13.5, -1.2);
\coordinate (U8_1) at (13.5, 0.8);

% -------------------------
% 2. Draw Edges
% -------------------------

% --- Main Trunk ---
\draw (T1) -- (T2);
\draw[midarrow, strred] (T2) -- (T3); % Red trunk flow
\draw (T3) -- (T4);
\draw (T4) -- (T5);
\draw (T5) -- (T6);
\draw (T6) -- (T7);
\draw[strred, midarrow] (T8) -- (T7); % Red trunk flow
\draw (T8) -- (T9);

% --- T1 Branches ---
% Mid
\draw (L1) -- (T1);
\draw[strblue, midarrow] (-2.2, 0.3) -- (L1);
\draw (-2.2, -0.3) -- (L1);
% Top (Blue In)
\draw (UL) -- (T1);
\draw[strblue, midarrow] (-1.8, 1.8) -- (UL);
\draw (-0.5, 1.8) -- (UL);
% Bottom (Blue In to Black)
\draw[strblue, midarrow] (DL) -- (T1);
\draw (-2, -0.8) -- (DL);
\draw (-1.8, -1.8) -- (DL);

% --- T2 Branches ---
\draw (D2_1) -- (T2);
\draw (D2_2a) -- (D2_1);
\draw (D2_2b) -- (D2_1);
\draw (2, -2.6) -- (D2_2a);
\draw (2.8, -2.6) -- (D2_2a);

% --- T3 Branches ---
\draw (U3_1) -- (T3);
\draw (U3_2a) -- (U3_1);
\draw (U3_2b) -- (U3_1);
\draw (3.4, 2.5) -- (U3_2b);
\draw (4.4, 2.5) -- (U3_2b);

% --- T4 Branches (Red Flow Out) ---
\draw (D4_1) -- (T4);
\draw[strred, midarrow] (7,-1) -- (T4);
\draw[strred, midarrow] (D4_2L) -- (D4_1);
\draw (D4_1) -- (D4_2R);

% --- T5 Branches ---
%\draw (U5_1) -- (T5);
%\draw (8.2, 1.2) -- (U5_1);
%\draw (9.2, 1) -- (U5_1);

% --- T6 Branches (Complex Top Tree) ---
\draw (C1) -- (T6);
\draw[strred, midarrow] (C2) -- (C1);
\draw (C3) -- (C1);
\draw[strred, midarrow] (8.7, 2.4) -- (C3);
\draw (9.2, 1.8) -- (C3);
\draw [strblue, midarrow] (C4) -- (C2);
\draw (C5) -- (C2);
\draw (C6) -- (C2);
\draw[strblue, midarrow] (C7) -- (C6);
\draw (C8) -- (C6);

% --- T6 Branches (Bottom Tree) ---
\draw (D1) -- (T6);
\draw (D2) -- (T6);
\draw (D3) -- (D1);
\draw (D4) -- (D1);
\draw (D5) -- (D2);

% --- T7 Branches (Vertical Top Tree) ---
\draw[strred, midarrow] (R1) -- (T7);
\draw[strblue, midarrow] (R2) -- (R1);
\draw (11, 3.1) -- (R2);
\draw (12, 3.1) -- (R2);
\draw (R3) -- (R1);
\draw[strblue, midarrow] (R4) -- (R3);
\draw[strblue, midarrow] (R5) -- (R3);

% --- T8 Branches ---
% Blue Down
\draw[strblue, midarrow] (D8_1) -- (T8);
\draw (D8_1) -- (14, -1.8);
\draw (D8_1) -- (13.2, -2);
% Black Up (3 tips)
\draw (U8_1) -- (T8);
\draw (13.2, 1.7) -- (U8_1);
\draw[strblue, midarrow] (14, 1.8) -- (U8_1);
\draw (14.5, 1) -- (U8_1);

% -------------------------
% 3. Draw Junction Dots
% -------------------------
\foreach \p in {
T1, T2, T3, T4, T6, T7, T8,
L1, UL, DL,
D2_1, D2_2a,
U3_1, U3_2b,
D4_1,
C1, C2, C3, C6,
D1, D2,
R1, R2, R3,
D8_1, U8_1}
{
\node[dot] at (\p) {};
}
\end{tikzpicture}

\caption{Every edge of a coherent family $H$ (blue or red) is below some maximal edge (red).}
\label{fig:coherent_max}
\end{figure}

\begin{figure}[htbp]
\centering

\begin{tikzpicture}[
scale=0.75, % Scales down the overall size of the diagram
very thick,
line cap=round,
line join=round,
line width=1.8pt, % Adjusted for smaller scale
midarrow/.style={
postaction={
decorate,
decoration={
markings,
% 0.5 is the exact middle. 
% (0.25 would be a quarter way, 0.9 near the end)
mark=at position 0.65 with {\arrow{>}}
}
}
},
strblue/.style={blue!90!cyan},
dot/.style={circle, fill=black, inner sep=2pt, outer sep=0pt},
>={Straight Barb[length=7pt, width=9pt]},
marr/.style={
postaction={
decorate,
decoration={
markings,
mark=at position 0.8 with {\arrow{>}} % Arrowheads exactly in the middle
}
}
}
]

% -------------------------
% 1. Define Coordinates
% -------------------------

% Main Trunk Nodes (y=0)
\coordinate (T1) at (0, 0);
\coordinate (T2) at (2, 0);
\coordinate (T3) at (4, 0);
\coordinate (T4) at (6, 0);
\coordinate (T5) at (8, 0);
\coordinate (T6) at (8, 0);
\coordinate (T7) at (10.5, 0);
\coordinate (T8) at (12.5, 0);
\coordinate (T9) at (14.5, 0);

% T1 - Far Left Branches
\coordinate (L1) at (-1.5, 0);
\coordinate (UL) at (-1.2, 1);
\coordinate (DL) at (-1.2, -1);

% T2 - Down Right Branch
\coordinate (D2_1) at (2.8, -1);
\coordinate (D2_2a) at (2.4, -2);
\coordinate (D2_2b) at (3.5, -1.6);

% T3 - Up Left Branch
\coordinate (U3_1) at (3.2, 1);
\coordinate (U3_2a) at (2.4, 1.5);
\coordinate (U3_2b) at (3.8, 1.8);

% T4 - Down Left Branch (Red Path)
\coordinate (D4_1) at (5.2, -1);
\coordinate (D4_2L) at (4.4, -2);
\coordinate (D4_2R) at (5.8, -1.8);

% T5 - Up Right Branch
\coordinate (U5_1) at (8.6, 0.6);

% T6 - Complex Top Left Tree (Red / Blue Path)
\coordinate (C1) at (7.5, 1.2);
\coordinate (C2) at (7, 2.2);
\coordinate (C3) at (8.3, 1.5);
\coordinate (C4) at (7.2, 3);
\coordinate (C5) at (8, 2.8);
\coordinate (C6) at (6, 3);
\coordinate (C7) at (5.6, 3.8);
\coordinate (C8) at (6.6, 3.8);

\coordinate (D1) at (8, -1);
\coordinate (D2) at (9.2, -1);
\coordinate (D3) at (7.6, -2);
\coordinate (D4) at (8.3, -2);
\coordinate (D5) at (9.4, -2);

% T7 - Vertical Top Right Tree (Red / Blue Path)
\coordinate (R1) at (11, 1.5);
\coordinate (R2) at (11.5, 2.5);
\coordinate (R3) at (9.5, 3);
\coordinate (R4) at (8.8, 3.8);
\coordinate (R5) at (10.2, 3.8);

% T8 - End Branches
\coordinate (D8_1) at (13.5, -1.2);
\coordinate (U8_1) at (13.5, 0.8);

% -------------------------
% 2. Draw Edges
% -------------------------

% --- Main Trunk ---
\draw (T1) -- (T2);
\draw[midarrow, strblue] (T2) -- (T3); % Red trunk flow
\draw (T3) -- (T4);
\draw (T4) -- (T5);
\draw (T5) -- (T6);
\draw (T6) -- (T7);
\draw[strblue, midarrow] (T7) -- (T8); % Red trunk flow
\draw (T8) -- (T9);

% --- T1 Branches ---
% Mid
\draw (L1) -- (T1);
\draw[strblue, midarrow] (-2.2, 0.3) -- (L1);
\draw (-2.2, -0.3) -- (L1);
% Top (Blue In)
\draw (UL) -- (T1);
\draw[strblue, midarrow] (-1.8, 1.8) -- (UL);
\draw (-0.5, 1.8) -- (UL);
% Bottom (Blue In to Black)
\draw[strblue, midarrow] (DL) -- (T1);
\draw (-2, -0.8) -- (DL);
\draw (-1.8, -1.8) -- (DL);

% --- T2 Branches ---
\draw (D2_1) -- (T2);
\draw (D2_2a) -- (D2_1);
\draw (D2_2b) -- (D2_1);
\draw (2, -2.6) -- (D2_2a);
\draw (2.8, -2.6) -- (D2_2a);

% --- T3 Branches ---
\draw (U3_1) -- (T3);
\draw (U3_2a) -- (U3_1);
\draw (U3_2b) -- (U3_1);
\draw (3.4, 2.5) -- (U3_2b);
\draw (4.4, 2.5) -- (U3_2b);

% --- T4 Branches (Red Flow Out) ---
\draw (D4_1) -- (T4);
\draw[strblue, midarrow] (7,-1) -- (T4);
\draw[strblue, midarrow] (D4_2L) -- (D4_1);
\draw (D4_1) -- (D4_2R);

% --- T5 Branches ---
%\draw (U5_1) -- (T5);
%\draw (8.2, 1.2) -- (U5_1);
%\draw (9.2, 1) -- (U5_1);

% --- T6 Branches (Complex Top Tree) ---
\draw (C1) -- (T6);
\draw[strblue, midarrow] (C2) -- (C1);
\draw (C3) -- (C1);
\draw[strblue, midarrow] (8.7, 2.4) -- (C3);
\draw (9.2, 1.8) -- (C3);
\draw [strblue, midarrow] (C4) -- (C2);
\draw (C5) -- (C2);
\draw (C6) -- (C2);
\draw[strblue, midarrow] (C7) -- (C6);
\draw (C8) -- (C6);

% --- T6 Branches (Bottom Tree) ---
\draw (D1) -- (T6);
\draw (D2) -- (T6);
\draw (D3) -- (D1);
\draw (D4) -- (D1);
\draw (D5) -- (D2);

% --- T7 Branches (Vertical Top Tree) ---
\draw[strblue, midarrow] (R1) -- (T7);
\draw[strblue, midarrow] (R2) -- (R1);
\draw (11, 3.1) -- (R2);
\draw (12, 3.1) -- (R2);
\draw (R3) -- (R1);
\draw[strblue, midarrow] (R4) -- (R3);
\draw[strblue, midarrow] (R5) -- (R3);

% --- T8 Branches ---
% Blue Down
\draw[strblue, midarrow] (D8_1) -- (T8);
\draw (D8_1) -- (14, -1.8);
\draw (D8_1) -- (13.2, -2);
% Black Up (3 tips)
\draw (U8_1) -- (T8);
\draw (13.2, 1.7) -- (U8_1);
\draw[strblue, marr] (U8_1) -- (14, 1.8);
\draw[strblue, line cap=round, dash pattern=on 0pt off 4pt] (14, 1.8) -- (14.7, 2.9);
\draw (14.5, 1) -- (U8_1);

\node[blue!90!cyan] at (14.85, 3.15) {\Large $\xi$};

% -------------------------
% 3. Draw Junction Dots
% -------------------------
\foreach \p in {
T1, T2, T3, T4, T6, T7, T8,
L1, UL, DL,
D2_1, D2_2a,
U3_1, U3_2b,
D4_1,
C1, C2, C3, C6,
D1, D2,
R1, R2, R3,
D8_1, U8_1}
{
\node[dot] at (\p) {};
}
\end{tikzpicture}

\caption{A coherent family $H$ of edges (blue) with no maximal edges, hence selecting an end $\xi$.}
\label{fig:coherent_end}
\end{figure}
\end{prop}

\begin{proof}
Part \labelcref{item:maximal_edges_transversal} is clear, so we only deal with part \labelcref{item:coherent_dichotomy}.

First note that for any two directed edges $e_0, e_1$ of $T$, the $\leq _T$-interval $[e_0, e_1]_{\le_T}$ is finite, either being empty or consisting of all directed edges of $T$ lying on the directed geodesic path through $T$ from $e_0$ to $e_1$.
In addition, for any two directed edges $e_0, e_1\in H$, there are only finitely many $e\in H$ with $e_0 \le_T e$ and $e_1 \not\le_T e$, because coherence of $H$ implies that every such $e$ must lie on the geodesic path through $T$ between $\orig(e_0)$ and $\orig(e_1)$, see \cref{fig:above_one_edge_but_not_other}.

\begin{figure}[htbp]
\centering

\begin{tikzpicture}[
scale=0.8, % Adjusted scale
blacknode/.style={circle, fill=black, inner sep=0pt, minimum size=3.5pt}, % Scaled down from 5.5pt
branch/.style={line width=1.2pt, black, line cap=round, line join=round}, % Scaled down from 1.8pt
redpath/.style={
line width=1.6pt, % Scaled down from 2.5pt
red!90!black, 
line cap=round, 
line join=round, 
decoration={markings, mark=at position 0.65 with {\arrow{Straight Barb[length=2mm, width=2.3mm]}}}, % Scaled down from 3mm/3.5mm
postaction={decorate}
},
bluepath/.style={
line width=1.6pt, % Scaled down from 2.5pt
blue!90!cyan, 
line cap=round, 
line join=round, 
decoration={markings, mark=at position 0.65 with {\arrow{Straight Barb[length=2mm, width=2.3mm]}}}, % Scaled down from 3mm/3.5mm
postaction={decorate}
},
blueedge/.style={
line width=1.6pt, % Scaled down from 2.5pt
blue!90!cyan, 
line cap=round, 
line join=round
}
]

% 1. Black Branches
% Branches from (2,7)
\draw[branch] (2,7) -- (2,8);
\draw[branch] (2,8) -- (1.5, 8.8);
\draw[branch] (2,8) -- (2, 8.8);
\draw[branch] (2,8) -- (2.6, 8.8);

\draw[branch] (2,7) -- (1.2, 7.6);
\draw[branch] (1.2, 7.6) -- (0.8, 8.3);
\draw[branch] (1.2, 7.6) -- (0.3, 7.8);

\draw[branch] (2,7) -- (0.8, 6.6);
\draw[branch] (0.8, 6.6) -- (0.4, 7.3);
\draw[branch] (0.8, 6.6) -- (0.2, 6.3);

% Branches from (3,6)
\draw[branch] (3,6) -- (2.3, 5.3);
\draw[branch] (2.3, 5.3) -- (1.8, 4.8);
\draw[branch] (2.3, 5.3) -- (2.6, 4.3);

% Branches from (4,5)
\draw[branch] (4,5) -- (4, 6.1);
\draw[branch] (4, 6.1) -- (3.6, 6.8);
\draw[branch] (4, 6.1) -- (4.3, 6.9);

\draw[branch] (4,5) -- (5, 5);
\draw[branch] (5, 5) -- (5, 5.8);
\draw[branch] (5, 5) -- (5.6, 4.6);

% Branches from (5,4)
\draw[branch] (5,4) -- (4, 3.8);
\draw[branch] (4, 3.8) -- (3.2, 3.7);
\draw[branch] (4, 3.8) -- (4.1, 3.1);
\draw[branch] (5,4) -- (5.1, 3.1);

% Branches from (6,6)
\draw[branch] (6,6) -- (5, 7);
\draw[branch] (5, 7) -- (4.3, 7.3);
\draw[branch] (5, 7) -- (5.1, 7.8);

\draw[branch] (6,6) -- (6.8, 6.8);
\draw[branch] (6.8, 6.8) -- (6.8, 7.8);
\draw[branch] (6.8, 7.8) -- (6.4, 8.8);
\draw[branch] (6.8, 7.8) -- (7.5, 8.4);

\draw[branch] (6.8, 6.8) -- (7.8, 7.1);
\draw[branch] (7.8, 7.1) -- (8.0, 7.8);
\draw[branch] (7.8, 7.1) -- (8.5, 6.5);

\draw[branch] (6,6) -- (6.9, 6.0);

% Branches from (6,5)
\draw[branch] (6,5) -- (6.4, 5.2);
\draw[branch] (6.4, 5.2) -- (6.8, 5.6);
\draw[branch] (6.4, 5.2) -- (7.2, 5.3);
\draw[branch] (6.4, 5.2) -- (7.0, 4.8);

% Branches from (6,4)
\draw[branch] (6,4) -- (6.5, 4.2);
\draw[branch] (6.5, 4.2) -- (7.0, 4.6);
\draw[branch] (6.5, 4.2) -- (7.2, 3.6);
\draw[branch] (7.2, 3.6) -- (7.6, 3.3);
\draw[branch] (7.2, 3.6) -- (7.8, 4.1);

% Branches from v=(6,3)
\draw[branch] (6,3) -- (7.1, 3.1);

\draw[branch] (6,3) -- (6, 2);
\draw[branch] (6, 2) -- (4.6, 0.9);
\draw[branch] (6, 2) -- (5.8, 0.9);
\draw[branch] (6, 2) -- (6.4, 1.0);
\draw[branch] (6, 2) -- (7.1, 1.0);

\draw[branch] (6,3) -- (7.1, 2);
\draw[branch] (7.1, 2) -- (8.1, 3);
\draw[branch] (8.1, 3) -- (9.0, 2.9);
\draw[branch] (8.1, 3) -- (8.0, 4);
\draw[branch] (8.0, 4) -- (9.0, 4.1);
\draw[branch] (8.0, 4) -- (7.9, 5.3);

\draw[branch] (7.1, 2) -- (8.2, 2.0);
\draw[branch] (8.2, 2.0) -- (9.3, 1.8);
\draw[branch] (8.2, 2.0) -- (8.8, 1.1);

\draw[branch] (7.1, 2) -- (8, 1);
\draw[branch] (8, 1) -- (7.8, 0);
\draw[branch] (8, 1) -- (9.0, -0.1);

% 2. Colored Paths
\draw[redpath] (2,7) -- (3,6);
\draw[redpath] (3,6) -- (4,5);
\draw[redpath] (4,5) -- (5,4);
\draw[redpath] (5,4) -- (6,3);

\draw[bluepath] (6,6) -- (6,5);
\draw[redpath] (6,4) -- (6,5);
\draw[redpath] (6,3) -- (6,4);

% 3. Black Dots on Junctions
\node[blacknode] at (2,7) {};
\node[blacknode] at (3,6) {};
\node[blacknode] at (4,5) {};
\node[blacknode] at (5,4) {};
\node[blacknode] at (6,3) {};
\node[blacknode] at (6,4) {};
\node[blacknode] at (6,5) {};
\node[blacknode] at (6,6) {};

% 4. Labels (Fonts scaled down from \Huge to \Large)
%\node at (5.55, 2.75) {\Large $\boldsymbol{v}$};
\node[red!90!black] at (2.2, 6.2) {\Large $e_0$};
\node[blue!90!cyan] at (5.6, 5.5) {\Large $e_1$};

\end{tikzpicture}

\caption{The red edges are the only edges $e$ coherent with $e_0$ and $e_1$ and satisfy $e_0 \le_T e$ and $e_1 \not\le_T e$.}
\label{fig:above_one_edge_but_not_other}
\end{figure}

Now suppose that some $e_0 \in H$ is not below a $\leq_T$-maximal element of $H$.
Then for each $n \in \N$, having defined $e_n \in H$ with $e_0 \le_T e_n$, we may define $e_{n+1} \in H$ as the immediate successor of $e_n$ in the $\le_T$ order on $H$. 
(The existence of some successor in $H$ follows from our assumption on $e_0$ and the finiteness of $\leq _T$-intervals, and its uniqueness follows from $H$ being coherent.) 
By construction, the $e_n$ all lie on a geodesic ray through $T$, so the sequence $(\orig(e_n))_{n\in\N}$ converges to the end $\xi$ of $T$ represented by this ray. 

Since $e_0\leq _T e_n$ for all $n\in \N$, for each $e \in H$ there can be only finitely many $n$ with $e \nle_T e_n$. 
It follows that $(H,\leq_T)$ is a directed set and that the net $(\orig(e))_{e \in H}$ converges to the end $\xi$.
\end{proof}

This is all compatible with the Borel setting, which we make precise as follows.

\begin{prop}\label{Borel_coherent}
Let $T$ be an acyclic locally countable Borel graph on a standard Borel space $X$ and let $H$ be a Borel set of directed edges of $T$ that is coherent and meets every $T$-component. 

\begin{enumerate}[(i)]
\item If $(H,\leq_T)$ has no maximal element then for each $x\in X$, letting $H_x$ be the set of all directed edges in $H$ belonging to the $T$-component of $x$, the limit $\xi_x \defeq \lim _{e \in H_x} \orig (e) \in \del_T [x]_{E_T}$ exists, and the map $x \mapsto \xi _x$ is an $E_T$-invariant Borel end selection of one end from each $T$-component.

\item If every directed edge in $H$ is $\leq_T$-below a $\leq_T$-maximal element of $H$, then the equivalence relation $F_H$ is smooth.
\end{enumerate}
\end{prop}
\begin{proof}
This follows from \cref{coherent_orientation_prop} and Luzin--Novikov uniformization.
\end{proof}

We often use this through the following lemma.

\begin{lemma}[Pruning]\label{pruning}
Let $T$ be an acyclic locally countable mcp Borel graph on a standard probability space $(X,\mu )$, with associated Radon--Nikodym cocycle $\rho$.

\begin{enumerate}[(a)]
\item \label{pruning:coherent} Let $H$ be a coherent Borel set of directed edges such that every $e \in H$ is $\leq_T$-below some $\leq_T$-maximal element of $H$. 
Then, after discarding a null set, the half-space $\Vo (e)$ is $\rho$-finite for every $e \in H$. 

\item \label{pruning:convex} Let $Y$ be a $T$-convex Borel subset of $X$. 
Then, after discarding a null set, the half-space $\Vo (e)$ is $\rho$-finite for every directed edge $e$ of $T\rest{[Y]_{E_T}}$ with $\Vo (e) \cap Y =\0$.
\end{enumerate}
\end{lemma}

\begin{proof}
Part \labelcref{pruning:convex} follows from \labelcref{pruning:coherent} by taking as $H$ the set of all directed edges $e$ of $T\rest{[Y]_{E_T}}$ with $\Vo (e)\cap Y =\0$. For part \labelcref{pruning:coherent}, note that the associated equivalence relation $F_H$ is smooth and hence, after discarding a null set, every $F_H$-class is $\rho$-finite. The conclusion follows since for each $e\in H$ the half-space $\Vo (e)$ is contained in an $F_H$-class.
\end{proof}

\subsubsection{The Helly property for subtrees}

\cref{coherent_orientation_prop}\labelcref{item:coherent_dichotomy} may be seen as a reformulation of the following version of the Helly property for subtrees. 
%While it is clearly a special case, it also implies the general case as shown in the proof to follow.

\begin{lemma}[Helly property for subtrees]\label{Helly}
Let $T$ be a tree and let $\FC$ be a family of subtrees of $T$ whose pairwise intersections are nonempty. Then the subtrees in $\FC$ all share a common end or a common vertex.
\end{lemma}
\begin{proof}
Let $H$ be the set of all directed edges $e$ of $T$ such that $\Vo (e)$ is disjoint from some member of $\mathcal{F}$. We may assume $H$ is nonempty, since otherwise either $\mathcal{F}=\{ T \}$ or $\mathcal{F}=\emptyset$, so the conclusion is either trivially or vacuously true. 
The fact that the subtrees in $\mathcal{F}$ have pairwise nonempty intersection implies that $H$ is coherent.
If there is no $\leq_T$-maximal element then by \cref{coherent_orientation_prop}\labelcref{item:coherent_dichotomy} the limit $\lim _{e\in H} \orig (e)$ is an end common to all the subtrees in $\mathcal{F}$. Otherwise, there is some $e\in H$ that is $\leq_T$-maximal, and hence $\term (e)$ is a vertex of every subtree in $\mathcal{F}$.
\end{proof}

\subsection{Topologies on the end completion}\label{subsec:topologies}

Let $T$ be a tree on a vertex set $V$. 
Recall that the \defterm{edge-removal topology} on $\-V^T$ is the topology generated by the sets $U \cup \del_T U$, with $U$ ranging over all half-spaces of $T$.
Note that the sets $C \cup \del _T C$, where $C$ ranges over subsets of $V$ with finite edge-boundary in $T$, form a clopen basis for this topology, making it zero-dimensional (and justifying the name).

In fact, $\-V^T$ equipped with the edge-removal topology may also be identified with the Stone dual of the Boolean algebra of all subsets of $V$ having finite edge-boundary in $T$.
This makes it clear that the edge-removal topology is compact. 
Nevertheless, we give an alternative geometric proof of compactness for the reader's convenience.

If $C \subseteq V$ has finite edge-boundary in $T$, then the closure of $C$ in the edge-removal topology is exactly $C \cup \del _T C$.

\begin{prop}
    The edge-removal topology is compact.
\end{prop}
\begin{proof}
For a subset $A$ of $\-{V}^T$, let $\cl{e}{T}(A)$ denote the closure of $A$ in the edge-removal topology.
By the Alexander subbase theorem, it is enough to show that, given a collection $\HC$ of half-spaces of $T$ such that the family $\set{\cl{e}{T}(U)}_{U \in \HC}$ has the finite intersection property, the intersection $\bigcap_{U \in \HC} \cl{e}{T}(U)$ is nonempty.
This follows by applying the Helly property (\cref{Helly}) to the family $\HC$.
\end{proof}

In particular, if $V$ is countable, then the edge-removal topology on $\-V^T$ is Polish, by the Urysohn metrization theorem.
In addition, $\del _T V$ is $G_\delta$ in $\-V^T$ (and hence Polish), since each ball in $V$ (with respect to the graph distance of $T$) is closed.

While in this article we mainly work with the edge-removal topology, another natural topology on $\-V^T$ to consider is the \defterm{vertex-removal topology}, which is generated by all subsets of $V$ together with the sets $U\cup\del_T U$, with $U$ ranging over all half-spaces of $T$.
The vertex-removal topology is not compact when $T$ is not locally finite, but on the subspace $\del_T V$, the vertex-removal and the edge-removal topologies coincide (even when $T$ is not locally finite).
In fact, the closures of half-spaces coincide in both topologies, and the collection of all such half-space closures that contain a given end $\xi \in \del_T V$ is a neighborhood basis of $\xi$ in both topologies.
Finally, observe that the vertex-removal topology coincides with the edge-removal topology when $T$ is locally finite.

\subsection{The classical Adams dichotomy}\label{subsec:classical_Adams}

Adams \cite{Adams90} proves that an ergodic pmp treeing is amenable if and only if a.e.\ component has at most two ends. 
In particular, in the nonamenable case a.e.\ component has at least three ends; 
this can be strengthened to perfectly many ends, as remarked in the introduction. 
We now give a short, self-contained proof of the full classical Adams dichotomy, using \cref{smooth=rho-finite,backward_mass_transport,pruning}, and \cref{end_selection}.

\begin{theorem}[Classical Adams Dichotomy]\label{classic_Adams}
Let $T$ be a treeing of a pmp countable Borel equivalence relation $E$ on $(X,\mu )$. 
\begin{enumerate}[(a)]
\item If $E$ is $\mu$-amenable then a.e.\ $T$-component has at most two ends.
\item If $E$ is $\mu$-nowhere amenable then the space of ends of a.e.\ $T$-component is nonempty with no isolated points.
\end{enumerate}
\end{theorem}

\begin{proof}
Assume first that $E$ is $\mu$-amenable.
By \cref{smooth=rho-finite}, after discarding the vertices contained in finite $T$-components, we may assume that $E$ is $\mu$-nowhere smooth.

Let $x\mapsto \EC_E(x)$ be the essentially unique maximum $E$-invariant Borel $T$-end selection granted by \cref{end_selection}.
Considering separately the invariant Borel sets where $\EC_E(x)$ has cardinality one and two, it is enough to analyze the following two cases.

\begin{case}{1}{$|\EC_E(x)|=1$ for all $x\in X$} Let $f : X \to X$ be the map that moves each $x \in X$ one step closer to the unique end in $\EC_E(x)$.
By \cref{smooth=rho-finite}\labelcref{part:coc-finite_preimages}, after discarding a null set we may assume that $f$ is finite-to-one.
We now iteratively remove leaves from $T$: define $Y\defeq \bigcap _{n\in \N}X_n$ where $X_0\defeq X$ and $X_{n+1}$ is the set of all vertices of $X_n$ of degree at least $2$ in  $T\rest{X_n}$. 
Then each $T$-component not meeting $Y$ is one-ended and, since $f$ is finite-to-one, $T\rest{Y}$ has no leaves and $f(Y)=Y$. 
It remains to show that $Y$ is null, which will imply that a.e.\ $T$-component is one-ended. 
If $Y$ is non-null then, since each $x\in Y$ has at least one $f$-preimage still in $Y$, \cref{backward_mass_transport} shows that, after discarding a null set, $f\rest{Y}$ is a bijection on $Y$. 
This contradicts the essential maximality of the $T$-end selection $x\mapsto\EC_E(x)$. 
\end{case}

\begin{case}{2}{$|\EC_E(x)|=2$ for all $x\in X$}
Let $L$ be the set of all vertices $x$ that lie on the $T$-geodesic between the two ends of $\EC_E(x)$, and let $p:X\to L$ be the map sending each vertex to the unique closest point in $L$.  
Then \cref{smooth=rho-finite} implies that, after discarding a null set, $p$ is finite-to-one, and hence for each $x\in X$ the two ends in $\EC_E(x)$ are the only ends of the $T$-component of $x$.
\end{case}

Assume now that $E$ is $\mu$-nowhere amenable.
By \cref{end_selection}, after discarding a null set each $T$-component has at least three ends.
Then the set $H$, of all directed edges $e$ of $T$ such that $\Vo(e)$ accumulates on at most one end, is coherent.
Moreover, every $e \in H$ is below a $\leq _T$-maximal element of $H$, since otherwise \cref{coherent_orientation_prop}\labelcref{item:coherent_dichotomy} and the definition of $H$ would imply that some component has at most two ends. 
Thus, by  \cref{pruning}\labelcref{pruning:coherent} and the assumption that $E$ is pmp, after discarding a null set each half-space $\Vo(e)$ with $e\in H$ is finite, hence $T$ has no isolated ends.
\end{proof}

\section{Ends of topographical significance}\label{sec:vanishing_ends}

In this section, we define three notions of smallness/largeness for ends in an acyclic mcp graph. 
The notion of vanishing ends will play a particularly important role in our analysis.

\begin{defn}\label{def:barytropic-vanishing-finite_geod}
Let $T$ be an acyclic locally countable mcp graph on a standard probability space $(X,\mu)$.
As usual, we let $\rho : E_T \rightarrow \Rpos$ denote the associated Radon--Nikodym cocycle.
Let $C$ be a $T$-component and $\xi \in \partial_T C$.
\begin{itemize}
\item We say that $\xi$ is \dfn{$\rho$-barytropic} if every half-space of $T$ whose closure contains $\xi$ is $\rho$-infinite.
Otherwise, i.e., if there is a $\rho$-finite half-space of $T$ whose closure contains $\xi$, we say that $\xi$ is \dfn{$\rho$-abarytropic}.

\item Define the \dfn{upper $\rho$-weight of $\xi$ with respect to $x$} by
\[
\urho^x(\xi) \defeq \limsup_{y \to \xi} \rho^x(y).
\]
We say that $\xi$ is \dfn{$\rho$-vanishing} if $\urho^x(\xi) = 0$, and $\xi$ is called \dfn{$\rho$-nonvanishing} if $\urho^x(\xi) > 0$, i.e., there is $\e > 0$ such that for every half-space $U \subseteq [x]_T$ with $\-{U} \ni \xi$, there is some $y \in U$ with $\rho^x(y) \ge \e$.

\begin{figure}[htbp]
\centering

\begin{tikzpicture}[
x=0.6cm, y=0.5cm, % Scaled down by 50%
every path/.style={line cap=round, line join=round},
dot/.style={circle, fill=black, inner sep=0pt, minimum size=2.5pt}, % Smaller dots
bigdot/.style={circle, fill=black, inner sep=0pt, minimum size=4pt}, % Smaller dots
myred/.style={color={rgb,255:red,210; green,0; blue,0}},
rdot/.style={circle, fill={rgb,255:red,210; green,0; blue,0}, inner sep=0pt, minimum size=4pt}
]

% Define custom colors matching the image
\colorlet{myblue}{blue!90!cyan}
\definecolor{mygreen}{RGB}{10, 120, 80}

% ---------------------------------------------------------
% 1. Coordinates for the main blue path nodes
% (Coordinates stay the same, but the x/y grid is smaller)
% ---------------------------------------------------------
\coordinate (N0) at (-3.0, 3.0);   % x
\coordinate (N1) at (-1.0, 2.5);
\coordinate (N2) at (0.5, 1.2);    % root of y0 subtree
\coordinate (N3) at (1.5, -0.3);   % root of y1 subtree
\coordinate (N4) at (2.2, -1.8);
\coordinate (N5) at (3.2, -3.5);   % root of y2 subtree
\coordinate (N6) at (3.5, -5.2);

% ---------------------------------------------------------
% 2. Dashed Green Boundaries (U0, U1, U2)
% ---------------------------------------------------------
% U0
%\draw[line width=0.8pt, dashed, color=mygreen] % Thinner lines
%    (1.5, 7.0) .. controls (0.5, 6.5) and (-1.0, 6.4) ..
%    (-2.0, 2.75) .. controls (-2.5, 0.5) and (-2.5, -2.0) .. (-2.0, -4.0);

% U0 - Sweeps right of N1 and left of y0
\draw[line width=0.8pt, dashed, color=mygreen] 
plot [smooth, tension=0.6] coordinates { 
(0.8, 6.5)     % Top right origin
(-1.3, 5)     % Avoids rightmost N1 branch
(-2.4, 2.5) 
(-3, -0.5)   % Passes between y0 branch and N0 bottom branch
(-3.5, -4.0)   % Exits bottom left
};

\node[mygreen] at (-1.4, 3.3) {$\boldsymbol{U_0}$}; % Removed \Large

% U1
\draw[line width=0.8pt, dashed, color=mygreen]
(4.0, 2) .. controls (4.0, 2) and (2.0, 2.0) ..
(0.8, 0.45) .. controls (0, -1) and (-0.4, -4.0) .. (-0.6, -6.5);
\node[mygreen] at (1.8, 0.5) {$\boldsymbol{U_1}$};

% U2
\draw[line width=0.8pt, dashed, color=mygreen]
(6.1, -1.8) .. controls (5, -1.8) and (3, -1.8) ..
(2.5, -3) .. controls (2.0, -4.0) and (2.0, -6.0) .. (2.0, -8.0);
\node[mygreen] at (3.6, -2.8) {$\boldsymbol{U_2}$};

% ---------------------------------------------------------
% 3. Main Blue Path
% ---------------------------------------------------------
% Solid curve
\draw[line width=2pt, myblue] plot [smooth, tension=0.7] coordinates { (N0) (N1) (N2) (N3) (N4) (N5) (N6) };
% Dashed trailing end
\draw[line width=2pt, myblue, line cap=round, dash pattern=on 0pt off 4pt] (N6) .. controls (3.7, -6) and (3.7, -6.2) .. (3.8, -7.7);
\node[myblue] at (3.8, -8.3) {$\boldsymbol{\xi}$};
% ---------------------------------------------------------
% 4. Black Tree Branches (drawn behind dots)
% ---------------------------------------------------------
\begin{scope}[line width=0.6pt] % Thinner black lines
% --- N0 (x) Subtree ---
\draw (N0) -- (-2.2, 4.2);
\coordinate (n0ul) at (-3.8, 4.5); \draw (N0) -- (n0ul);
\draw (n0ul) -- (-4.8, 5.5);
\draw (n0ul) -- (-3.2, 5.8);
\coordinate (n0dl) at (-4.5, 1.8); \draw (N0) -- (n0dl);
\draw (n0dl) -- (-5.2, 2.8);
\coordinate (n0dl2) at (-5.5, 0.5); \draw (n0dl) -- (n0dl2);
\draw (n0dl2) -- (-6.5, 1.5);
\draw (n0dl2) -- (-6.0, -1.0);

% --- N1 Subtree ---
\coordinate (n1ur) at (-0.2, 4.0); \draw (N1) -- (n1ur);
\draw (n1ur) -- (-0.8, 5.2);
\coordinate (n1ur2) at (1.2, 4.8); \draw (n1ur) -- (n1ur2);
\draw (n1ur2) -- (0.8, 6.0);
\draw (n1ur2) -- (2.2, 3.8);

% --- N2 Subtree (contains y0) ---
\coordinate (n2dl) at (-0.5, -0.5); \draw (N2) -- (n2dl);
\coordinate (y0) at (-1.8, -0.6);   \draw (n2dl) -- (y0);
\draw (y0) -- (-2.8, -1.2);
\draw (y0) -- (-2.2, -2.5);
\coordinate (n2dl2) at (-0.4, -2.2); \draw (n2dl) -- (n2dl2);
\draw (n2dl2) -- (-1.2, -3.5);

% --- N3 Subtree (contains y1) ---
\coordinate (y1) at (3.5, 0); \draw (N3) -- (y1);
\draw (y1) -- (4.2, 1.0);
\draw (y1) -- (5.0, -0.2);
\draw (y1) -- (4.5, -1.2);

% --- N4 Subtree ---
\coordinate (n4dl) at (1.2, -3.5); \draw (N4) -- (n4dl);
\draw (n4dl) -- (0, -3.6);
\coordinate (n4dl2) at (0.5, -5.5); \draw (n4dl) -- (n4dl2);
\draw (n4dl2) -- (-0.3, -6.2);

% --- N5 Subtree (contains y2) ---
\coordinate (n5r) at (5.0, -3.8); \draw (N5) -- (n5r);
\draw (n5r) -- (6.0, -2.2);
\coordinate (y2) at (5.8, -5.2); \draw (n5r) -- (y2);
\draw (y2) -- (7.0, -5.5);
\draw (y2) -- (5.8, -6.8);

% --- N6 Subtree ---
\draw (N6) -- (2.5, -6.8);
\end{scope}

% ---------------------------------------------------------
% 5. Draw the Dots (layered on top)
% ---------------------------------------------------------
% Blue path nodes
\node[bigdot] at (N0) {};
\node[dot] at (N1) {};
\node[dot] at (N2) {};
\node[dot] at (N3) {};
\node[dot] at (N4) {};
\node[dot] at (N5) {};
\node[dot] at (N6) {};

% Subtree internal branching nodes
\node[dot] at (n0ul) {};
\node[dot] at (n0dl) {};
\node[dot] at (n0dl2) {};
\node[dot] at (n1ur) {};
\node[dot] at (n1ur2) {};
\node[dot] at (n2dl) {};
\node[dot] at (n2dl2) {};
\node[dot] at (n4dl) {};
\node[dot] at (n4dl2) {};
\node[dot] at (n5r) {};

% Red nodes
\node[rdot] at (y0) {};
\node[rdot] at (y1) {};
\node[rdot] at (y2) {};

% ---------------------------------------------------------
% 6. Text Labels
% ---------------------------------------------------------
\node[below] at (N0) {$\boldsymbol{x}$};
\node[above=0pt, myred] at (y0) {$\boldsymbol{y_0}$}; % Removed \Large
\node[below=0pt, myred] at (y1) {$\boldsymbol{y_1}$};
\node[above=5pt, right, myred] at (y2) {$\boldsymbol{y_2}$};

\end{tikzpicture}

\caption{A nonvanishing end $\xi$ with $\rho^x(y_n) \ge \e$.}
\label{fig:nonvanishing_end}
\end{figure}

\item Say that $\xi$ \dfn{has $\rho$-finite geodesics} if every (equivalently: some) geodesic ray converging to $\xi$ is $\rho$-finite; 
otherwise, we say that $\xi$ has \dfn{$\rho$-infinite geodesics}.
\end{itemize}
\end{defn}

Note that all of these definitions except for $\urho^x(\xi)$ only depend on $\xi$ and do not depend on the choice of basepoint in $C$.

Observe that $\xi$ being $\rho$-abarytropic implies that $\xi$ is both $\rho$-vanishing and has $\rho$-finite geodesics.
We show in \cref{vanishing_implies_cocycle-finite_geodesics} that in a.e.\ $T$-component, every $\rho$-vanishing end of $T$ has $\rho$-finite geodesics.
However, no other implication is true even after discarding a $E_T$-invariant null set.
Indeed, \cref{example:least_deletion} gives an example with ends that have $\rho$-finite geodesics, but are $\rho$-barytropic, in fact $\rho$-nonvanishing.
Furthermore, \cref{example:free_group_boundary} gives an example with ends that are $\rho$-vanishing and have $\rho$-finite geodesics but are $\rho$-barytropic.

\begin{defn}\label{def:mono-poly}
Let $(X,\mu)$ be a standard probability space and $f : X \to X$ be an acyclic countable-to-one Borel function such that $E_f$ is mcp and $\mu$-nowhere smooth.
Let $\rho : E_f \to \Rpos$ be the associated Radon--Nikodym cocycle.
Then we say that $f$ is 
\begin{itemize}
\item \dfn{$\rho$-monobarytropic} if a.e.\ $T_f$-component has exactly one $\rho$-barytropic end.

\item \dfn{$\rho$-dibarytropic} if a.e.\ $T_f$-component has exactly two $\rho$-barytropic ends.

\item \dfn{$\rho$-polybarytropic} if a.e.\ $T_f$-component has at least three $\rho$-barytropic ends.
\end{itemize}
\end{defn}

We will show in \cref{essential_ends=barytropic} (or more generally, \cref{measure_class_invariant}) that the notions in \cref{def:mono-poly} are invariants of the measure class of $\mu$.
Therefore, it is safe to drop the reference to $\rho$ in our terminology.
We will also do so for the terms barytropic and nonvanishing if $\rho$ is clear from the context.

\subsection{Smoothness via abarytropic and vanishing ends}

\begin{prop}\label{smooth_iff_all_vanishing_ends}
Let $T$ be an acyclic locally countable mcp graph on a standard probability space $(X,\mu)$, with associated Radon--Nikodym cocycle $\rho : E_T \rightarrow \Rpos$.
The following are equivalent:
\begin{enumerate}[(1)]
\item\label{item:smooth} $E_T$ is smooth on a conull set.
\item\label{item:loc-finite} All ends of $T$ in a.e.\ $T$-component are $\rho$-abarytropic.
\item\label{item:vanishing} All ends of $T$ in a.e.\ $T$-component are $\rho$-vanishing.
\end{enumerate}
\end{prop}
\begin{proof}
If $E \defeq E_T$ is smooth, then a.e.\ $E$-class is $\rho$-finite by \cref{smooth=rho-finite}. 
It is therefore enough to show that \labelcref{item:vanishing} implies \labelcref{item:smooth}. 

Suppose without loss of generality that all ends of $T$ are $\rho$-vanishing. 
For a given $E$-class $C$ and a vertex $x \in C$, the restriction of $T$ to the $T$-convex hull $C_x$, of $\set{y \in C : \rho^\bullet(y) \geq \rho^\bullet(x)}$, is endless (i.e., contains no infinite geodesic ray) since each of its ends would have to be nonvanishing. 
It follows that $C_x$, being $T$-convex and endless, is closed in the edge-removal topology. 
The family $(C_x)_{x\in C}$ has the finite intersection property because $\rho^\bullet(x) \le \rho^\bullet(z)$ implies $C_z \subseteq C_x$ for all $x,z \in C$. 
Thus, by compactness of the edge-removal topology (or the Helly property for subtrees, \cref{Helly}), the intersection $K_C \defeq \bigcap_{x \in C} C_x$ is nonempty, and $T \rest{K_C}$ is connected and endless. 

We now implement a countable transfinite pruning process, as in 
\cite[p.~379 after Lemma 2.4]{Miller:ends_of_graphs_I}.
Let $K_C^0 \defeq K_C$ and for each countable ordinal $\alpha > 0$, let $K_C^\alpha$ denote the set of all $y \in \bigcap_{\beta < \alpha} K_C^\beta$ which have degree at least two in $T \rest{\bigcap_{\beta < \alpha} K_C^\beta}$.
Because $T \rest{K_C}$ is connected and endless, $K_C^\alpha = \0$ for some countable ordinal $\alpha$ and we let $\alpha(C)$ denote the least such ordinal.
Note that, by compactness, $\alpha(C)$ is not a limit ordinal and let $\beta(C)$ denote its predecessor.
Then $K_C^{\beta(C)}$ is a vertex or a pair of $T$-adjacent vertices since the sets $K_C^\alpha$ are $T$-convex.

By analytic boundedness, specifically \cite[Theorem 35.23]{bible} applied to the rank function $x \mapsto \alpha([x]_E)$, the ranks $\alpha([x]_E)$ are uniformly bounded by some countable ordinal $\alpha$.
Hence the set $\{x \in X : x \in K_{[x]_E}^{\beta([x]_E)}\}$ is Borel and therefore witnesses the smoothness of $E$.
\end{proof}

\subsection{Abarytropy of isolated ends}

\begin{prop}\label{isolated=>vanishing}
Let $T$ be an acyclic locally countable mcp graph on a standard probability space $(X,\mu)$, with associated Radon--Nikodym cocycle $\rho : E_T \rightarrow \Rpos$.
Assume that every $T$-component has at least three ends. 
Then, after discarding a null set, every isolated end of $T$ is $\rho$-abarytropic, hence $\rho$-vanishing.
\end{prop}

\begin{proof}
Let $H$ be the set of directed edges $e$ of $T$ such that $\Vo(e)$ accumulates on at most one end.
Note that $H$ is coherent since each $T$-component has at least three ends. 
Moreover, every $e \in H$ is below a $\leq _T$-maximal element of $H$, since otherwise \cref{coherent_orientation_prop}\labelcref{item:coherent_dichotomy} and the definition of $H$ would imply that some component has at most two ends. 
Thus, the conclusion follows from \cref{pruning}\labelcref{pruning:coherent}.
\end{proof}

\section{The Radon--Nikodym core}

In this section, we introduce the Radon--Nikodym core of an mcp acyclic graph, which is central to our analysis of these graphs.

\begin{defn}
Let $(X,\mu)$ be a standard probability space.

\begin{enumerate}[(a), leftmargin=*]
\item Let $T$ be an acyclic locally countable mcp graph on $(X,\mu)$, with associated Radon--Nikodym cocycle $\rho$.
The \dfn{Radon--Nikodym core} of $T$ (with respect to $\mu$), denoted $\Core_\mu(T)$, is the subset of $\-X=X\cup\del_T X$ obtained by discarding the closure of each $\rho$-finite half-space, i.e.,
\[
\Core_\mu (T)\defeq \{ z\in \-X : \rho^\bullet (U)=\infty \text{ for every half-space $U$ of $T$ with $z\in \-U$}\} .
\]
Thus, $\Core_\mu (T)\cap \del_T X$ is the set of all $\rho$-barytropic ends of $T$.

The \dfn{Radon--Nikodym vertex core} of $T$ (with respect to $\mu$), is the set $\rnc_T\defeq \Core_\mu(T)\cap X$, i.e., the set of all points $x \in X$ such that every half-space of $T$ containing $x$ is $\rho$-infinite.

\item Let $f : X \to X$ be an acyclic countable-to-one Borel function on $X$ such that $E_f$ is mcp, with associated Radon--Nikodym cocycle $\rho$.
The \dfn{directed Radon--Nikodym vertex core} of $f$ (with respect to $\mu$), denoted $\rnc_f$, is defined to be the set of all points $x \in X$ whose back $f$-orbit $f^{-\N}(x)$ is $\rho$-infinite.
\end{enumerate}

We also extend these definitions to the case when $\mu$ is a finite but not necessarily a probability measure: if $\mu = 0$ then define both $\rnc_T$ and $\rnc_f$ to be empty, and if $\mu$ is nonzero then let $\rnc_T$ and $\rnc_f$ be defined with respect to the normalized version of $\mu$. 
\end{defn}

Note that $\rnc_f \supseteq \rnc_{T_f}$, and we show below in \cref{core=X_infty} that, after discarding a null set, these two sets coincide.

\subsection{Basic properties}

The Radon--Nikodym core and vertex core are both convex; the vertex core also has the following minimality property among convex Borel sets.

\begin{lemma}[Minimality of the vertex core]\label{core_minimality}
Let $T$ be a locally countable acyclic mcp graph on $(X,\mu)$.
Let $Y \subseteq X$ be a $T$-convex Borel set. 
Then, after discarding a null set, $\rnc_T\cap [Y]_{E_T}\subseteq Y$.
\end{lemma}

\begin{proof}
This follows immediately from \cref{pruning}\labelcref{pruning:convex} and the definition of $\rnc_T$.
\end{proof}

This directly implies:

\begin{cor}
Let $T$ be a locally countable acyclic mcp graph on $(X,\mu)$. Then, after discarding an $E_T$-invariant Borel null set, $\Core_\mu(T)$ is the convex hull of the set of all barytropic ends of $T$.  
\end{cor}

\begin{lemma}\label{Core-bi-inf-or-perfect} Let $T$ be a locally countable acyclic mcp graph on $(X,\mu)$, and let $T_\infty$ denote the restriction of $T$ to $\rnc_T$. 
Then almost every $T_\infty$-component is either a bi-infinite line or a leafless tree with no isolated ends.
\end{lemma}

\begin{proof}
We first assume without loss of generality that $\rnc_T$ is an $E_T$-complete section. 
In particular, $E_T$ is $\mu$-nowhere smooth by \cref{smooth=rho-finite}, and we may therefore assume that $\rnc_T$ meets no $E_T$-class in a single point. 
Removing all $T_\infty$-leaves from $\rnc_T$ still results in a $T$-convex $E_T$-complete section, hence by \cref{core_minimality}, the set of all points whose $T_\infty$-component contains a $T_\infty$-leaf is null.
We may therefore assume that $T_\infty$ has no leaves.
In addition, we may assume that no $T_\infty$-component is a bi-infinite line, toward showing that $T_\infty$ has no isolated ends.
Then the $T_\infty$-convex hull of the set of all vertices of $T_\infty$-degree at least $3$ is a complete section for $E_T$. 
Hence by \cref{core_minimality}, after discarding a null set, this convex hull is equal to $X_T^{\mu}$, so $T_\infty$ has no isolated ends.
\end{proof}

\subsection{Coincidence of directed and undirected Radon--Nikodym vertex cores}

The following lemma will be strengthened in \cref{cocycle_along_f}\labelcref{forward-nonvanishing}, 
where it is shown that if each forward $f$-end is $\rho$-vanishing then $E_f$ is smooth on a conull set.

\begin{lemma}\label{front_half_cocycle_finite}
Let $f:X\to X$ be an acyclic countable-to-one Borel function on a standard probability space $(X,\mu)$ such that $E_f$ is measure-class preserving, with associated Radon--Nikodym cocycle $\rho$.
Suppose that all forward $f$-ends are $\rho$-abarytropic.
Then $E_f$ is smooth on a conull subset of $X$.
\end{lemma}

\begin{proof}
Consider the set $H$ of all directed edges $e$ of $T_f$ such that $\Vo (e)$ is $\rho$-finite and forward $f$-invariant.
By hypothesis, the set $H$ meets every $T_f$-component.
Each $T_f$-component on which $H$ is incoherent is $\rho$-finite, hence $E_f$ is smooth on the union of all such components. 
We may therefore assume that $H$ is coherent, and hence (by forward $f$-invariance of each $\Vo (e)$ for $e\in H$) any two elements of $H$ in the same $T_f$-component are $\leq _T$-comparable.

Therefore, the intersection of $H$ with each $T_f$-component is an infinite geodesic ray or a geodesic line. 
Let $X_H$ be the set of vertices incident to some directed edge in $H$, and let $D\defeq \{  x\in X_H : \rho ^\bullet (x)> \rho ^\bullet (f^n(x))\text{ for all }n\geq 1 \}$, so that $D$ is a complete section for $E_f$. 

The next return map $f_D:D\to D$ is injective on $D$, and $E_f$ is smooth on the saturation $C$ of $D\setminus f_D(D)$.  
The function $f_D$ restricts to a bijection from $D\setminus C$ to itself, so $D\setminus C$ is null by \cref{backward_mass_transport}, and hence $X\setminus C$ is null as well. 
\end{proof}

\begin{prop}\label{core=X_infty}
Let $f : X \to X$ be an acyclic countable-to-one Borel function on a standard probability space $(X,\mu)$, such that $E_f$ is measure-class preserving, with associated Radon--Nikodym cocycle $\rho$. 
Then, after discarding a null set, the directed Radon--Nikodym vertex core of $f$ coincides with the Radon--Nikodym vertex core of $T_f$. 
\end{prop}

\begin{proof}
The containment $\rnc_{T_f} \subseteq \rnc_f$ is clear.
Let $Y$ be the union of all $T_f$-components containing a $\rho$-finite half-space of $T_f$ that is forward $f$-invariant, and let $Z \defeq X \setminus Y$.
Then it is clear that $\rnc_{T_f} \cap Z= \rnc_f \cap Z$ and, after discarding a null set, $E_f$ is smooth on $Y$ by \cref{front_half_cocycle_finite}.
Therefore, after discarding another null set, $\rnc_f$ is contained in $Z$, completing the proof. 
\end{proof}

\subsection{Null Radon--Nikodym vertex core}

\begin{theorem}[Characterization of null Radon--Nikodym vertex core]\label{null_core}
Let $T$ be an acyclic locally countable mcp Borel graph on a standard probability space $(X,\mu )$, with associated Radon--Nikodym cocycle $\rho$.
Suppose that $E_T$ is $\mu$-nowhere smooth.
Then, the Radon--Nikodym vertex core of $T$ is null if and only if, after discarding a null set, $T=T_f$ for some acyclic Borel function $f:X \to X$ with $\rho$-finite back-orbits.
Moreover, if such a function $f$ exists, and if $g:X\to X$ is another Borel function with $T_g=T$, then $g=f$ on a conull set. 
\end{theorem}

\begin{proof}
It is clear that if $T = T_f$ for some Borel function $f : X \to X$ with $\rho$-finite back orbits, then the Radon--Nikodym vertex core of $T$ is empty.
Now, suppose that such a function $f$ exists and let $g:X\to X$ be another Borel function with $T_g=T$. 
Suppose toward a contradiction that $g(x)\neq f(x)$ for all $x$ belonging to a non-null subset $Y$ of $X$, which we may assume is an $E_T$-complete section without loss of generality.
Then for all $x\in X$, the end $\xi _g (x) \defeq \lim _{n\to \infty}g^n(x)$ is in the closure of some $f$-back orbit, which by assumption is a $\rho$-finite half-space of $T$.
Therefore, applying \cref{front_half_cocycle_finite} to $g$ implies that $E_T$ is smooth on a conull set, a contradiction.

Turning now to the converse, suppose that the Radon--Nikodym vertex core of $T$ is empty, towards the goal of defining a Borel function $f$ with $T = T_f$ and $\rho$-finite back-orbits.
By \cref{smooth=rho-finite}, after discarding a null set, we may assume that each $T$-component is $\rho$-infinite.

Let $H$ be the set of all directed edges $e$ of $T$ such that the half space $\Vo[T](e)$ is $\rho$-finite.
Since each $T$-component is $\rho$-infinite, $H$ is coherent.
Furthermore, it follows from the Radon--Nikodym vertex core being empty that $(H, \le_T)$ has no maximal elements, so \cref{Borel_coherent} yields a Borel $E_{T}$-invariant selection $x\mapsto \xi (x)$ of exactly one end of $T$ from each $T$-component.
Thus, $T = T_f$, where $f: X\to X$ is the Borel function sending each $x \in X$ to its successor along the geodesic ray in $T$ from $x$ to $\xi (x)$. 
Observe that $f$ has $\rho$-finite back-orbits, since each back-orbit of $f$ is contained in $\Vo(e)$ for some $e\in H$.
\end{proof}

From this and a characterization of essentially one-ended acyclic Borel functions, proven later in \cref{char-rho-finite-back-orbits}, we immediately obtain:

\begin{cor}\label{null_core=>ess_one-ended}
Let $T$ be an acyclic locally countable mcp graph on a standard probability space $(X,\mu )$. 
Let $Z$ be the complement of the $E_T$-saturation of $\rnc_T$.
Then $T \rest{Z}$ is essentially at most one-ended.
Furthermore, $\rnc_T$ is null if and only if $T$ is essentially at most one-ended.
\end{cor}

\begin{proof}
Assuming without loss of generality that $Z$ is a complete section and $E_T$ is $\mu$-nowhere smooth, \cref{null_core} implies that $T = T_f$ for some acyclic Borel function $f : X \to X$ with $\rho$-finite back-orbits, hence $T$ is essentially one-ended by \cref{char-rho-finite-back-orbits}.

It remains to show that if $T$ is essentially at most one-ended then $\rnc_T$ is null, which follows from pruning (\cref{pruning}\labelcref{pruning:convex}).
\end{proof}

\subsection{Idempotence of the Radon--Nikodym vertex core}

In this subsection, we prove idempotence of the operation of taking the Radon--Nikodym vertex core of an acyclic locally countable mcp Borel graph $T$.
As a consequence, when $T=T_f$ for some acyclic Borel function $f$, the topography of the Radon--Nikodym cocycle $\rho$ on $T_f$ can be reduced to two cases: 
(1) the case where all back $f$-orbits are $\rho$-finite, and (2) the case where all back $f$-orbits are $\rho$-infinite.

In fact, proving idempotence of the core of $T$ in general quickly reduces (via \cref{null_core}) to the special case where $T$ is generated by a function. 
This case in turn follows from a mass transport argument (\cref{nonrecurrent_mass_transport}) together with \cref{rho-infinite_back-orbits=>on_all_rho-infinite_ends}. 
While simple, \cref{rho-infinite_back-orbits=>on_all_rho-infinite_ends} is critical for our analysis of the behavior of the Radon--Nikodym cocycle along ends of a treeing induced by the function.

\begin{prop}\label{rho-infinite_back-orbits=>on_all_rho-infinite_ends}
Let $f : X \to X$ be an acyclic countable-to-one Borel function on a standard probability space $(X,\mu)$ such that $E_f$ is mcp, with associated Radon--Nikodym cocycle $\rho$. 
Let $Y \subseteq \rnc_f$ be a Borel set. 
Then, after discarding an $E_f$-invariant null set, for every $y\in Y$ we have:

\begin{enumerate}[(a)]
\item \label{rho-infinite_back-orbits=>on_all_rho-infinite_ends:forward_recurrence} $Y$ is forward $f$-recurrent in $[y]_{E_f}$,

\item \label{rho-infinite_back-orbits=>on_all_rho-infinite_ends:meets_rho-infinite_back-sets} Each $\rho$-infinite $T_f$-convex subset of $[y]_{E_f}$ meets $Y$ infinitely often,

\item \label{rho-infinite_back-orbits=>on_all_rho-infinite_ends:convex-hall=X}  $\rnc_f\cap [y]_{E_f}= \Conv _{T_f}(Y) \cap [y]_{E_f}$.
\end{enumerate}
\end{prop}
\begin{proof}
For \labelcref{rho-infinite_back-orbits=>on_all_rho-infinite_ends:forward_recurrence}, 
it is enough to show that the set $D$, of all points $z\in Y$ for which there is no $n\geq 1$ with $f^n(z)\in Y$, is null. 
For $z\in D$ its back $f$-orbit $f^{-\N}(z)$ is $\rho$-infinite (since $z\in \rnc_f$) and coincides with the preimage $r_{D,f}^{-1}(z)$ of $z$ under the retraction $r_{D,f}:f^{-\N}(D)\to D$ to $D$ along $f$.
But $r_{D,f}$ is almost surely $\rho$-finite-to-one by \cref{smooth=rho-finite}, hence $D$ must be null.

For \labelcref{rho-infinite_back-orbits=>on_all_rho-infinite_ends:meets_rho-infinite_back-sets}, after discarding a null set we may assume that $r_{Y,f}$ is $\rho$-finite-to-one and that $Y$ is $f$-forward recurrent in $[y]_{E_f}$. 
Thus, $r_{Y,f}(S)$ is infinite for each $\rho$-infinite subset $S$ of $[y]_{E_f}$. 
If $S$ is additionally $T_f$-convex, then $r_{Y,f}(S)$ can contain at most one point of $Y$ not in $S$, so $S \cap Y$ is infinite.

\labelcref{rho-infinite_back-orbits=>on_all_rho-infinite_ends:convex-hall=X} follows from \labelcref{rho-infinite_back-orbits=>on_all_rho-infinite_ends:forward_recurrence} and \labelcref{rho-infinite_back-orbits=>on_all_rho-infinite_ends:meets_rho-infinite_back-sets}.
Alternatively, it also follows from \cref{core_minimality,core=X_infty}.
\end{proof}

\begin{remark}
\cref{rho-infinite_back-orbits=>on_all_rho-infinite_ends}\labelcref{rho-infinite_back-orbits=>on_all_rho-infinite_ends:forward_recurrence} and \cref{nonrecurrent_mass_transport} may be viewed as reinterpretations of classical ergodic-theoretic arguments which make use of the Frobenius--Perron transfer operator, for example as presented by Aaronson in \cite[Proposition 1.3.1]{Aaronson:book}.
\end{remark}

\begin{lemma}\label{nonrecurrent_mass_transport}
Let $(X,\mu )$ be a standard probability space, and let $f : X \to X$ be an acyclic countable-to-one Borel function on $X$ such that $E_f$ is mcp on $(X,\mu )$, with associated Radon--Nikodym cocycle $\rho$.
Then for each $C > 0$ the set 
\[
X_C \defeq \set{x \in X : \rho ^x (f^{-\N}(x)) \le C} 
\]
is $\mu$-nowhere forward $f$-recurrent.
\end{lemma}

\begin{proof}
Define a mass transport where each $x \in X$ sends $1$ to $f^n(x)$ if $f^n(x) \in X_C$ and $n \ge 0$, and $x$ sends $0$ to all other points.
Then the mass sent out by each $x \in X$ is the cardinality of $\set{n \in \N : f^n(x) \in X_C}$, while the mass received by each $y \in X$ is $\rho^y(f^{-\N}(y))$ if $y \in X_C$, and $0$ otherwise.
By the mass transport principle, the expected mass sent out is equal to the expected mass received, which is bounded by $C < \infty$.
Therefore, for a.e.\ $x \in X$, there are only finitely many $n \in \N$ for which $f^n(x) \in X_C$.
\end{proof}

\begin{theorem}[Idempotence of the Radon--Nikodym vertex core]\label{Core_idempotent} %\label{X_inf}
Let $T$ be an acyclic locally countable mcp Borel graph on a standard probability space $(X,\mu )$, and assume that the Radon--Nikodym vertex core $\rnc_T$ of $T$ is non-null. Let $T_{\infty}$ be the restriction of $T$ to $\rnc_T$, and let $\mu_\infty$ be the normalized restriction of $\mu$ to $\rnc_T$.  
Then, after discarding a null set, the Radon--Nikodym vertex core of $T_{\infty}$ (with respect to $\mu_\infty$) coincides with $\rnc_T$.
\end{theorem}

\begin{proof}
Assume first that $T=T_f$ for some acyclic countable-to-one Borel function $f:X\to X$, and let $f_{\infty}:\rnc_f \to \rnc_f$ be the restriction of $f$ to $\rnc_f$.
By \cref{core=X_infty}, after discarding a null set, $\rnc_T = \rnc_f$ and $T_{\infty}=T_{f_{\infty}}$, so it is enough to show that for each $C>0$ the set $Y_C$, of all $x\in \rnc_f$ with $\rho ^x(f_{\infty}^{-\N}(x))\leq C$, is null. 
By \cref{nonrecurrent_mass_transport} (applied to $f_{\infty}$ and $\rnc_f$ equipped with the measure $\mu_\infty$), the set $Y_C$ is $\mu_\infty$-nowhere forward $f_\infty$-recurrent, hence also $\mu$-nowhere forward $f$-recurrent. 
Therefore, $Y_C$, being a subset of $\rnc_f$, is $\mu$-null by \cref{rho-infinite_back-orbits=>on_all_rho-infinite_ends}\labelcref{rho-infinite_back-orbits=>on_all_rho-infinite_ends:forward_recurrence} applied to $f$.

For the general case, let $Y$ be the Radon--Nikodym vertex core of $T_{\infty}$.
By \cref{core_minimality}, after discarding a null set, $\rnc_T \cap [Y]_{E_T}=Y$, so it is enough to show that the set $\rnc_T\setminus [Y]_{E_T}$ is null.
Supposing otherwise, we may assume without loss of generality that $Y$ is empty and that $\rnc_T$ is an $E_T$-complete section, so that each $T$-component is $\rho$-infinite.
By \cref{smooth=rho-finite}, $E_T$ is $\mu$-nowhere smooth, so $E_{T_\infty}$ is $\mu$-nowhere smooth, and hence
by \cref{null_core}, $T_\infty = T_{f_\infty}$ for some Borel function $f_\infty : \rnc_T \to \rnc_T$.
Since $\rnc_T$ is $T$-convex, this implies that $T = T_f$ for the Borel function $f : X \to X$ extending $f_\infty$ and defined on $X \setminus \rnc_T$ by sending each $x \in X \setminus \rnc_T$ to its successor along the unique geodesic path in $T$ from $x$ to the set $\rnc_T$.
Thus, after discarding a null set, $\rnc_T = Y = \emptyset$, contradicting that $\rnc_T$ has positive measure.
\end{proof}

\section{Characterizations of essential one-endedness}\label{subsec:coc-finite_back_orbits}

In this section, we give several equivalent characterizations of when an mcp locally countable acyclic graph $T$ is essentially one-ended.
In the case where $T=T_f$ is generated by an acyclic Borel function $f$, this gives a geometric cocycle-free characterization (\cref{char-rho-finite-back-orbits}) of all back $f$-orbits being cocycle-finite. 
The proof relies on \cref{lem:masstransport}, obtained via mass transport, in tandem with a Borel--Cantelli argument in \cref{ess_one_end_loc_fin} that reduces to the case where $f$ is finite-to-one.

\begin{lemma}\label{lem:masstransport}
Let $f : X \to X$ be an acyclic countable-to-one Borel function on a standard probability space $(X,\mu)$ such that $E_f$ is mcp, with associated Radon--Nikodym cocycle $\rho$.
Let $P:X\rightarrow [1,\infty ]$ be given by $P(x)\defeq \rho ^x (f^{-\N}(x))$. Then 
\[
\sum _{n\geq 0}\frac{1}{P(f^n(x))} <\infty
\]
for almost every $x\in X$ (where $\frac{1}{\infty}$ is interpreted as $0$). In particular, $\lim _{n\rightarrow\infty}P(f^n(x)) = \infty$ for almost every $x\in X$.
\end{lemma}

\begin{proof}
Define a mass transport function $h : E_f\rightarrow [0,\infty)$ by $h (x,y) = \frac{1}{P(y)}$ if $x\in f^{-\N}(y)$ and $P(y)<\infty$, and $h (x,y)=0$ otherwise. Then each $x\in X$ sends out mass $\sum _{y\in [x]_{E_f}}h (x,y)=\sum _{n\geq 0}\frac{1}{P(f^n(x))}$. The mass received by $y\in X$ is 
\[
\sum _{x\in [y]_{E_f}}h (x,y)\rho ^y(x)=\sum _{x\in f^{-\N}(y)}\frac{1}{P(y)}\rho ^y(x) = 1
\]
if $P(y)<\infty$, and is $0$ if $P(y)=\infty$. Therefore, by the mass transport principle (\cref{eq:mass_transport}),
\[
\int _X \sum _{n\geq 0}\frac{1}{P(f^n(x))} \, d\mu = \mu (\{ y: P(y) < \infty \} )\leq 1
\]
from which the conclusion follows.
\end{proof}

\begin{lemma}\label{ess_one_end_loc_fin}
Let $f : X \to X$ be an acyclic countable-to-one Borel function on a standard probability space $(X,\mu)$ such that $E_f$ is mcp.
If $f$ is one-ended, then after discarding a null set, there exists a forward $f$-invariant Borel $E_f$-complete section on which $f$ is finite-to-one.
\end{lemma}

\begin{proof}
Note that $f$ being one-ended is equivalent to the relation $<_f$ on $X$, defined by $x_0<_f x_1$ $\Leftrightarrow$ $f^n(x_0)=x_1$ for some $n\geq 1$, being well-founded.  %Define the forward $f$-invariant sets $X_{\alpha}$ by transfinite induction on ordinals as follows: define $X_0 \defeq X$, define $X_{\alpha + 1}$ to be the set of all vertices $x\in X_{\alpha}$ which have degree at least $2$ in $T|X_{\alpha}$ (i.e.\ non-leaves of $T|X_{\alpha}$), and define $X_{\lambda}=\bigcap _{\alpha < \lambda} X_{\alpha}$ for limit ordinals $\lambda$. (Equivalently, $X_\alpha$ is the set of points in $X$ with $<_f$-rank at least $\alpha$.) 

By the boundedness principle for well-founded analytic relations \cite[Theorem 31.1]{bible} the rank of $<_f$ is a countable ordinal. Note that the $<_f$-rank of each $E_f$-class is a limit ordinal, so by working separately on the set of $E_f$-classes of a given rank, we may assume without loss of generality that every $E_f$-class has $<_f$-rank exactly $\beta$ for some fixed countable limit ordinal $\beta$. For each $\alpha <\beta$ let $L_{\alpha}$ be the set of points in $X$ of $<_f$-rank exactly $\alpha$. 

Fix a sequence $(\e _{\alpha})_{\alpha <\beta}$ of positive reals with $\sum _{\alpha<\beta}\e _{\alpha}<\infty$. We will define a sequence $(Z_{\alpha})_{\alpha < \beta}$ of backward $f$-invariant Borel sets with $Z_{\alpha}\subseteq f^{-\N}(L_{\alpha})$ and $\mu (Z_{\alpha})<\e _{\alpha}$ such that the set $f^{-1}(x)\setminus Z_{\alpha}$ is finite for all $x\in L_{\alpha}$, $\alpha <\beta$.

Assume first that we have constructed sets $Z_{\alpha}$, $\alpha <\beta$, with these properties and we show that, after discarding a null set, the set $X_{\beta}\defeq X\setminus \bigcup _{\alpha <\beta}Z_{\alpha}$ satisfies the conclusion of the lemma. 
It is clear that $X_{\beta}$ is a forward $f$-invariant set on which $f$ is finite-to-one, so we only need to show that $X_{\beta}$ is an $E_f$-complete section.
By the Borel--Cantelli Lemma, after discarding a null set we may assume that each $x\in X$ belongs to only finitely many of the $Z_{\alpha}$. 
Then for each $x\in X$, since the sequence of $<_f$-ranks of $f^n(x)$, $n\geq 0$, is strictly increasing and cofinal in $\beta$, it follows that there is some $n\geq 0$ such that $f^n(x)$ does not belong to any of the sets $Z_{\alpha}$, and hence $f^n(x)\in X_{\beta}$. 
%Therefore, $X_{\beta}$ is an $E_f$-complete section. 

We now show how to define the sets $Z_{\alpha}$ for all $\alpha < \beta$.
By Luzin--Novikov uniformization there exists a
%countable partition $X=\bigsqcup_{n\in \N}A_n$ of $X$ into Borel sets on which $f$ is injective, and for each $n\in \N$ we let $g_n:f(A_n)\to A_n$ denote the inverse of $f\rest{A_n}$.
%Borel labeling $\ell : X\to \N$ such that for each $x\in X$ and $n\in \N$ there is at most one $y\in f^{-1}(x)$ with $\ell (y)=n$; we denote this $y$ by $g_n(x)$. 
Borel proper $\N$-coloring of the directed edges of $T_f$ (i.e., at each vertex, distinct incoming edges are assigned distinct colors), and for $x\in X$ we let $g_n(x)$ denote the unique $y\in f^{-1}(x)$ for which the directed edge from $y$ to $x$ is colored by $n$, if such an element exists. 
For a fixed $\alpha <\beta$, the sets $f^{-\N}(g_n(L_{\alpha}))$, $n\in \N$, are pairwise disjoint, so we can find a large enough $n_{\alpha}\in \N$ such that the backward $f$-invariant set $Z_{\alpha}\defeq \bigcup _{n\geq n_{\alpha}} f^{-\N}(g_n(L_{\alpha}))$ has measure less than $\e _{\alpha}$.
For each $x\in X$ the preimage $f^{-1}(x)$ is enumerated by $\{ g_n(x): x\in \mathrm{dom}(g_n)\}$, hence if $x\in L_{\alpha}$ then $f^{-1}(x)\setminus Z_{\alpha}$ is finite.
\end{proof}

\begin{theorem}\label{char-rho-finite-back-orbits}
Let $f : X \to X$ be an acyclic countable-to-one Borel function on a standard probability space $(X,\mu)$, such that $E_f$ is measure class preserving, with associated Radon--Nikodym cocycle $\rho$. 
The following are equivalent:
\begin{enumerate}[(1)]
    \item\label{item:back-orbits-rho-finite} Almost all back $f$-orbits are $\rho$-finite. 
    \item\label{item:forward_summable} There exists a strictly positive Borel function $Q : X \rightarrow \Rpos$ with $\sum _{n\geq 0} Q(f^n(x))<\infty$ for almost every $x\in X$.
    
    \item\label{item:one-ended} After discarding a null set there is a forward $f$-recurrent Borel $E_f$-complete section $Y\subseteq X$ for which the next return map $f_Y : Y\rightarrow Y$ is one-ended.
    
    \item\label{item:ess_one_ended} After discarding a null set there is a forward $f$-invariant Borel $E_f$-complete section $Z\subseteq X$ on which the restriction $f:Z\rightarrow Z$ is one-ended.
    
    \item\label{item:ess_one_ended_loc_fin} After discarding a null set there is a set $Z$ as in \labelcref{item:ess_one_ended} on which the restriction $f:Z\rightarrow Z$ is moreover finite-to-one.
\end{enumerate}
\end{theorem}

\begin{proof}
The implications \labelcref{item:ess_one_ended_loc_fin}$\Rightarrow$\labelcref{item:ess_one_ended}$\Rightarrow$\labelcref{item:one-ended} are trivial. We will prove \labelcref{item:back-orbits-rho-finite}$\Rightarrow$\labelcref{item:forward_summable}$\Rightarrow$\labelcref{item:one-ended}$\Rightarrow$\labelcref{item:ess_one_ended_loc_fin} and \labelcref{item:one-ended}$\Rightarrow$\labelcref{item:back-orbits-rho-finite}.

\labelcref{item:back-orbits-rho-finite}$\Rightarrow$\labelcref{item:forward_summable}: If \labelcref{item:back-orbits-rho-finite} holds then the function $Q(x)\defeq 1/\rho ^x(f^{-\N}(x))$ is strictly positive almost everywhere and $\sum _{n\geq 0}Q(f^n(x))<\infty$ almost surely by
\cref{lem:masstransport}, so \labelcref{item:forward_summable} holds.

\labelcref{item:forward_summable}$\Rightarrow$\labelcref{item:one-ended}: Assume that \labelcref{item:forward_summable} holds. We may assume that the sum $S(x)\defeq\sum _{n\geq 0}Q(f^n(x))$ is finite for all $x\in X$.  Take $Y$ to be the set of all $y\in X$ satisfying $S(y)\leq 1$ and $Q(y)>Q(f^n(y))$ for all $n\geq 1$. 
The set $Y$ is a forward $f$-recurrent $E_f$-complete section since $\lim _{n\rightarrow\infty}S(f^n(x))= 0$ and $\lim _{n\rightarrow\infty} Q(f^n(x))=0$ for all $x\in X$. 
It remains to show that the next return map $f_Y:Y\rightarrow Y$ is one-ended. Suppose otherwise, so that there exists a sequence $y_0,y_1,y_2,\dots$ of distinct points in $Y$ with $y_{m+1} \in f^{-\N}(y_m)$ for all $m$. On the one hand we have $S(y_m)\leq 1$ for all $m$, but on the other hand by definition of $Y$ we have $Q(y_0)<Q(y_1)<\cdots$ and hence 
\[
S(y_m)\geq \sum _{i=0}^m Q(y_i)> (m+1)Q(y_0)\rightarrow\infty \; \text{ as $m \to \infty$},
\]
a contradiction, so \labelcref{item:one-ended} holds.

\labelcref{item:one-ended}$\Rightarrow$\labelcref{item:ess_one_ended_loc_fin}: Assume that \labelcref{item:one-ended} holds. 
By applying \cref{ess_one_end_loc_fin} to $f_Y:Y\rightarrow Y$ and the normalized restriction of $\mu$ to $Y$ we obtain, after discarding a null set, a forward $f$-recurrent Borel $E_f$-complete section $Y_0\subseteq Y$ for which the next return map $f_{Y_0}:Y_0\rightarrow Y_0$ is one-ended and finite-to-one, and hence has finite back-orbits. 
Taking $Z$ to be the convex hull of $Y_0$ in $T_f$ shows that \labelcref{item:ess_one_ended_loc_fin} holds.

\labelcref{item:one-ended}$\Rightarrow$\labelcref{item:back-orbits-rho-finite}: Assume again that \labelcref{item:one-ended} holds.
Since $Y$ is forward $f$-recurrent, to show \labelcref{item:back-orbits-rho-finite} holds it is enough to show that the set $Y_{\infty} \defeq \rnc_f \cap Y$ is null. 
Let $L_{\infty}$ consist of all points in $Y_{\infty}$ whose back $f$-orbit contains no other point of $Y_{\infty}$. Since $f_Y$ is one-ended, the back $f$-orbit of each point of $Y_{\infty}$ meets $L_{\infty}$, so it is enough to show that $L_{\infty}$ is null. This follows from \cref{rho-infinite_back-orbits=>on_all_rho-infinite_ends}\labelcref{rho-infinite_back-orbits=>on_all_rho-infinite_ends:forward_recurrence}, since $L_{\infty}$ is nowhere forward $f$-recurrent.
\end{proof}

%One nice thing about condition \labelcref{item:one-ended} is that it makes no reference to the cocycle. 

\begin{remark}\label{remark:essential_one-ended_characterization}
Each of the conditions in \cref{char-rho-finite-back-orbits} other than the first has a purely Borel counterpart. The conditions \labelcref{item:forward_summable} and \labelcref{item:one-ended} are equivalent in the purely Borel setting. Indeed, the above proof already shows that the implication \labelcref{item:forward_summable}$\Rightarrow$\labelcref{item:one-ended} holds in the pure Borel setting. For the reverse implication, assuming \labelcref{item:one-ended}, fix a summable sequence $(\e_\alpha)_{\alpha < \beta}$, where $\beta$ is the rank of the analytic relation $<_{f_Y}$, as defined in the proof of \cref{ess_one_end_loc_fin}. 
Since the equivalence relation on $X$ of having the same image under $f_Y$ is smooth, we can easily define a Borel function $Q : X \to \Rpos$ satisfying $\sum_{x \in f_Y^{-1}(y)} Q(x) \le \e_\alpha$ for every $y \in Y$ of $<_{f_Y}$-rank $\alpha$, for all $\alpha < \beta$. This implies \labelcref{item:forward_summable}.

However, the other nontrivial implications \labelcref{item:one-ended}$\Rightarrow$\labelcref{item:ess_one_ended} and \labelcref{item:ess_one_ended}$\Rightarrow$\labelcref{item:ess_one_ended_loc_fin} do not hold even in the Baire category setting, i.e.\ after discarding a meager set with respect to a given Polish topology. 
Indeed, one may show via localization arguments that the Baire category analogues of implications \labelcref{item:one-ended}$\Rightarrow$\labelcref{item:ess_one_ended} and \labelcref{item:ess_one_ended}$\Rightarrow$\labelcref{item:ess_one_ended_loc_fin} fail, respectively, for the shift map on the space of all strictly decreasing sequences of rationals in $(0, 1)$, and the shift map on the space of all increasing sequences of ordinals in $\omega^2$.
\end{remark}

\begin{cor}\label{ess_one_ended_trees}
Let $T$ be an acyclic locally countable mcp graph on a standard probability space $(X,\mu)$, with the associated Radon--Nikodym cocycle $\rho$.
Suppose that $E_T$ is $\mu$-nowhere smooth.
Then $T$ is essentially one-ended on an $E_T$-conull set if and only if a.e.\ $T$-component has exactly one $\rho$-barytropic end.

Moreover, if the above equivalent conditions hold then, after discarding a null set, there is $T$-convex Borel complete section $Y \subseteq X$ on which $T$ is one-ended and locally finite.
\end{cor}
\begin{proof}
Both conditions of the equivalence imply, after discarding an $E_T$-invariant null set, that $T = T_f$ for the Borel function $f : X \to X$ that maps each $x \in X$ one step closer towards the distinguished end. The statement then follows from \cref{char-rho-finite-back-orbits}.
\end{proof}

\begin{cor}[Radon--Nikodym vertex core dichotomy]
    Let $f : X \to X$ be an acyclic countable-to-one Borel function on a standard probability space $(X,\mu)$ such that $E_f$ is mcp, with associated Radon--Nikodym cocycle $\rho : E_f \to \Rpos$. 
    Suppose $E_f$ is $\mu$-ergodic.
    \begin{enumerate}[(a)]
        \item\label{part:RN-core_complete_section} If $\rnc_f$ is an a.e.\ $E_f$-complete section, then $\rnc_f$ is a.e.-contained in $Z$ for every $T_f$-convex Borel $E_f$-complete section $Z$.

        \item\label{part:RN-core_empty} If $\rnc_f$ is null, then there exists a vanishing sequence of forward $f$-invariant Borel $E_f$-complete sections on which $f$ is finite-to-one.
    \end{enumerate}
\end{cor}

\begin{proof}
    Part \labelcref{part:RN-core_complete_section} follows immediately from
    \cref{rho-infinite_back-orbits=>on_all_rho-infinite_ends}\labelcref{rho-infinite_back-orbits=>on_all_rho-infinite_ends:convex-hall=X}, while part \labelcref{part:RN-core_empty} is derived from the implication \labelcref{item:back-orbits-rho-finite}$\Rightarrow$\labelcref{item:ess_one_ended_loc_fin} of \cref{char-rho-finite-back-orbits} via successive prunings.
\end{proof}

\section{Cocycle-finiteness of back geodesics and vanishing back ends}

Throughout this section we work under the following hypotheses:

\begin{workinghyp}\label{running_hypothesis_function}
$f : X \to X$ is an acyclic countable-to-one Borel function on a standard probability space $(X,\mu)$, with $E_f$ mcp and having associated Radon--Nikodym cocycle $\rho : E_f \to \Rpos$.
\end{workinghyp}
This section analyzes the back ends of $f$. We are mainly concerned with the case where $f$ is \emph{$\mu$-nowhere essentially two-ended}, by which we mean that there is no positive measure $T_f$-convex Borel set on which $T_f$ is two-ended.

\subsection{Cocycle-finiteness of back-geodesics}

Given a function $f$ as in \cref{running_hypothesis_function} and a point $x\in X$, we define the \dfn{backward geodesic weight of $f$ at $x$} to be the quantity
\[
\Sigma_f(x)\defeq \sup \{ \rho ^x([x,v]_{T_f}) : v\in f^{-\N}(x) \} .
\]
This subsection is concerned with finiteness of the function $\Sigma_f$. Given a positive integer $n$ we say that $f$ is \emph{at least $n$-to-$1$} if $|f^{-1}(x)|\geq n$ for all $x\in X$.

\begin{lemma}\label{5-to-1}
Let $f$ be as in \cref{running_hypothesis_function}. Suppose that $f$ is at least $5$-to-$1$. Then the set of all $x\in X$ for which $\Sigma_f(x)<\infty$ is an a.e.\ $E_f$-complete section.
\end{lemma}
\begin{proof}
Suppose otherwise, and without loss of generality assume that $\Sigma_f(x)=\infty$ for all $x\in X$. Then for every point $x \in X$, there is a point $v_0 \defeq v_0(x) \in f^{-\N}(x)\setminus \{ x \}$ such that $\rho^x((x,v_0]_{T_f}) \ge 1$ while $\rho^x((x,v_0)_{T_f}) < 1$. Moreover, there is another such point $v_1 \defeq v_1(x)$ such that $x$ lies on the geodesic $V_x \defeq [v_0,v_1]_{T_f}$ from $v_0$ to $v_1$ (for example, start with any $y\in f^{-1}(x)\setminus (x,v_0]_{T_f}$ and use that $\Sigma_f(y)=\infty$).
We refer to the sets $(x, v_0]_{T_f}$ and $(x, v_1]_{T_f}$ as the \textit{horns} of $V_x$, see \cref{fig:V}.

\begin{figure}[htbp]
\centering

% Define colors to match the hand-drawn markers
\colorlet{myblue}{blue!90!cyan}
\definecolor{myred}{RGB}{210, 0, 0}
\definecolor{mygreen}{RGB}{0, 110, 60}

\begin{tikzpicture}[
    scale=0.6, % Reduced scale
    bnode/.style={circle, fill=myblue, inner sep=0pt, minimum size=3.5pt}, % Smaller nodes
    rnode/.style={circle, fill=myred, inner sep=0pt, minimum size=3.5pt},
    knode/.style={circle, fill=black, inner sep=0pt, minimum size=3.5pt},
    lbl/.style={text=orange, font=\footnotesize\sffamily}, % Adjusted font size
    txt/.style={text=black, font=\normalsize},
    thick line/.style={line width=1pt}, % Adjusted line thickness
    % Arrow style scaled down
    arr/.style={postaction={decorate,decoration={
        markings,
        mark=at position 0.55 with {\arrow{Straight Barb[length=1.5mm, width=1.5mm]}}
    }}}
]

% -------------------------
% Coordinates Definition
% -------------------------
\coordinate (y) at (0,0);
\coordinate (fy) at (0,-1.5);

% Left branch coordinates
\foreach \i in {1,...,7} \coordinate (L\i) at (-0.7*\i, 1.3*\i);

% Right branch coordinates
\foreach \i in {1,...,5} \coordinate (R\i) at (0.7*\i, 1.3*\i);

% -------------------------
% Edges
% -------------------------
% Draw edges first so they sit nicely under the node circles

% Black center edges
\draw[thick line, arr] (y) -- (fy);
\draw[thick line, arr] (L1) -- (y);
\draw[thick line, arr] (R1) -- (y);

% Blue left edges
\draw[thick line, myblue, arr] (L7) -- (L6);
\draw[thick line, myblue, arr] (L6) -- (L5);
\draw[thick line, myblue, arr] (L5) -- (L4);
\draw[thick line, myblue, arr] (L4) -- (L3);
\draw[thick line, myblue, arr] (L3) -- (L2);
\draw[thick line, myblue, arr] (L2) -- (L1);

% Red right edges
\draw[thick line, myred, arr] (R5) -- (R4);
\draw[thick line, myred, arr] (R4) -- (R3);
\draw[thick line, myred, arr] (R3) -- (R2);
\draw[thick line, myred, arr] (R2) -- (R1);

% -------------------------
% Nodes (Circles)
% -------------------------
\node[knode] at (y) {};
\node[knode] at (fy) {};
\foreach \i in {1,...,7} \node[bnode] at (L\i) {};
\foreach \i in {1,...,5} \node[rnode] at (R\i) {};

% -------------------------
% Labels
% -------------------------
% Left Branch (Green numbers)
\node[text=myblue, right=1pt of L7] {$v_0$};
\node[lbl, left=1pt of L7] {1.2};
\node[lbl, left=1pt of L6] {0.01};
\node[lbl, left=1pt of L5] {0.11};
\node[lbl, left=1pt of L4] {0.7};
\node[lbl, left=1pt of L3] {0.03};
\node[lbl, left=1pt of L2] {0.1};
\node[lbl, left=1pt of L1] {0.02};

% Right Branch (Green numbers)
\node[text=myred, left=0pt of R5] {$v_1$};
\node[lbl, right=1pt of R5] {0.15};
\node[lbl, right=1pt of R4] {0.01};
\node[lbl, right=1pt of R3] {0.4};
\node[lbl, right=1pt of R2] {0.2};
\node[lbl, right=1pt of R1] {0.3};

% Center Nodes (y and f(y))
\node[lbl, left=1pt of y] {1};
\node[txt, right=1pt of y] {$x$};
\node[txt, right=1pt of fy] {$f(x)$};

% -------------------------
% Global Flow Arrow (Dashed)
% -------------------------
\draw[dashed, line width=1pt, -{Straight Barb[length=2.5mm, width=2.5mm]}]
    (4.5, 3.5) -- (4.5, 1) node[midway, right=0pt, font=\normalsize] {$f$};

\end{tikzpicture}

\caption{The set $V_x$ with values of $\rho^x$ in orange and horns of $V_x$ in blue and red.}
\label{fig:V}
\end{figure}

By Luzin--Novikov uniformization, we may assume that the maps $v_0$ and $v_1$ are Borel, and hence that the map $x \mapsto V_x$ is also Borel. Let $E_V$ be the equivalence relation generated by the binary relation $\{ (x,y) : y \in V_x \}$.

\begin{claim*}
Each $E_V$-class is $\rho$-infinite, and hence $E_V$ is $\mu$-nowhere smooth.
\end{claim*}

\begin{pf}
The idea is simple: for each $x \in X$, we may assign to each $y \in [x]_{E_V}$ one horn $h_y$ of $V_y$ in such a way that the horns assigned to distinct points are disjoint.
This yields the claim by iterating the inequality $\rho^x(h_y) \ge \rho^x(y)$.

Formally, after fixing $x \in X$, we recursively define a sequence $(I_x^n)_{n \ge 0}$ of pairwise disjoint subsets of $[x]_{E_V} \cap f^{-\N}(x)$, beginning with $I_x^0=\{ x \}$, where at each stage $n\geq 1$ we will define $I_x^n$ to be the union, $I_x^{n} \defeq \bigcup_{y \in I_x^{n-1}} h_y$, of $T_f$-paths $h_y$, for $y \in I_x^{n-1}$, with the following  (see \cref{fig:disjoint_horns}):

\begin{enumerate}[(i.$n$)]
\item $h_y = (y,v_i(y)]_{T_f}$ for some $i \in \set{0,1}$, for every $y \in I_x^{n-1}$.

\item %$I_x^n = \bigsqcup_{y \in I_x^{n-1}} h_y$, and moreover, 
for distinct $y_0,y_1 \in I_x^{n-1}$ the sets $f^{-\N}(h_{y_0})$ and $f^{-\N}(h_{y_1})$ are disjoint. 

\item $f^{-\N}(I_x^n)$ is disjoint from $\bigcup_{m < n} I_x^m$.

\item $\rho^x(I_x^n) \geq 1$.
\end{enumerate}

\begin{figure}[htbp]
\centering

\begin{tikzpicture}[
scale=0.9,
% Reduced arrow size slightly so they fit on the shorter, distant edges
mid arrow/.style={
postaction={
decorate,
decoration={
markings,
mark=at position 0.65 with {\arrow{Straight Barb[length=1.5mm, width=2mm]}}
}
}
},
node/.style={circle, inner sep=1.6pt},
blue node/.style={node, fill=blue!90!cyan},
red node/.style={node, fill=red!90!black},
black node/.style={node, fill=black},
blue edge/.style={blue!90!cyan, thick, mid arrow},
red edge/.style={red!90!black, thick, mid arrow},
black edge/.style={black, thick, mid arrow},
black stick/.style={black, thick},
orange edge/.style={orange!90!black, thick}
]

% ==========================================
% PARAMETERS TO TWEAK THE FRACTAL SHAPE
% ==========================================
\def\Lzero{2.4}  % Starting edge length at distance 0 (Increased from 1.6 for longer edges)
\def\alpha{0.83} % Length decay: 1.0 = same length, 0.5 = shrinks twice as fast
\def\beta{0.82}  % Angle decay: 1.0 = static angle, 0.5 = bends upward rapidly
\def\Azero{65}   % Starting angle offset (Increased from 45 for a wider spread)

% ==========================================
% 1. BUILD DYNAMIC COORDINATES
% ==========================================
\coordinate (P-0-0) at (0, 0);

% Build Red Spine (v direction: branches left)
\foreach \v in {1,...,5} {
\pgfmathtruncatemacro{\prevv}{\v-1}
% d = \prevv (graph distance of the starting node)
\pgfmathsetmacro{\len}{\Lzero * pow(\alpha, \prevv)}
\pgfmathsetmacro{\ang}{90 + \Azero * pow(\beta, \prevv)} 
\path (P-0-\prevv) ++(\ang : \len) coordinate (P-0-\v);
}

% Build Blue Branches (u direction: branches right off the red spine)
\foreach \v in {1,...,5} {
\foreach \u in {1,...,6} {
\pgfmathtruncatemacro{\prevu}{\u-1}
\pgfmathtruncatemacro{\d}{\prevu + \v} % True graph distance!
\pgfmathsetmacro{\len}{\Lzero * pow(\alpha, \d)}
\pgfmathsetmacro{\ang}{90 - \Azero * pow(\beta, \d)}
\path (P-\prevu-\v) ++(\ang : \len) coordinate (P-\u-\v);
}
}

% ==========================================
% 2. DRAW NODES AND EDGES
% ==========================================

% The first black node "x" at the bottom
\node[black node, label={below:\textsf{x}}] (N-0-0) at (P-0-0) {};

% Red nodes and edges along the bottom-left boundary
\foreach \v in {1,...,5} {
\node[red node] (N-0-\v) at (P-0-\v) {};
}
\foreach \v in {2,...,5} {
\pgfmathtruncatemacro{\prevv}{\v-1}
\draw[red edge] (N-0-\v) -- (N-0-\prevv);
}
\draw[black edge] (N-0-1) -- (N-0-0);

% Macro to automate drawing a blue layer and its incoming edges perfectly
\newcommand{\drawbluelayer}[2]{
% #1 = y-level (\v), #2 = max nodes (\u)
\foreach \u in {1,...,#2} {
\node[blue node] (N-\u-#1) at (P-\u-#1) {};

% Re-calculate the incoming angle for the stick/orange edge based on distance
\pgfmathtruncatemacro{\d}{\u + #1}
\pgfmathsetmacro{\len}{\Lzero * pow(\alpha, \d)}
\pgfmathsetmacro{\ang}{90 + \Azero * pow(\beta, \d)} % Angles get smaller!

% Dynamically place the incoming edges along the curving trajectory
\path (P-\u-#1) ++(\ang : {0.25 * \len}) coordinate (S-\u-#1);
\path (P-\u-#1) ++(\ang : {0.75 * \len}) coordinate (O-\u-#1);

\draw[black stick] (S-\u-#1) -- (N-\u-#1);
\draw[orange edge] (O-\u-#1) -- (S-\u-#1);
}

% Draw horizontal blue edges for this layer
\ifnum#2>1
\foreach \u in {2,...,#2} {
\pgfmathtruncatemacro{\prevu}{\u-1}
\draw[blue edge] (N-\u-#1) -- (N-\prevu-#1);
}
\fi

% The black edge connects back dynamically
\draw[black edge] (N-1-#1) -- (N-0-#1);
}

% Draw grids using the exact sizes from your original code
\drawbluelayer{1}{6} % nodes y=1
\drawbluelayer{2}{3} % nodes y=2
\drawbluelayer{3}{4} % nodes y=3
\drawbluelayer{4}{4} % nodes y=4
\drawbluelayer{5}{2} % nodes y=5

\begin{comment}
\foreach \u in {0,...,3} {
\foreach \v in {0,...,4} {
\node[blue node] (N-\u-\v) at (P-\u-\v) {};
}
}
% ... rest of commented edges ...
\end{comment}

\end{tikzpicture}

\caption{The sets $I_x^1$ in red, $I_x^2$ in blue, and $I_x^3$ in orange.}
\label{fig:disjoint_horns}
\end{figure}

Given this construction, it follows that $\bigcup_{n \ge 0} I_x^n$ is $\rho$-infinite, and thus so is $[x]_{E_V}$.

Turning to the construction, define $I_x^0 \defeq \set{x}$ and $h_x \defeq (x,v_0(x)]_{T_f}$. Then $I_x^1=h_x$, and (i.1)--(iv.1) are clearly satisfied. Let $n \ge 1$ and assume $I_x^n = \bigcup_{y \in I_x^{n-1}} h_y$ is defined and satisfies (i.$n$)--(iv.$n$). Then each $z \in I_x^n$ belongs to $h_y$ for a unique $y \in I_x^{n-1}$ and our choice of $v_0(z)$ and $v_1(z)$ ensures that at least one of the paths $(z,v_0(z)]_{T_f}$ and $(z,v_1(z)]_{T_f}$ is disjoint from $h_y$; we choose such a path and define it to be our $h_z$, thereby fulfilling (i.$n+1$). Let $I_x^{n+1} \defeq \bigcup_{z \in I_x^n} h_z$.

To show (ii.$n+1$), let $z_0,z_1 \in I_x^n$ and suppose that $f^{-\N}(h_{z_0})$ meets $f^{-\N}(h_{z_1})$. 
Because of (ii.$n$), the points $z_0$ and $z_1$ must both lie on $h_y$ for some $y \in I_x^{n-1}$. 
But the sets $f^{-\N}(h_z)$, $z \in h_y$, are pairwise disjoint by their construction, so $z_0 = z_1$.

By (i.$n+1$) and (ii.$n+1$), the set $f^{-\N}(I_x^{n+1})$ is disjoint from $I_x^n$ and is contained in $f^{-\N}(I_x^n)$, so together with (iii.$n$) this implies (iii.$n+1$).

As for (iv.$n+1$), for each $z\in I_x^n$ our choice of $v_i(z)$ ($i=0,1$) ensures that $\rho^x(h_z) \ge \rho^x(z)$, hence $\rho^x(I_x^{n+1}) = \sum_{z \in I_x^n} \rho^x(h_z) \ge \rho^x(I_x^n) \ge 1$ by (iv.$n$).
\end{pf}

Since each $E_V$-class is $T_f$-convex, $E_V$ is smooth on the union of all $E_V$-classes that are not forward $f$-invariant, hence this union is null by the claim. Thus, after discarding an $E_f$-invariant null set, each $E_V$-class is forward $f$-invariant, which implies $E_V = E_f$.

For each $x \in X$, let $\partial V_x \defeq \{ v_0(x), v_1(x)\}$ and $V_x' \defeq (x,v_0(x))_{T_f} \cup (x,v_1(x))_{T_f}$. We define a mass transport $\phi : E_f \to [0, 1]$ by 
\[
\phi(x,y) \defeq 
\begin{cases}
\rho^x(y) & \text{if } y \in V_x'
\\
1 - \rho^x \big((x, y)_{T_f}\big) & \text{if } y \in \partial V_x
\\
0 & \text{otherwise},
\end{cases}
\]
where the definitions of $v_0(x)$ and $v_1(x)$ ensure that the value $1 - \rho^x \big((x, y)_{T_f}\big)$ in the second case is positive.
Each $x \in X$ sends out mass exactly $\sum_{y \in [x]_{E_f}} \phi(x,y) = 2$. 
For each $y \in X$, letting $W_y \defeq \set{x \in X : y \in V_x'}$, the relative mass received by $y$ is 
\begin{equation}\label{eq:mass-in}
\sum_{x \in [y]_{E_f}} \phi(x,y) \rho^y(x) \ge \sum_{x \in W_y} \rho^x(y) \rho^y(x) = |W_y|.
\end{equation}
For each $x\in W_y$, the set $V_x$ meets $f^{-1}(y)$ in exactly one point which we denote by $z_x$. Moreover, the image of $W_y$ under the map $x\mapsto z_x$ contains $f^{-1}(y)\setminus V_y$, see \cref{fig:5-2=3}. 
Indeed, since $[y]_{E_V} = [y]_{E_f} \supseteq f^{-1}(y)$, for each $z \in f^{-1}(y)$ there must be some point $x$ with $y\in V_x-\partial V_x$ and $z\in V_x$, so if additionally $z \not\in V_y$ then $y \neq x$ and hence $x \in W_y$ and $z = z_x$. 

\begin{figure}[htbp]
\centering

% Define custom colors to match the image
\definecolor{myred}{RGB}{210,0,0}
\definecolor{myblue}{RGB}{0,127,255}
\definecolor{myorange}{RGB}{255,153,0}
\definecolor{mygreen}{RGB}{0,102,51}

\begin{tikzpicture}[
    scale=0.7, % Scaled down from 1.0
    % Style for all vertices
    v/.style={circle, fill=black, inner sep=2pt}, % Slightly smaller nodes
    % Style for directed edges (very thick with an arrow past the midpoint)
    ->-/.style={
        very thick, % Scaled down from ultra thick
        postaction={
            decorate,
            decoration={
                markings, 
                mark=at position 0.65 with {\arrow{Straight Barb[length=2.5mm,width=2.5mm]}} % Smaller arrows
            }
        }
    }
]

% 1. Define incoming path coordinates (starts of the arrows)
\coordinate (topB) at (-2.5, 3.2);
\coordinate (topR) at (-2.8, 1.5);
\coordinate (midB) at (-3.2, 0);
\coordinate (botR) at (-2.8, -1.5);
\coordinate (botB) at (-2.5, -3.2);

% 2. Place all the nodes (vertices)
% Left side nodes (incoming)
\node[v, fill=mygreen, label={[mygreen, label distance=-4pt]right:$z_w$}] (zw) at (-1.25, 1.6) {};
\node[v, fill=myred] (r1) at (-1.4, 0.75) {};
\node[v, fill=myblue, label={[myblue, label distance=-4pt]above:$z_u$}] (zu) at (-1.6, 0) {};
\node[v, fill=myred] (r2) at (-1.4, -0.75) {};
\node[v, fill=myorange, label={[myorange, label distance=-3pt]right:$z_v$}] (zv) at (-1.25, -1.6) {};

% Center node
\node[v, fill=myred, label={[myred, label distance=-2pt]below:$y$}] (x) at (0,0) {};

% Right side nodes (horizontal sequence)
\node[v] (n1) at (1.5,0) {};
\node[v, fill=myblue, label={[myblue, label distance=-2pt]below:$u$}] (u) at (3.0,0) {};
\node[v, fill=myorange, label={[myorange, label distance=-2pt]below:$v$}] (v) at (4.5,0) {};
\node[v] (n2) at (6.0,0) {};
\node[v] (n3) at (7.5,0) {};
\node[v, fill=mygreen, label={[mygreen, label distance=-2pt]below:$w$}] (w) at (9.0,0) {};

% 3. Draw the edges
% Left side incoming black edges
\draw[->-] (topB) -- (zw);
\draw[->-] (zw) -- (x);
\draw[->-] (midB) -- (zu);
\draw[->-] (zu) -- (x);
\draw[->-] (botB) -- (zv);
\draw[->-] (zv) -- (x);

% Left side incoming red edges
\draw[->-, myred] (topR) -- (r1);
\draw[->-, myred] (r1) -- (x);
\draw[->-, myred] (botR) -- (r2);
\draw[->-, myred] (r2) -- (x);

% Right side outgoing horizontal edges
\draw[->-] (x) -- (n1);
\draw[->-] (n1) -- (u);
\draw[->-] (u) -- (v);
\draw[->-] (v) -- (n2);
\draw[->-] (n2) -- (n3);
\draw[->-] (n3) -- (w);

% 4. Add the Vx label
\node[myred] at (-3.3, 1.5) {$V_y$};

\end{tikzpicture}

\caption{$f^{-1}(y) \setminus V_y$ contains at least three points $z_u, z_v, z_w$ corresponding to some points $u,v,w \in W_y$.}
\label{fig:5-2=3}
\end{figure}

It follows that $|W_y| \ge |f^{-1}(y)\setminus V_y| \ge 5 - 2 = 3$. 
By \labelcref{eq:mass-in}, this contradicts the mass transport principle (\cref{eq:mass_transport}).
\end{proof}

\begin{theorem}[Cocycle-finiteness of back-geodesics]\label{summable_back-rays}
Let $f$ be as in \cref{running_hypothesis_function}, and assume in addition that $f$ is $\mu$-nowhere essentially two-ended. Then $\Sigma_f(x) <\infty$ for a.e.\ $x\in X$. 
In particular, for a.e.\ $x\in X$, every back $f$-end of $T_f \rest{[x]_{T_f}}$ has $\rho$-finite geodesics.
\end{theorem}
\begin{proof}
Assume toward a contradiction that the set $C$ of all $x\in X$ with $\Sigma_f(x)=\infty$ has positive measure. 
We may assume without loss of generality that $C$ is an $E_f$-complete section. 
Since $C$ is a forward $f$-invariant Borel $E_f$-complete section contained in the directed Radon--Nikodym vertex core $\rnc_f$, after discarding a null set we may assume that $C=\rnc_f$ by \cref{rho-infinite_back-orbits=>on_all_rho-infinite_ends}\labelcref{rho-infinite_back-orbits=>on_all_rho-infinite_ends:convex-hall=X}. 
Considering $E_f \rest{C}$ and $C$ with the push-forward of $\mu$ under the retraction $r_{C,f}$ to $C$ along $f$, so that the associated Radon--Nikodym cocycle is given by $(y,z) \mapsto \rho^z(r_{C,f}^{-1}(y)) / \rho^z(r_{C,f}^{-1}(z))$, it follows that $\Sigma_{f\rest{C}}(x)$ is still infinite and the $f \rest{C}$-back orbit of $x$ is cocycle-infinite for each $x\in C$. 
Therefore, after replacing $X$ by $\rnc_f$, we may as well assume that $X=\rnc_f$ and that $\Sigma_f(x) = \infty$ for all $x \in X$.
Then, by \cref{Core-bi-inf-or-perfect}, since $f$ is $\mu$-nowhere essentially two-ended, after discarding another null set we may assume that each component of $T_f$ is a leafless tree with no isolated ends.

Assume first that we can find a forward $f$-recurrent Borel complete section $Y$ for $E_f$ such that the next return map $f_Y$ is at least $5$-to-$1$, and we show how to complete the proof. 
Let $\mu_Y$ be the push-forward of $\mu$ to $Y$ under the retraction $r_{Y,f}$ to $Y$ along $f$. 
Let $\rho_Y$ be the Radon--Nikodym cocycle of $E_{f_Y}$ with respect to $\mu_Y$, namely, $\rho_Y(y,z) = \rho^z(r_{Y,f}^{-1}(y)) / \rho^z(r_{Y,f}^{-1}(z))$ for all $(y,z) \in E_{f_Y}$. Then the map $f_Y$, considered on the probability space $(Y, \mu_Y)$, satisfies that every point $y \in Y$ has $\Sigma_{f_Y}(y) = \infty$, contradicting \cref{5-to-1}.

To find such a set $Y$, first let $Z$ be the set of points of $X$ with $T_f$-degree at least $3$ and observe that, since each component of $T_f$ is leafless with no isolated ends, each end of $T_f$ is an accumulation point of $Z$.
In particular, the next return map $f_Z$ is at least $2$-to-$1$. 
Let $T_Z$ be the graph generated by $f_Z$ and let $T_Z^{\le 2}$ be the graph on $Z$ where two vertices are adjacent if they are distinct and their $T_Z$-distance is at most $2$. Finally, let $Y$ be a Borel maximal $T_Z^{\le 2}$-independent subset of $Z$.
By \cref{rho-infinite_back-orbits=>on_all_rho-infinite_ends}\labelcref{rho-infinite_back-orbits=>on_all_rho-infinite_ends:forward_recurrence}, after discarding a null set $Y$ is forward $f$-recurrent, and it is easy to see that $f_Y$ is at least $8$-to-$1$.
\end{proof}

\subsection{Cocycle-vanishing of back-ends}

Let $f$ be as in \cref{running_hypothesis_function}, and assume that $f$ is $\mu$-nowhere essentially two-ended. While the previous section describes the behavior of the cocycle along back-geodesics, the goal of this subsection is to understand the limiting behavior of the cocycle at back-ends.
\cref{summable_back-rays} in particular implies that the cocycle converges to $0$ along every back-geodesic, and we show in this section that it in fact vanishes along \emph{any} sequence of vertices converging to a back-end. 

\begin{lemma}\label{having_bigger_points_behind}
Let $f$ be as in \cref{running_hypothesis_function}. Assume that for every $x \in X$ there is some $z \in f^{-\N}(x) \setminus \set{x}$ with $\rho^\bullet (z) \geq \rho^\bullet (x)$. Then there is a $\mu$-conull $E_f$-invariant Borel set on which $f$ is bijective.
\end{lemma}

\begin{proof}
The hypothesis implies that $X=\rnc_f$ and moreover $\Sigma_f(x)=\infty$ for all $x\in X$. Thus, by \cref{summable_back-rays} and measure-theoretic exhaustion we see that $f$ is $\mu$-essentially two-ended. 
But since $X=\rnc_f$, \cref{rho-infinite_back-orbits=>on_all_rho-infinite_ends}\labelcref{rho-infinite_back-orbits=>on_all_rho-infinite_ends:convex-hall=X} implies that $f$ is bijective a.e.
\end{proof}

\begin{notation}
For metric spaces $Y,Z$, a set $D \subseteq Y$, points $y_\w \in Y$, $z_\w \in Z$, and a function $F : Y \to Z$, we write $\lim_{y \to_D y_\w} F(y) = z_\w$ if for any sequence $(y_n) \subseteq D$ converging to $y_\w$ in $Y$, the sequence $(F(y_n))$ converges to $z_\w$ in $Z$. The notation $\limsup_{y \to_D y_\w}$ and $\liminf_{y \to_D y_\w}$ is defined analogously.
\end{notation}

\begin{lemma}\label{lem:back_decreasing}
Let $f$ be as in \cref{running_hypothesis_function}, and let $T\defeq T_f$.
Assume in addition that $\rnc_f=X$ and that $f$ is $\mu$-nowhere essentially two-ended. Suppose that $D \subseteq X$ is a Borel $E_f$-complete section such that $\rho^\bullet (x) \leq \rho^\bullet (f_D(x))$ for every $x \in D$. Then for a.e.\ $x\in X$, for every back $f$-end $\eta \in \del_T(f^{-\N}(x))$ we have
\[
\lim_{y \to_D \eta} \rho ^x (y) = 0 .
\]
Moreover, for a.e.\ $x \in X$ we have
\[
\lim_{n\to\infty}\sup \{ \rho ^x(y) : y\in D\cap f^{-[n,\infty )}(x)\} = 0.
\]
\end{lemma}

\begin{proof}
By \cref{smooth=rho-finite}\labelcref{part:coc-finite_preimages}, after discarding a null set we may assume that the maps $r_{D,f}$ and $f_D$ are both $\rho$-finite-to-one. 
By \cref{rho-infinite_back-orbits=>on_all_rho-infinite_ends}\labelcref{rho-infinite_back-orbits=>on_all_rho-infinite_ends:forward_recurrence} we may also assume that $D$ is forward $f$-recurrent (so $r_{D,f}$ is defined on all of $X$, and $f_D$ on all of $D$) and that $\Conv _T (D)=X$. 
For each $x \in X$ and $\eta \in \del_T(f^{-\N}(x))$ define $\overline{\rho}_D ^x(\eta) \defeq \limsup_{y \to_D \eta} \rho ^x(y)$.

\begin{claim+}\label{claim:meetgeodesic}
For each $x\in D$ and $\eta \in \del_T (f^{-\N}(x))$ with $\overline{\rho}_D^x(\eta )>0$, the intersection $D\cap [x,\eta )_T$ is infinite and $\rho ^x (\cdot )$ converges to $\overline{\rho}_D^x(\eta )$ along this set.
\end{claim+}

\begin{pf}
Since $\rho$ is nondecreasing in the $f$-direction along $D$, it is enough to show that $D\cap [x,\eta )_T$ is infinite. Suppose otherwise. 
Then, again using that $\rho$ is nondecreasing in the $f$-direction along $D$, there must be a sequence $(y_n)_{n\in \N}$ in $D\cap f^{-\N}(x)$ converging to $\eta$ and lying outside of $[x, \eta)_T$, such that each geodesic ray $(y_n, \eta)_T$ is disjoint from $D$ and $\lim _{n\rightarrow\infty} \rho ^x(y_n)=\overline{\rho}_D^x(\eta )>0$, see \cref{fig:yn_in_D_converging_to_eta}. 

\begin{figure}[htbp]
\centering

\colorlet{myblue}{blue!90!cyan}

\begin{tikzpicture}[
scale=0.8,
bnode/.style={circle, fill=black, inner sep=1.8pt},
rnode/.style={circle, fill=red, inner sep=1.8pt},
bluenode/.style={circle, fill=myblue, inner sep=1.8pt},
maincurve/.style={myblue, thick},
treecurve/.style={black, thick},
% Define the arrow style to be placed in the middle of paths
->-/.style={decoration={markings, mark=at position 0.55 with {\arrow{Stealth[length=2.5mm, width=1.8mm]}}}, postaction={decorate}}
]

% Define main curve coordinates
\coordinate (M1) at (0, 4);
\coordinate (M2) at (1.2, 4.4);
\coordinate (M3) at (2.5, 4.7);
\coordinate (M4) at (3.8, 4.85);
\coordinate (M5) at (5.0, 4.9);
\coordinate (M6) at (6.2, 4.8);
\coordinate (M7) at (7.4, 4.5);
\coordinate (M8) at (8.6, 4.0);
\coordinate (M9) at (9.7, 3.3);
\coordinate (M10) at (10.8, 2.5);
\coordinate (x) at (11.9, 1.6);

% --- DRAW TREES ---
% Tree 1 (Root: M1)
\node[bnode] (T1a) at (-0.8, 2.5) {};
\node[bnode] (T1b) at (-1.4, 1.2) {};
\node[rnode, label={[text=red, label distance=-3pt]-90:$y_{13}$}] (y13) at (-1.8, 0) {};
\node[rnode, label={[text=red, label distance=-3pt]-90:$y_{12}$}] (y12) at (-0.6, 0) {};
\node[rnode, label={[text=red, label distance=-3pt]-90:$y_{11}$}] (y11) at (0.2, 1.2) {};
\draw[treecurve, ->-] (T1a) to[bend left=10] (M1);
\draw[treecurve, ->-] (T1b) to[bend left=10] (T1a);
\draw[treecurve, ->-] (y13) to[bend left=10] (T1b);
\draw[treecurve, ->-] (y12) to[bend right=10] (T1b);
\draw[treecurve, ->-] (y11) to[bend right=10] (T1a);

% Tree 2 (Root: M3)
\node[bnode] (T2a) at (1.8, 3.2) {};
\node[bnode] (T2b) at (1.4, 1.8) {};
\node[rnode, label={[text=red, label distance=-3pt]-90:$y_{10}$}] (y10) at (1.1, 0.4) {};
\draw[treecurve, ->-] (T2a) to[bend left=10] (M3);
\draw[treecurve, ->-] (T2b) to[bend left=10] (T2a);
\draw[treecurve, ->-] (y10) to[bend left=10] (T2b);

% Tree 3 (Root: M4)
\node[bnode] (T3a) at (3.4, 3.3) {};
\node[rnode, label={[text=red, label distance=-3pt]-90:$y_9$}] (y9) at (2.4, 1.9) {};
\node[bnode] (T3b) at (3.2, 2.0) {};
\node[rnode, label={[text=red, label distance=-3pt]-90:$y_8$}] (y8) at (3.0, 0.6) {};
\draw[treecurve, ->-] (T3a) to[bend left=10] (M4);
\draw[treecurve, ->-] (y9) to[bend left=10] (T3a);
\draw[treecurve, ->-] (T3b) to[bend right=10] (T3a);
\draw[treecurve, ->-] (y8) to[bend left=10] (T3b);

% Tree 4 (Root: M5)
\node[bnode] (T4a) at (4.6, 3.4) {};
\node[bnode] (T4b) at (4.2, 2.0) {};
\node[rnode, label={[text=red, label distance=-3pt]-90:$y_7$}] (y7) at (3.9, 0.6) {};
\draw[treecurve, ->-] (T4a) to[bend left=10] (M5);
\draw[treecurve, ->-] (T4b) to[bend left=10] (T4a);
\draw[treecurve, ->-] (y7) to[bend left=10] (T4b);

% Tree 5 (Root: M8) - Contains y6, y5, and y4 branched from the same blue node
\node[bnode] (T6a) at (8.2, 3.0) {};
\node[bnode] (T6b) at (7.5, 1.8) {};
\node[rnode, label={[text=red, label distance=-3pt]-90:$y_6$}] (y6) at (6.7, 0.5) {};
\node[bnode] (T6c) at (8.5, 1.8) {};
\node[rnode, label={[text=red, label distance=-3pt]-90:$y_5$}] (y5) at (7.8, 0.6) {};
\node[rnode, label={[text=red, label distance=-3pt]-90:$y_4$}] (y4) at (8.6, -0.2) {};
\draw[treecurve, ->-] (T6a) to[bend left=10] (M8);
\draw[treecurve, ->-] (T6b) to[bend left=10] (T6a);
\draw[treecurve, ->-] (y6) to[bend left=10] (T6b);
\draw[treecurve, ->-] (T6c) to[bend right=10] (T6a);
\draw[treecurve, ->-] (y5) to[bend left=10] (T6c);
\draw[treecurve, ->-] (y4) to[bend right=10] (T6c);

% Tree 6 (Root: M9)
%\node[bnode] (T7a) at (9.4, 1.8) {};
\node[rnode, label={[text=red, label distance=-3pt]-90:$y_3$}] (y3) at (9.4, 1.8) {};
%\draw[treecurve, ->-] (T7a) to[bend left=10] (M9);
\draw[treecurve, ->-] (y3) to[bend left=10] (M9);

% Tree 7 (Root: M10)
\node[bnode] (T8a) at (10.2, 1.4) {};
\node[bnode] (T8b) at (9.7, 0.2) {};
\node[rnode, label={[text=red, label distance=-3pt]-90:$y_2$}] (y2) at (9.4, -0.9) {};
\node[bnode] (T8c) at (10.6, 0.2) {};
\node[rnode, label={[text=red, label distance=-3pt]-90:$y_1$}] (y1) at (10.5, -0.9) {}; 
\draw[treecurve, ->-] (T8a) to[bend left=10] (M10);
\draw[treecurve, ->-] (T8b) to[bend left=10] (T8a);
\draw[treecurve, ->-] (y2) to[bend left=10] (T8b);
\draw[treecurve, ->-] (T8c) to[bend right=10] (T8a);
\draw[treecurve, ->-] (y1) to[bend right=10] (T8c);

% Tree 8 (Root: x)
\node[bnode] (T9a) at (11.5, 0.5) {};
\node[rnode, label={[text=red, label distance=-3pt]-90:$y_0$}] (y0) at (11.2, -0.6) {};
\draw[treecurve, ->-] (T9a) to[bend left=10] (x);
\draw[treecurve, ->-] (y0) to[bend left=10] (T9a);

% --- DRAW MAIN CURVE ---
% Dotted to solid start (eta)
\coordinate (M_start) at (-2.4, 1.8);
\coordinate (M_mid) at (-1.2, 3.0);
\draw[maincurve, dotted, ->-] (M_start) to[bend left=4] (M_mid);
\draw[maincurve, ->-] (M_mid) to[bend left=4] (M1);
%\node[myblue, left] at (M_start) {$\eta$};
\node[label={[text=myblue, label distance=-6pt]-140:$\eta$}] at (M_start) {};

% Solid curve segments
\draw[maincurve, ->-] (M1) to[bend left=4] (M2);
\draw[maincurve, ->-] (M2) to[bend left=4] (M3);
\draw[maincurve, ->-] (M3) to[bend left=4] (M4);
\draw[maincurve, ->-] (M4) to[bend left=4] (M5);
\draw[maincurve, ->-] (M5) to[bend left=4] (M6);
\draw[maincurve, ->-] (M6) to[bend left=4] (M7);
\draw[maincurve, ->-] (M7) to[bend left=4] (M8);
\draw[maincurve, ->-] (M8) to[bend left=4] (M9);
\draw[maincurve, ->-] (M9) to[bend left=4] (M10);
\draw[maincurve, ->-] (M10) to[bend left=4] (x);

% Black continuation after x
\draw[thick, black, ->-] (x) to[bend left=4] (13.5, 0.6);
\draw[thick, black, dotted, ->-] (13.5, 0.6) to[bend left=4] (15, -0.5);

% Nodes on the main curve (drawn last to sit perfectly on top of lines)
\node[bluenode] at (M1) {};
\node[bluenode] at (M2) {};
\node[bluenode] at (M3) {};
\node[bluenode] at (M4) {};
\node[bluenode] at (M5) {};
\node[bluenode] at (M6) {};
\node[bluenode] at (M7) {};
\node[bluenode] at (M8) {};
\node[bluenode] at (M9) {};
\node[bluenode] at (M10) {};
\node[rnode, label={[text=red, label distance=-4pt]60:$x$}] at (x) {};

% Top f dashed arrow
%\draw[thick, dashed, -{Stealth[length=3mm, width=2.5mm]}] (5.0, 5.6) -- (6.4, 5.6) node[midway, above] {$f$};

\end{tikzpicture}

\caption{A sequence $(y_n) \subseteq D$ converging to $\eta$, where the points of $D$ are red.}
\label{fig:yn_in_D_converging_to_eta}
\end{figure}

The set $\{ y_n \} _{n\in \N}$ is $\rho$-infinite, but it is mapped to a single point under $r_{D,f}$, contradicting that $r_{D,f}$ is $\rho$-finite-to-one.
\end{pf}

Define $\overline{\rho}_D(x) \defeq \sup \set{\overline{\rho}_D ^x(\eta) : \eta \in \del_T (f^{-\N}(x))}$ for each $x\in X$.

\begin{claim+}\label{claim:supattained}
For each $x \in D$ there is some $y \in D$ with $f_D(y)=x$ and $\overline{\rho}_D(x) = \rho^x(y) \overline{\rho}_D(y)$. 
\end{claim+}
\begin{pf}
This is clear when $\overline{\rho}_D(x)=0$ since any $y$ in the nonempty set $f_D^{-1}(x)$ satisfies the claim, so assume that $\overline{\rho}_D(x)>0$. By \cref{claim:meetgeodesic} we have the equality
\begin{equation}\label{eq:supformula}
\overline{\rho}_D(x) = \sup \set{\rho^x(y) \overline{\rho}_D(y) : y \in f_D^{-1}(x)} .
\end{equation}
Since $f_D^{-1}(x)$ is $\rho$-finite and $\overline{\rho}_D(y)\leq 1$ for $y\in D$, only finitely many points $y\in f_D^{-1}(x)$ satisfy $\rho ^x (y)\overline{\rho}_D(y)>\overline{\rho}_D(x)/2$. Therefore, the supremum in \labelcref{eq:supformula} is in fact attained.
\end{pf}

Thus, by Luzin--Novikov uniformization, there is a Borel function $g : D \to D$ with $g(x) \in f_D^{-1}(x)$ and $\overline{\rho}_D (x) = \rho^x(g(x)) \overline{\rho}_D(g(x))$ for all $x \in D$. For each $x\in D$ the sequence $\rho ^x(x), \rho ^x(g(x)), \rho ^x(g^2(x)),\dots$ is nonincreasing and by the definition of $\overline{\rho}_D$ its limit is bounded above by $\overline{\rho}_D(x)$. On the other hand, by the cocycle identity, for each $n$ we have $\overline{\rho}_D(x)= \rho ^x(g^n(x))\overline{\rho}_D(g^n(x))\leq \rho ^x(g^n(x))$, and hence $\lim _{n\rightarrow\infty}\rho ^x(g^n(x)) = \overline{\rho}_D(x)$. It follows that if $x$ belongs to the $g$-invariant set $D_+\defeq \{ y\in D :\overline{\rho}_D(y)>0 \}$ then the $E_{g}$-class of $x$ is $\rho$-infinite. Therefore $E_{g}$ is nowhere smooth on $D_+$, so almost every $E_{g}$-class contained in $D_+$ is forward $f$-recurrent and hence spans exactly two ends of $T$. Thus, each component of $T$ restricted to the convex hull of $D_+$ is two-ended, so the set $D_+$ must be null because $T$ is $\mu$-nowhere essentially two-ended. This proves the first statement of the lemma.

If the last statement of the lemma fails then, since $D$ is forward $f$-recurrent, there is an $\epsilon >0$ and a positive measure set of $x\in D$ such that the set $A_x \defeq \{ y\in D\cap f^{-\N}(x) : \rho ^x(y)>\epsilon \}$ is infinite. The graph $T_{f_D}|A_x$ is connected by the hypothesis that $\rho^\bullet(y) \leq \rho^\bullet(f_D(y))$ for all $y\in D$, so it is an infinite tree rooted at $x$, and it is finitely branching since $f_D$ is $\rho$-finite-to-one. By K\"{o}nig's Lemma it contains an infinite branch, but this contradicts the first statement of the lemma.
\end{proof}

\begin{theorem}[Cocycle-vanishing of back-ends]\label{back_ends_converge_to_0}
Let $f$ be as in \cref{running_hypothesis_function}, and assume that $f$ is $\mu$-nowhere essentially two-ended. Then 
\[
\lim_{n\rightarrow\infty}\sup \{ \rho ^x(y):y\in f^{-[n,\infty )}(x)\} = 0
\]
for a.e.\ $x\in X$. In particular, if we let $\xi ^+(x)\defeq \lim _{n\rightarrow\infty}f^n(x)\in \del_{T_f}[x]_{E}$, then for a.e.\ $x\in X$, every end of $T _f \rest{[x]_{E}}$ other than $\xi ^+(x)$ is $\rho$-vanishing.
\end{theorem}

\begin{proof}
Throughout the proof we write $C$ for the directed Radon--Nikodym vertex core $\rnc_f$.
If $\mu (C)=0$ then the conclusion holds, so we may assume that $C$ has positive measure, and after restricting to the $E_f$-saturation of $C$ we may moreover assume that $C$ is an $E_f$-complete section. 
Let $\mu _C$ be the push-forward of $\mu$ to $C$ under the retraction $r_{C,f}$ to $C$ along $f$.
Let $\rho_C$ be the Radon--Nikodym cocycle of $E_{f_C}$ with respect to $\mu _C$, namely, $\rho _C(y,z) = \rho^z(r_{C,f}^{-1}(y)) / \rho^z(r_{C,f}^{-1}(z))$ for all $(y,z) \in E_{f_C}$. 
Then, after replacing $f$, $X$, $\mu$, and $\rho$ respectively by $f_C$, $C$, $\mu _C$, and $\rho _C$, we may assume without loss of generality that $X=C$.
In particular, $f^{-\N}(x)\setminus \{ x \}$ is nonempty for all $x\in X$.

Since $f$ is $\mu$-nowhere essentially two-ended, the contrapositive of \cref{having_bigger_points_behind} implies that after discarding a null set, the set $D$, of all $x \in X$ with $\rho^\bullet (z) < \rho^\bullet (x)$ for all $z \in f^{-\N}(x)\setminus \{ x \}$, is an $E_f$-complete section. 
By \cref{rho-infinite_back-orbits=>on_all_rho-infinite_ends}\labelcref{rho-infinite_back-orbits=>on_all_rho-infinite_ends:meets_rho-infinite_back-sets}, after discarding another null set, we may assume that $D$ meets every $\rho$-infinite geodesic ray in $T_f$.

%By \cref{rho-infinite_back-orbits=>on_all_rho-infinite_ends}, $\Conv_T D = X$, up to a null set, which we discard.

\begin{claim+}\label{claim:Dbehind}
For each $x\in X\setminus D$ there is some $y\in f^{-\N}(x)\cap D$ such that $\rho^\bullet (y)\geq \rho^\bullet (x)$.  
\end{claim+}

\begin{pf}
Suppose the claim fails for some $x_0 \in X\setminus D$ so that $\rho^\bullet (y)<\rho^\bullet (x_0)$ for all $y\in f^{-\N}(x_0)\cap D$.
Since $x_0\notin D$ there is some $x_1\in f^{-\N}(x_0)\setminus \{ x_0 \}$ with $\rho^\bullet (x_1)\geq \rho^\bullet (x_0)$. 
Note $x_1\not\in D$ by our assumption on $x_0$. 
Likewise, we can find some $x_2\in f^{-\N}(x_1)\setminus \{ x_1 \}$ with $\rho^\bullet (x_2)\geq \rho^\bullet (x_1)\geq \rho^\bullet (x_0)$, and hence also $x_2\not\in D$.  
Continuing in this way, we obtain an infinite sequence $x_0,x_1,x_2,\dots$ with $\rho^\bullet (x_n)\geq \rho^\bullet (x_0)$ that spans the geodesic ray $[x_0,\eta )_{T_f}$ from $x_0$ to some back $f$-end $\eta$. 
This geodesic ray is $\rho$-infinite, so it contains some point $y$ of $D$. %by part \labelcref{rho-infinite_back-orbits=>on_all_rho-infinite_ends:meets_rho-infinite_back-sets} of \cref{rho-infinite_back-orbits=>on_all_rho-infinite_ends}. 
But then for $n$ large enough we have $x_n\in f^{-\N}(y)\setminus \{ y\}$ and hence $\rho^\bullet (y)>\rho^\bullet (x_n)\geq \rho^\bullet (x_0)$, a contradiction.
\end{pf}

By \cref{claim:Dbehind} and \cref{lem:back_decreasing}, after discarding a null set, for every $x\in X$ we have 
\[
\sup \{ \rho ^x(y):y\in f^{-[n,\infty )}(x)\} \leq \sup \{ \rho ^x(y) : y\in D\cap f^{-[n,\infty )}(x) \} \longrightarrow 0 \ \text{ as }\ n\rightarrow\infty . \qedhere
\]

%By \cref{lem:back_decreasing}, after discarding a null set we have that $\lim_{y \to_D \eta} \rho^\bullet(y) = 0$ for every back $f$-end $\eta$. for every $x\in X$. Given a sequence $(x_n)$ which converges to a back $f$-end $\eta$, by \cref{claim:Dbehind} for each $n$ we can find $y_n\in f^{-\N}(x)\cap D$ with $\rho^\bullet(y_n) \geq \rho^\bullet(x_n)$. Hence $\limsup_n \rho^\bullet(x_n) \leq \lim _n \rho^\bullet(y_n)=0$.  
\end{proof}

%Tey Berendschot, Stefaan Vaes: https://arxiv.org/abs/2005.06309

\section{Behavior of the cocycle along ends of a function}\label{sec:behavior}

Throughout this section we continue to work under \cref{running_hypothesis_function}. 

\begin{prop}[Essential end decomposition]\label{end_decomposition}
Assume \cref{running_hypothesis_function}.
Then there is a unique (up to null sets) partition of $X$ into $E_f$-invariant Borel sets $Y_0,Y_1,Y_2,Y_{\infty}$ where 
\begin{itemize}
\item $Y_0$ is the unique $\mu$-maximal $E_f$-invariant Borel set on which $E_f$ is smooth.

\item $Y_i$, $i = 1,2$, is the unique $\mu$-maximal $E_f$-invariant Borel subset of $X$ on which $T_f$ is essentially $i$-ended and $E_f$ is $\mu$-nowhere smooth.

\item $Y_{\infty}$ is the unique $\mu$-maximal $E_f$-invariant Borel subset of $X$ on which $T_f$ is $\mu$-nowhere essentially finitely-ended and $E_f$ is $\mu$-nowhere smooth.
\end{itemize}
\end{prop}

\begin{proof}
This follows from measure-theoretic exhaustion and \cref{0-ended=>smooth}, \cref{3-ended=>smooth}, and \cref{uniquely_ended}.
\end{proof}

\begin{prop}\label{coboundary=>finite-ended}
Assume \cref{running_hypothesis_function}, and suppose in addition that $\rho$ is a coboundary.
Then, after discarding a null set, $X$ splits into $E_f$-invariant Borel sets $X_1$ and $X_2$ such that $f$ is essentially one-ended on $X_1$ and essentially two-ended on $X_2$.
\end{prop}
\begin{proof}
It is clear that $f$ is essentially one-ended on each $E_f$-invariant Borel set on which $E_f$ is smooth.
We may therefore assume without loss of generality that $E_f$ is $\mu$-nowhere smooth.

By \cref{equivalent_measures}\labelcref{coboundary=invariant}, there is an $E_f$-invariant $\sigma$-finite Borel measure $\nu$ on $X$ equivalent to $\mu$.
Since $\nu$ is $\sigma$-finite, there is a Borel $E_f$-complete section $Y$ of finite $\nu$-measure.
Rescaling $\nu \rest{Y}$ if needed, we may assume that it is a probability measure on $Y$. 
Let $g \defeq f_Y \rest{Y}$. 
By \cref{smooth=rho-finite}, since $E_f \rest{Y}$ is pmp, after discarding a null set $g$ is an acyclic finite-to-one Borel function, and by \cref{classic_Adams} there is a partition $Y = Y_1 \cup Y_2$ into $E_g$-invariant Borel sets such that each $T_g$-component in $Y_i$ has exactly $i$-many ends, for $i = 1,2$. 
The fact that $g$ is finite-to-one ensures that the restriction of $f$ to the convex hull $\Conv_{T_f} Y_i$ is $i$-ended, witnessing that $f$ is essentially $i$-ended on $X_i \defeq [Y_i]_{E_f}$, for $i = 1,2$.
\end{proof}

\subsection{The forward end}

\begin{lemma}\label{cocycle_along_f}
Assume \cref{running_hypothesis_function}, and let $\xi ^+ (x) \defeq \lim _{n\to \infty}f^n(x) \in \del_T [x]_{E_f}$ for each $x\in X$. 

\begin{enumerate}[label=\textnormal{(\alph*)}]

\item\label{forward-lim=0} Suppose that $\lim _{n\rightarrow\infty}\rho ^x(f^n(x))=0$ for a.e.\ $x\in X$. 
Then, after discarding a null set, there is a forward $f$-invariant Borel $E_f$-complete section on which $f$ is one-ended.
In particular, $f$ is essentially one-ended.

\item\label{forward-nonvanishing} After discarding a null set, $E_f$ is smooth on the set of all $x \in X$ for which $\xi^+(x)$ is $\rho$-vanishing.

\item\label{finite+nonzero=>coboundary} If for a.e. $x \in X$, at least one of
\[
\liminf_{n\rightarrow\infty}\rho ^x(f^n(x)), 
\;\;
\limsup_{n\rightarrow\infty}\rho ^x(f^n(x)),
\;\;
\liminf_{y\to \xi ^+(x) }\rho ^x(y),
\;
\text{ and }
\;
\limsup_{y\to \xi ^+(x) }\rho ^x(y)
\]
is finite and nonzero, then $\rho$ is a coboundary and $X$ splits into $E_f$-invariant Borel sets $X_1$ and $X_2$ such that $f$ is essentially one-ended on $X_1$ and essentially two-ended on $X_2$ a.e. 

\item\label{forward-lim=infinity} If $\lim _{n\rightarrow\infty}\rho ^x(f^n(x))=\infty$ for a.e.\ $x\in X$ then $f$ is $\mu$-nowhere essentially two-ended.

%\item\label{automorphism} If $f$ is a Borel automorphism of $X$ then 
%\[
%\liminf_{n\rightarrow\infty}\rho ^x(f^{-n}(x))=\liminf_{n\rightarrow\infty}\rho ^x(f^{n}(x)) \text{ and } \limsup_{n\rightarrow\infty}\rho ^x(f^{-n}(x))=\limsup_{n\rightarrow\infty}\rho ^x(f^{n}(x)).
%\]
\end{enumerate}
\end{lemma}
\begin{proof}
\labelcref{forward-lim=0} Define $Y \defeq \set{x \in X : \rho^\bullet(x) > \rho^\bullet(f^n(x)) \text{ for every } n \geq 1}$. 
By the assumption of the case, after discarding a null set, $Y$ is an $E_f$-complete section. 
By \cref{char-rho-finite-back-orbits}, it is enough to show that the next return map $f_Y:Y\to Y$ is one-ended a.e., so suppose towards a contradiction that the (analytic, hence $\mu$-measurable) set $Z$, of points in $Y$ lying on a bi-infinite geodesic line through $T_{f_Y}$, has positive measure. 
This contradicts \cref{backward_mass_transport} applied to $f_Z\rest{Z}$.

\labelcref{forward-nonvanishing} Assume without loss of generality that $\xi^+(x)$ is $\rho$-vanishing for all $x \in X$. 
Together with part \labelcref{forward-lim=0} this implies that, after discarding a null set, all ends of $T_f$ are $\rho$-vanishing.
(Alternatively, one could apply \cref{back_ends_converge_to_0} instead of \labelcref{forward-lim=0}.)
The statement then follows from \cref{smooth_iff_all_vanishing_ends}.

\labelcref{finite+nonzero=>coboundary} Splitting the space $X$ into four $E_f$-invariant Borel sets, we assume that $\limsup_{n \rightarrow \infty} \rho^x(f^n(x))$ is finite and nonzero for a.e.\ $x\in X$; the proof for the other three cases is similar.

Define $m : X \to \Rpos$ by $m(x) \defeq \big(\limsup_{n \to \infty} \rho ^x (f^n(x))\big)^{-1}$. Then $\rho^x(y) = \frac{m(y)}{m(x)}$ for all $(x,y) \in E_f$ since
\[  
\rho^x(y) \cdot m(y)^{-1}
= 
\limsup_{n \to \infty} \rho^x(y) \rho^y(f^n(y))
= 
\limsup_{n \to \infty} \rho^x(f^n(y))
= 
\limsup_{n \to \infty} \rho^x (f^n(x)) = m(x)^{-1}.
\]
Thus $\rho$ is a coboundary and the rest follows from \cref{coboundary=>finite-ended}.

\begin{comment}
Define a measure $\nu$ on $X$ by $d \nu \defeq \frac{1}{m} \, d\mu$, i.e.\ $\nu(A) \defeq \int_A \frac{1}{m} \, d\mu$ for every Borel $A \subseteq X$. Now $\nu$ is $E_f$-invariant because mass transport holds: for each Borel function $F : E_f \to [0,1]$,

\begin{align*}
\int_X \sum_{y \in [x]_{E_f}} F(x,y) \, d\nu(x) 
&=
\int_X \sum_{y \in [x]_{E_f}} F(x,y) m(x)^{-1}\, d\mu(x) 
\\
&=
\int_X \sum_{y \in [x]_{E_f}} F(x,y) m(y)^{-1} \rho^x(y) \, d\mu(x) 
\\
\eqcomment{$\rho$-mass transport}
&=
\int_X \sum_{x \in [y]_{E_f}} F(x,y) m(y)^{-1} \rho^x(y) \rho^y(x) \, d\mu(y) 
\\
&=
\int_X \sum_{x \in [y]_{E_f}} F(x,y) \, d\nu(y).
\end{align*}

\end{comment}

\labelcref{forward-lim=infinity} It is enough to show that any forward $f$-invariant Borel set $Y$ on which $f$ restricts to a bijection is null. By the hypothesis on $\rho$, the set
\[
Z \defeq \set{y \in Y : \rho^\bullet(y) < \rho^\bullet(f^n(y)) \text{ for all } n \ge 1}
\]
is an $E_f \rest{Y}$-complete section with the property that the next return map $g \defeq f_Z$ satisfies $\rho^\bullet(z) < \rho^\bullet(g(z))$ for every $z \in Z$.
The conclusion therefore follows from \cref{backward_mass_transport}.
\end{proof}

As shown in \cref{example:least_deletion}, there are examples of one-ended functions $f$ as in \labelcref{running_hypothesis_function} such that each geodesic $[x,\xi ^+(x))_{T_f}$ is $\rho$-finite (and hence the hypothesis in \labelcref{forward-lim=0} of \cref{cocycle_along_f} holds) yet the relation $E_f$ is $\mu$-nowhere smooth and thus $\xi ^+(x)$ is $\rho$-nonvanishing a.e.\ in these examples by \labelcref{forward-nonvanishing} of \cref{cocycle_along_f}.

The conclusion in \labelcref{forward-lim=infinity} of \cref{cocycle_along_f} cannot be strengthened: \cref{subsec:examples:shifts} and \cref{example:least_deletion} contain respectively $\mu$-nowhere finitely-ended and one-ended examples satisfying $\lim _{n\rightarrow\infty}\rho ^x(f^n(x))=\infty$ for a.e.\ $x\in X$.

\subsection{Behavior in terms of essential number of ends}

To fully describe the behavior of the cocycle along ends it is enough to restrict to the cases given by the essential end decomposition (\cref{end_decomposition}). 
The case where $T_f$ is essentially one-ended is handled in \cref{char-rho-finite-back-orbits}.

\begin{prop}[Essentially two-ended]\label{two-ended_function}
Assume \cref{running_hypothesis_function}, and suppose in addition that $f$ is essentially two-ended and that $E_f$ is $\mu$-nowhere smooth. 
Then, after discarding a null set, each $T_f$-component has exactly one $\rho$-nonvanishing back end and there exists an essentially unique forward $f$-invariant Borel $E_f$-complete section $X_2 $ such that $f \rest{X_2} : X_2 \to X_2$ is a bijection; moreover, $X_2$ is the Radon--Nikodym vertex core of $T_f$. 
For $x\in X$ let $\xi ^+ (x) \defeq \lim _{n\to \infty}f^n(x) \in \del_{T_f} [x]_{E_f}$, and let $\xi ^{-}(x)$ denote the end of $T_f \rest{X_2 \cap [x]_{E_f}}$ distinct from $\xi ^+ (x)$. 
Then for a.e.\ $x\in X_2$ we have
\begin{align*}
\liminf_{y\to \xi ^+(x) } \rho ^x(y) = \liminf _{y\to \xi ^{-}(x)}\rho ^x(y) < \infty
\ &\text{ and } \ \limsup _{y\to \xi ^+(x) } \rho ^x(y) = \limsup _{y\to \xi ^{-}(x)}\rho ^x(y) > 0
\\
\liminf_{n \to \infty} \rho^x(f^n(x)) = \liminf _{n \to \infty}\rho ^x(f^{-n}(x)) < \infty
\ &\text{ and } \ \limsup _{n \to \infty} \rho ^x(f^n(x)) = \limsup _{n \to \infty}\rho ^x(f^{-n}(x)) > 0.
\end{align*}
In addition, if we let $X'$ denote the set of all $x\in X$ for which one of these limits inferior is positive or one of these limits superior is finite, then the restriction of $\rho$ to $E\rest{X'}$ is a coboundary.
\end{prop}

\begin{proof}
Let $Z$ be a Borel $T_f$-convex $E_f$-complete section on which $T_f$ is two-ended. 
Then the set $X_2$, of all points $x \in Z$ lying on the geodesic between the two ends of $T_f\rest{Z\cap [x]_{E_f}}$, has the required property.

We prove only the first displayed line of the proposition, since the proof of the second line is almost identical.
Define
\begin{align*}
i_\pm &\defeq \liminf_{y\to \xi ^\pm(x) } \rho ^x(y) 
\\
s_\pm &\defeq \limsup _{y\to \xi ^\pm(x) } \rho ^x(y),
\end{align*}
and note that $i_+$ and $i_-$ are finite by \cref{cocycle_along_f}\labelcref{forward-lim=infinity}.
By \cref{cocycle_along_f}\labelcref{forward-lim=0} and \cref{uniquely_ended}, $E_f$ is smooth on the set on which at least one of $s_+$ and $s_-$ is zero, and hence this set is null.
The desired conclusion holds on the $E_f$-invariant Borel set where both $i_+ = 0 = i_-$ and $s_+ = \infty = s_-$, so we may assume without loss of generality that at least one of $i_+, i_-, s_+, s_-$ is nonzero and finite. 

Then $\rho$ is a coboundary by \cref{cocycle_along_f}\labelcref{finite+nonzero=>coboundary}. 
Thus, there is a Borel $m : X \to \Rpos$ such that $\rho(x,y) = m(x) \cdot m(y)^{-1}$ for each $(x,y) \in E_f$, and hence, it is enough to prove that $\limsup_{y\rightarrow \xi ^+(x)} m(y) = \limsup_{y\rightarrow \xi ^-(x)} m(y)$ and $\liminf_{y\rightarrow \xi ^+(x)} m(y) = \liminf_{y\rightarrow \xi ^-(x)} m(y)$. 
We only write the proof for 
\[
\ell_+ \defeq \limsup_{y\rightarrow \xi ^+(x)} m(y) = \limsup_{y\rightarrow \xi ^-(x)} m(y)\eqdef \ell_-
\]
since the argument for $\liminf$ is analogous.
Note that the functions $\ell_+$ and $\ell_-$ are $E_f$-invariant.
Let $Y$ be the ($E_f$-invariant) set of points for which $\ell_- > \ell_+$ and let $L : X \to \Rpos$ be an $E_f$-invariant Borel function such that $\ell_- > L > \ell_+$ pointwise, e.g., if $\ell_- < \infty$, then let $L \defeq (\ell_- + \ell_+)/2$ and if $\ell_- = \infty$, then take $L \defeq \ell_+ + 1$. 
Then the convex hull of the set $\set{y \in Y : m(y) > L}$ meets every $f$-orbit in $Y$ in a non-forward-$f$-recurrent set, witnessing the smoothness of $E_f$ on $Y$, so $Y$ is null.
Thus, $\ell_- \le \ell_+$, and switching the roles of $+$ and $-$, we also get $\ell_- \ge \ell_+$, which completes the proof.

The fact that $X_2$ is the Radon--Nikodym  vertex core of $T_f$ now follows from \cref{pruning}, and this implies the essential uniqueness of $X_2$ as well as the fact that each $T_f$-component has exactly one $\rho$-nonvanishing back end.
\end{proof}

\begin{prop}\label{funct_nowhere_fin_ended}
Assume \cref{running_hypothesis_function}, and suppose in addition that $f$ is $\mu$-nowhere essentially finitely-ended. 
For $x\in X$ let $\xi^+ (x) \defeq \lim _{n\to \infty}f^n(x) \in \del_{T_f} [x]_{E_f}$. 
Then for a.e.\ $x \in X$ we have:
\begin{enumerate}[(a)]
\item\label{unique_nonvanishing_end} The end $\xi^+ (x)$ is the unique $\rho$-nonvanishing end in $\del_{T_f} [x]_{E_f}$.

\item\label{infinite_limsup}  $\displaystyle{\limsup_{n\rightarrow\infty}\rho ^x(f^n(x)) = \infty}$. 

\item\label{0_liminf}  $\displaystyle{\liminf_{y \to \xi^+(x)}\rho ^x(y) = 0}$. 

\item\label{summable_back_geodesics} For each end $\eta$ of $T_f \rest{[x]_{E_f}}$ distinct from $\xi^+(x)$, the geodesic ray $[x, \eta)_{T_f}$ is $\rho$-finite.
Moreover, the backward geodesic weight of $f$ at $x$, i.e., the quantity
\[
\Sigma_f(x)\defeq \sup \{ \rho ^x([x,v]_{T_f}) : v\in f^{-\N}(x) \},
\]
is finite.
\end{enumerate}
\end{prop}

\begin{proof}
Part \labelcref{unique_nonvanishing_end} follows from \cref{back_ends_converge_to_0} and \cref{cocycle_along_f}\labelcref{forward-nonvanishing}.
For \labelcref{infinite_limsup}, first note that $\limsup _{n \rightarrow \infty}\rho^x(f^n(x)) > 0$ by \cref{cocycle_along_f}\labelcref{forward-lim=0}, and hence \cref{cocycle_along_f}\labelcref{finite+nonzero=>coboundary} implies that $\limsup _{n\rightarrow\infty}\rho ^x(f^n(x))=\infty$. 
For part \labelcref{0_liminf}, note that each neighborhood of $\xi^+(x)$ in $\del_{T_f} [x]_{E_f}$ contains other ends because otherwise $f$ would be $\mu$-somewhere essentially finitely-ended. These other ends are $\rho$-vanishing by part \labelcref{unique_nonvanishing_end}, which implies \labelcref{0_liminf}.
Finally, part \labelcref{summable_back_geodesics} is immediate from \cref{summable_back-rays}.
\end{proof}

\begin{cor}[Barytropy and essential number of ends]\label{essential_ends=barytropic}
Assume \cref{running_hypothesis_function} and suppose in addition that $E_f$ is $\mu$-nowhere smooth.
Then
\begin{enumerate}[(a)]
\item\label{item:monobarytropic} $f$ is $\rho$-monobarytropic if and only if, after discarding an $E_f$-invariant Borel null set, $f$ is essentially one-ended.

\item\label{item:dibarytropic} $f$ is $\rho$-dibarytropic if and only if, after discarding an $E_f$-invariant Borel null set, $f$ is essentially two-ended.

\item\label{item:polybarytropic} $f$ is $\rho$-polybarytropic if and only if $f$ is $\mu$-nowhere essentially finitely-ended.
\end{enumerate}
\end{cor}
\begin{proof}
The $\mu$-nowhere smoothness implies that a.e.\ $T_f$-component has at least one $\rho$-barytropic end by \cref{smooth_iff_all_vanishing_ends}.
Part \labelcref{item:monobarytropic} now follows immediately from \cref{char-rho-finite-back-orbits}.
The right-to-left implication of \labelcref{item:monobarytropic,item:dibarytropic} follows from the pruning (\cref{pruning}\labelcref{pruning:convex}) and the other implication follows by taking the convex hull of the two barytropic ends in a.e.\ $T_f$-component.
Finally, \labelcref{item:polybarytropic} follows from \labelcref{item:monobarytropic,item:dibarytropic} because $\mu$-nowhere smoothness implies, by \cref{3-ended=>smooth}, that if $f$ is $\mu$-somewhere essentially $n$-ended for some $n \in \N$, then $n$ equals $1$ or $2$.
\end{proof}

\section{Representative examples of cocycle behaviors}\label{sec:examples}

In this section we give examples of functions exhibiting various cocycle behaviors along geodesics.

\subsection{Multi-ended, forward convergence to $\infty$}\label{subsec:examples:shifts}

\begin{example}\label{example:one_sided_shift}
Let $k\geq 2$, equip $k^\N$ with the product measure $\nu ^\N$, where $\nu$ is the uniform distribution on $k$, let $X\subseteq k^\N$ be the co-countable set of sequences which are not eventually periodic, and let $h:X\to X$ be the one-sided shift map, which is acyclic and $k$-to-$1$.
The orbit equivalence relation $E_h$ is mcp with associated Radon--Nikodym cocycle $\rho$ determined by $\rho ^x(h(x))=k$ for all $x\in X$.
Thus, in each $T_h$-component, the forward $h$-end is nonvanishing and, in fact, each forward geodesic exhibits uniform exponential cocycle growth, while the back $h$-ends are vanishing and geodesics to these ends have uniform exponential cocycle decay. 
However, all back $h$-ends (and hence all ends of $T_h$) are barytropic since $\rho ^x(h^{-n}(x))=1$ for all $n\in\N$ and $x\in X$.
\end{example}

\begin{example}\label{example:free_group_boundary}
Let $\F_d$ be the free group on a free generating set $S$ of cardinality $2 \le d \le \aleph_0$.
The translation action of $\F_d$ by automorphisms on its standard Cayley graph $T$ extends to a continuous action on the completion $\-{\F_d}^T$.
We identify the boundary $\partial_T \F_d$ with the space $\partial \F_d$ of infinite (indexed by $\N$) reduced words in the generators and their inverses.
Let $m$ be a symmetric probability measure on $S^\pm \defeq S \cup S^{-1}$ assigning to every generator positive measure.
Then the associated network has a positive (weighted) isoperimetric constant, hence the $m$-random walk is transient by \cite[Theorem 6.7]{Lyons-Peres:book}.
Let $\mu$ denote the associated hitting probability measure on $\partial \F_d$ of the $m$-random walk starting at the identity.
Explicitly, the measure $\mu$ defined on the clopen cylinder $C_w$ of all infinite reduced words beginning with the nonempty finite reduced word $w = w_0 w_1 \dots w_k \in \F_d$ is given by the formula:
\[
\mu(C_w) \defeq m(w_0) P(w_0, w_1) P(w_1, w_2) \cdots P(w_{k-1}, w_k),
\]
where the (transition) matrix $P$ on $S^\pm$ is defined, for $a,b \in S^\pm$, by
$
P(a,b) \defeq \tfrac{m(b)}{m(S^\pm \setminus \{a^{-1}\})}
$
if $a^{-1} \ne b$, and $P(a,b) \defeq 0$ otherwise.
(This can be shown by induction on $k$.)
Note that the orbit equivalence relation $E$ of the boundary action of $\F_d$ on $\del \F_d$ coincides with $E_h$, where $h : \del \F_d \to \del \F_d$ is the one-sided shift map.
One now verifies, as in \cite[Proposition 5.3]{Tserunyan-Zomback}, that the Radon--Nikodym cocycle $\rho$ of the orbit equivalence relation $E$ of the boundary action of $\F_d$ on $(\del \F_d, \mu)$ is given by
\[
\rho^{h(x)}(x) = \frac{m(x_0) P(x_0, x_1)}{m(x_1)} = \frac{m(x_0)}{m(S^\pm \setminus \{x_0^{-1}\})}
\]
for $\mu$-a.e.\ $x \in \del \F_d$.
Note that $\alpha \defeq \sup_{a \in S^\pm} \frac{m(a)}{m(S^\pm \setminus \{a^{-1}\})} < 1$ and hence the back $h$-ends are vanishing and geodesics to these ends have uniform exponential cocycle decay, while the forward $h$-end is nonvanishing and, in fact, each forward geodesic exhibits uniform exponential cocycle growth.
\end{example}

\subsection{One-ended, all behaviors}\label{example:least_deletion}

Let $X\subseteq \{ 0, 1 \} ^{\N}$ be the set of all binary sequences containing infinitely many $1$'s and $0$'s, and let $f : X \rightarrow X$ be the \dfn{least deletion map}, i.e., $f$ flips the first $1$ to a $0$.
Thus, $X$ may be naturally identified with the space of all infinite and co-infinite subsets of $\N$, in which case $f$ corresponds to the map which deletes the least element from a given subset. 
Then each $T_f$-component is a locally finite one-ended tree.
The equivalence relation $E_f$ coincides with the equivalence relation $E_0$ of eventual equality of binary sequences.
For $p \in (0, 1)$ let $\mu _p$ be the $p$-weighted i.i.d.\ coin flip measure on $\{ 0, 1 \} ^{\N}$, i.e., $\mu _p =\nu _p ^{\N}$ where $\nu _p$ is the probability measure on $\{ 0, 1 \}$ with $\nu _p ( \{ 1 \} )=p$. 
Then $\mu _p(X)=1$ and $E_f$ is mcp on $(X,\mu_p)$ with associated Radon--Nikodym cocycle $\rho ^y(x) = \lambda ^{k_0-k_1}$, where $\lambda = \frac{p}{1-p}$, $k_0$ is the number of indices $n$ with $y(n)=0$ and $x(n)=1$, and $k_1$ is the number of indices $n$ with $y(n)=1$ and $x(n)=0$.

Observe that $\rho ^x(f^n(x))=\lambda ^{-n}$ for each $x\in X$.
Therefore, if $p>\frac{1}{2}$ then each forward $f$-geodesic is cocycle-finite since in this case $\lambda >1$, hence $\sum_{n\geq 0}\rho ^y(f^n(y))<\infty$; on the other hand, if $p<\tfrac{1}{2}$ then $\lambda <1$ and each forward $f$-geodesic has exponential cocycle-growth to $\infty$.

In an earlier draft of this article, we asked the following questions: do there exist examples of acyclic $f$ with $E_f$ mcp, where each $T_f$-component is one-ended and either of the following Radon--Nikodym topographies manifest?
\begin{enumerate}[leftmargin=*]
    \item\label{q:cocycle-oscillating} The cocycle has oscillatory behavior between $0$ and $\infty$ along each forward $f$-geodesic.
    \item\label{q:cocycle-vanishing_unsummably} Each forward $f$-geodesic is cocycle-infinite, yet the cocycle converges to $0$ along such geodesics.
\end{enumerate}

\noindent This has since been answered positively by Bell, Chu, and Rodgers in \cite[Examples 4.1 and 5.1]{kids}, where they prove that for appropriate choices of product measures on $\{ 0, 1\} ^\N$, the least deletion map exhibits either of the desired topographies.
After this result, we also found other examples of one-ended functions with oscillatory behavior of the Radon--Nikodym cocycle along the forward geodesic, see \cref{example:any-ended_oscillation}.

\subsection{Two-ended, oscillatory behavior}\label{example:odometer}

Let $X$ be as in \cref{example:least_deletion}, but now take $f:X\to X$ to be the odometer map $f(1^n0x)=0^n1x$, so that $E_f$ still coincides with $E_0$ on $X$, but now each $T_f$-component is a bi-infinite line.
Then for $p\in (0,1)$ with $p\neq \tfrac{1}{2}$ the measure $\mu _p$ is not equivalent to an $E_f$-invariant $\sigma$-finite measure (in fact, the relation $E_f$ on $(X,\mu _p )$ is type $\mathrm{III}_{\lambda}$ where $\lambda = \tfrac{p}{1-p}$), hence by \cref{two-ended_function} the Radon--Nikodym cocycle exhibits oscillatory behavior between $0$ and $\infty$ along each $T_f$-component in both directions, i.e., 
\[
\limsup _{n\to\infty}\rho ^x(f^{n}(x))=\limsup _{n\to\infty}\rho ^x(f^{-n}(x))=\infty  
\; \text{ and }\;  
\liminf _{n\to\infty}\rho ^x(f^{n}(x))=\liminf _{n\to\infty}\rho ^x(f^{-n}(x))=0
\]
for a.e.\ $x\in X$.

\subsection{All possibilities of number of ends with oscillatory behavior}\label{example:any-ended_oscillation}

We now build a family of examples containing for each number of ends a function with oscillatory forward behavior.
This in particular includes an example of a $\mu$-nowhere essentially finitely-ended function, and yet another example of a one-ended function.

Start with an acyclic countable-to-one Borel function $g:X\to X$ with $E_g$ mcp on $(X,\mu )$ such that $\limsup_n \rho^x(g^n(x)) =\infty$ for a.e.\ $x \in X$.

We will expand $(X,\mu )$ to a finite measure space $(\tilde{X},\tilde{\mu})$ and define a function $f:\tilde{X}\to \tilde{X}$ with $X$ an $E_f$-complete section such that the next-return map induced by $f$ on $X$ is $g$, and the Radon--Nikodym cocycle exhibits oscillatory behavior between $0$ and $\infty$ along each forward $f$-geodesic.

Since $E_g$ is hyperfinite, we may write it as an increasing union of finite subequivalence relations $F_n$, $n\in \N$, each of whose classes are $T_g$-connected. 
Let $X_n\subseteq X$ be the $F_n$-transversal consisting of the $g$-frontmost point of each $F_n$-class.
We may assume that $F_0$ is the equality relation on $X$ so that $X=X_0\supseteq X_1\supseteq X_2\supseteq\cdots$ is a vanishing sequence of $E_g$-complete sections. 
Let $\mu _n$ be the finite measure on $X_n$ determined by 
\[
\frac{d\mu_n}{d\mu}(x)=2^{-n}\min_{y\in [x]_{F_n}}\rho ^x(y)
\]
for all $x\in X_n$.
Then equip $\tilde{X}\defeq X\sqcup \bigsqcup_{n\geq 0}X_n\times \{n\}$ with the measure $\tilde{\mu}\defeq \mu + \sum _{n\geq 0}\mu_n\times \delta _n$, which is finite by construction.
Define $f:\tilde{X}\to\tilde{X}$ by 
\begin{align*}
f(x)&=(x,n) \text{ if }x\in X_n\setminus X_{n+1},\\ f(x,n)&=(x,n-1) \text{ if $x\in X_n$ and $n>0$},\\
f(x,0)&=g(x) \text{ for each $x\in X$}.
\end{align*}
It is clear that $X$ is an $E_f$-complete section.
The assumption that $\limsup _n \rho^x(g^n(x))=\infty$ implies that $\limsup_n \rho^x(f^n(x))=\infty$, since the sequence $(g^n(x))_{n\in \N}$ is cofinal in the forward $f$-geodesic ray starting at $x$.
To establish the claimed oscillation it is therefore enough to show that $\liminf _n \rho ^x(f^n(x))=0$ for each $x\in X$.
Recall that $r_{g,X_n}$ denotes the retraction map from $X$ to $X_n$ via $g$ (so $F_n$ coincides with relation of having the same $r_{g,X_n}$-image).
Given $x\in X$, the sequence $(r_{g,X_n}(x))_{n\in \N}$ is cofinal in the forward $f$-geodesic ray starting at $x$, and by definition of $\tilde{\mu}$ we have $\rho ^x(f(r_{g,X_n}(x)))\leq 2^{-n}$.

Lastly, note that $T_f$ is essentially $n$-ended (resp.\ $\mu$-nowhere essentially finitely-ended) if and only if $T_g$ is.
Thus, by starting with an appropriate function $g$, this construction gives examples of oscillating cocycle behavior along the forward geodesics in each of the three regimes.

\noindent
\begin{enumerate}[(i),leftmargin=*]
\item For a one-ended example, take the least deletion function in \cref{example:least_deletion} with the Bernoulli measure $\mu(p)$ on $2^\N$, where $p < \frac{1}{2}$.

\item For a two-ended example, take the odometer in \cref{example:odometer}.

\item For a $\mu$-nowhere essentially finitely-ended example, take $g$ to be the one-sided shift function $h$ from \cref{example:one_sided_shift}; since every $T_h$-end is barytropic, $h$ is $\mu$-nowhere essentially finitely-ended by \cref{essential_ends=barytropic}\labelcref{item:polybarytropic}.
\end{enumerate}

\section{The spaces of ends of topographical significance}

\begin{workinghyp}\label{Adams_workinghyp}
$T$ is an acyclic locally countable mcp Borel graph on a standard probability space $(X,\mu)$ with associated Radon--Nikodym cocycle $\rho :E_T\rightarrow \Rpos$.
\end{workinghyp}

\subsection{The space of nonvanishing ends and its relation to spaces of other ends}

Recall from \cref{def:barytropic-vanishing-finite_geod} that for a vertex $x \in X$ and $\xi \in \del_T [x]_{E_T}$, the \dfn{upper $\rho$-weight of $\xi$ with respect to $x$} is the quantity
\[
\urho^x(\xi) \defeq \limsup_{y \to \xi} \rho^x(y).
\]

The following theorem in particular implies that the upper $\rho$-weight of nonvanishing ends of $T$ is constant within a.e.\ $T$-component $C$ (with respect to a fixed basepoint $x \in C$).

\begin{theorem}[Constant upper weight of nonvanishing ends]\label{constant_weight}
Assume \cref{Adams_workinghyp}.
Then for a.e.\ $x \in X$, each $\rho$-nonvanishing end $\xi$ of $T \rest{[x]_{E_T}}$ satisfies
\begin{equation}\label{eq:sup-weight_attained}
\urho^x(\xi) = \sup_{y \in [x]_{E_T}} \rho^x(y).    
\end{equation}
Moreover, the set of nonvanishing ends of $T \rest{[x]_{E_T}}$ is a closed subset in $\del_T [x]_{E_T}$ for a.e.\ $x \in X$.
\end{theorem}

\begin{proof}
We will establish both conclusions simultaneously. 
In fact, the last statement follows from the first and the upper semi-continuity of $\urho^x$ and compactness of edge-removal topology, but we give an alternative proof.

By \cref{smooth=rho-finite}, we may assume that $E_T$ is $\mu$-nowhere smooth. Let $H$ be the set of directed edges $e$ of $T$ for which 
\[
\sup_{y \in \Vo(e)} \rho^{\orig(e)}(y) < \sup_{y \in [\orig(e)]_{E_T}} \rho^{\orig(e)}(y).
\]
Note that for $e \in H$ we in fact have $\sup_{y \in \Vo(e)} \rho^x(y) < \sup_{y \in [\orig(e)]_{E_T}} \rho^x(y)$ for each $x \in [\orig(e)]_{E_T}$, which implies that $H$ is coherent.
By partitioning $X$ into two $E_T$-invariant Borel sets, it is enough to consider the following two cases.

\begin{case}{1}{Every $e \in H$ is $\leq_T$-below a $\leq_T$-maximal edge of $H$}
Let $Y \defeq X \setminus \Vo(H)$ so that $Y$ is a $T$-convex Borel $E_T$-complete section. 
By definition, every end $\xi$ of $T \rest{Y}$ satisfies $\labelcref{eq:sup-weight_attained}$, and by \cref{pruning}, after discarding an $E_T$-invariant null set, all other ends of $T$ are vanishing.
\end{case}

\begin{case}{2}{$H$ meets every $T$-component, but there are no $\leq_T$-maximal edges of $H$} 
By \cref{coherent_orientation_prop}\labelcref{item:coherent_dichotomy} and the paragraph following \cref{coherent_orientation_prop}, the map $x \mapsto \xi_x \defeq \lim _{e\in H_x}\orig (e)$ is an $E_T$-invariant Borel selection of one end from each $T$-component, where $H_x$ is the intersection of $H$ with the $T$-connected component of $x$. 
By definition, for each $x \in X$, the end $\xi_x$ is the unique end of $T \rest{[x]_{E_T}}$ satisfying \labelcref{eq:sup-weight_attained}.
Let $f : X \to X$ be the function mapping each $x \in X$ to its successor on the geodesic ray through $T$ from $x$ to $\xi_x$. 
It follows that $f$ is $\mu$-nowhere essentially two-ended since otherwise both of these ends would satisfy \labelcref{eq:sup-weight_attained} by \cref{two-ended_function} (using the assumption that $E_T$ is $\mu$-nowhere smooth), contradicting the uniqueness of $\xi_x$. 
Finally, by \cref{back_ends_converge_to_0}, for a.e.\ $x \in X$, the unique nonvanishing end of $T \rest{[x]_{E_T}}$ is $\xi_x$.
\end{case}
\end{proof}

The definition of nonvanishing ends yields that the set of nonvanishing ends within each connected component is $F_\sigma$; however, the following measure-theoretic argument shows that this set is closed within almost every connected component.

\begin{cor}\label{at_least_2_nonvanishing_core} 
Assume \cref{Adams_workinghyp}, and suppose that each $T$-component has more than one nonvanishing end. 
Then, after discarding a null set, an end of $T$ is nonvanishing if and only if it is barytropic (i.e., it belongs to $\Core_\mu(T)$).
In particular, in a.e.\ $T$-component, the set of nonvanishing ends is closed. 
\end{cor}
\begin{proof}
By \cref{constant_weight}, we may assume without loss of generality that the set of nonvanishing ends of every $T$-component is closed.
Let $Y$ be the set of all vertices that lie in the convex hull of the set of all nonvanishing ends of $T$, so that $\del_T Y$ is the set of all nonvanishing ends of $T$.
Since $Y$ is an analytic set, it is measurable, so we may assume that $Y$ is Borel after discarding an $E_T$-invariant null set.
Furthermore, $Y$ is a $T$-convex subset of $\rnc_T$, so by \cref{core_minimality}, after discarding another null set, $Y = \rnc_T$.
\end{proof}

Recall from \cref{def:barytropic-vanishing-finite_geod} that an end $\xi$ of $T$ is said to have $\rho$-finite geodesics if every (equivalently: some) geodesic ray converging to $\xi$ is $\rho$-finite; otherwise, we say that $\xi$ has $\rho$-infinite geodesics.

\begin{cor}\label{vanishing_implies_cocycle-finite_geodesics}
Assume \cref{Adams_workinghyp}.
Then, after discarding a null set, every $\rho$-vanishing end has $\rho$-finite geodesics.
\end{cor}

\begin{proof}
By \cref{at_least_2_nonvanishing_core,smooth_iff_all_vanishing_ends}, it is enough to consider the case where every connected component of $T$ has exactly one nonvanishing end, in which case $T=T_f$ where $f : X \to X$ is the Borel map which moves each point of $X$ one step closer to the unique nonvanishing end of its $T$-component.
Then $f$ is $\mu$-nowhere essentially two-ended by \cref{two-ended_function}, and hence all $\rho$-vanishing ends of $T_f$ have $\rho$-finite geodesics by \cref{summable_back-rays}.
\end{proof}

\begin{theorem}[Perfectly many nonvanishing ends]\label{nonamenable_perfect_nonvanishing}
Assume \cref{Adams_workinghyp}, and suppose that $E_T$ is $\mu$-nowhere amenable. 
Then, after discarding a null set, the set of nonvanishing ends meets the boundary of each $T$-component in a nonempty, closed, and perfect set. 
\end{theorem}
\begin{proof}
By \cref{smooth_iff_all_vanishing_ends,end_selection}, $\mu$-nowhere amenability implies that a.e.\ $T$-component has more than two nonvanishing ends. 
By \cref{at_least_2_nonvanishing_core}, after discarding a null set, the set of nonvanishing ends of $T$ coincides with $\del_T \rnc_T$. 
In addition, \cref{isolated=>vanishing} applied to $T \rest{\rnc_T}$ implies that, after discarding another null set, $\del_T \rnc_T$ has no isolated ends.
\end{proof}

\subsection{Barytropic and nonvanishing as measure class invariants}

\begin{cor}
Assume \cref{Adams_workinghyp}, and suppose that $E_T$ is ergodic.
If $\rho$ is not a coboundary then, after discarding a null set, $\limsup_{y\to\xi}\rho^x(y)=\infty$ for every $x\in X$ and every nonvanishing end $\xi$ of $T\rest{[x]_{E_T}}$, i.e., each nonvanishing end of $T$ has infinite $\rho$-weight.
\end{cor}
\begin{proof}
Towards the contrapositive, suppose that there is a positive measure set of $x \in X$ such that $T \rest{[x]_{E_T}}$ has nonvanishing ends of finite $\rho$-weight. 
By ergodicity, this happens on a conull set. 
By \cref{constant_weight}, after discarding a null set, for each point $x \in X$ the supremum $s(x) \defeq \sup_{y \in [x]_{E_T}} \rho^x(y)$ is finite. 
By the cocycle identity, $s(x) = \rho^x(y) s(y)$ for $E_T$-equivalent points $x,y \in X$, so $\rho$ is a coboundary.
\end{proof}

The following theorem shows that the notions of vanishing and nonvanishing of ends essentially only depend on a given measure class rather than on the Radon--Nikodym cocycle of a particular representative measure.

\begin{theorem}\label{measure_class_invariant}
Let $T$ be a locally countable acyclic Borel graph on a standard Borel space $X$, and let $\mathscr{C}$ be an $E_T$-invariant measure-class of Borel probability measures on $X$.
Let $\mu$ and $\nu$ be probability measures in $\mathscr{C}$ and let $\rho_\mu$ and $\rho_\nu$ be their associated Radon--Nikodym cocycles on $E_T$.
Then, after discarding a null set, we have:

\begin{enumerate}[(a), leftmargin=*]
\item\label{item:vanishing_mcp-invariance} An end $\xi$ of $T$ is $\rho_\mu$-nonvanishing if and only if $\xi$ is $\rho_\nu$-nonvanishing.

\item\label{item:barytropic_mcp-invariance} An end $\xi$ of $T$ is $\rho_\mu$-barytropic if and only if $\xi$ is $\rho_\nu$-barytropic.

\item\label{item:RN-core_mcp-invariance} $\Core_\mu (T) = \Core_\nu (T)$.
\end{enumerate}
\end{theorem}

\begin{proof}
Part \labelcref{item:RN-core_mcp-invariance} follows from \labelcref{item:barytropic_mcp-invariance}.

For part \labelcref{item:vanishing_mcp-invariance}, it is enough to show that every $\rho_\mu$-vanishing end is also $\rho_\nu$-vanishing.
By decomposing $X$ into $E_T$-invariant Borel sets, the analysis reduces to the following cases.

\begin{case}{1}{Every end of $T$ is $\rho_\mu$-vanishing} 
By the backward implication of \cref{smooth_iff_all_vanishing_ends}, $E_T$ is smooth on a conull set, and hence by the forward implication of \cref{smooth_iff_all_vanishing_ends}, all ends of $T$ are $\rho_\nu$-vanishing after discarding a null set.
\end{case}

\begin{case}{2}{Every $T$-component has exactly one $\rho_\mu$-nonvanishing end}
In this case, $E_T$ is $\mu$-nowhere smooth by \cref{smooth=rho-finite}, and $T = T_f$ for the Borel map $f : X \to X$ which moves each point of $X$ one step closer to the unique $\rho_\mu$-nonvanishing end in its $T$-component.
It follows from \cref{two-ended_function} that $f$ is $\mu$-nowhere essentially two-ended, and hence also $\nu$-nowhere essentially two-ended.
Applying \cref{back_ends_converge_to_0}, we conclude that for a.e.\ $x \in X$, all ends of $T \rest{[x]_{E_T}}$ other than $\xi ^+(x)\defeq \lim _{n\rightarrow\infty}f^n(x)$ are $\rho_\nu$-vanishing.
\end{case}

\begin{case}{3}{Every $T$-component has at least two $\rho_\mu$-nonvanishing ends}
By \cref{at_least_2_nonvanishing_core}, after discarding a null set, the set of $\rho_\mu$-vanishing ends coincides with the complement of $\del_T \rnc_T$ in $\del_T X$.
All ends in this set are also $\rho_\nu$-vanishing (after discarding another null set) by \cref{pruning} applied to $\rnc_T$ and the measure $\nu$.
\end{case}

For part \labelcref{item:barytropic_mcp-invariance}, by \cref{at_least_2_nonvanishing_core} and part \labelcref{item:vanishing_mcp-invariance}, it suffices to handle the case where every $T$-component has exactly one nonvanishing end.
Let $f$ be the function as in Case 2 above, so $T = T_f$.
It then suffices to prove for back $f$-ends that $\rho_\mu$-abarytropic implies $\rho_\nu$-abarytropic.
It is enough to consider the following cases:

\begin{case}{1}{All back $f$-ends are $\rho_\mu$-abarytropic}
Then by \cref{char-rho-finite-back-orbits}, after discarding a null set there is a forward $f$-invariant Borel $E_f$-complete section $Z \subseteq X$ on which the restriction $f:Z\rightarrow Z$ is one-ended. 
By the same theorem, it now follows that all back $f$-ends are $\rho_\nu$-abarytropic.
\end{case}

\begin{case}{2}{Every $T_f$-component has a $\rho_\mu$-barytropic back $f$-end}
Then the $\rho_\mu$-abarytropic ends of $T$ are exactly the ends in $\partial_T X \setminus \partial_T \rnc_T$, all of which are $\rho_\nu$-abarytropic by \cref{pruning}\labelcref{pruning:convex}.
\end{case}
\end{proof}

\cref{measure_class_invariant} immediately implies the following.

\begin{cor}
Assume \cref{Adams_workinghyp}.
If $\mu$ is equivalent to an $E_T$-invariant probability measure then every $T$-end in a.e.\ $T$-component is $\rho$-nonvanishing.
\end{cor}

\subsection{Generalizations of the Adams dichotomy}

The following trichotomy combines our generalizations of the classical Adams dichotomy \cite{Adams90} and Miller's theorem \cite[Theorem A]{Miller:ends_of_graphs_I} to the measure class preserving setting.

\begin{theorem}[Trichotomy for acyclic mcp graphs]\label{trichotomy}
Let $T$ be an acyclic locally countable mcp Borel graph on a standard probability space $(X, \mu)$. %, with associated Radon--Nikodym cocycle $\rho :E_T\rightarrow \Rpos$.
\begin{enumerate}[(a)]
    \item\label{trichotomy:smooth} $E_T$ is smooth on a conull set if and only if all ends of a.e.\ $T$-component are vanishing.
    
    \item\label{trichotomy:amenable} $E_T$ is amenable and $\mu$-nowhere smooth if and only if a.e.\ component of $T$ has exactly one or two nonvanishing ends.

    \item\label{trichotomy:nowhere-amenable} $E_T$ is $\mu$-nowhere amenable if and only if for a.e.\ $T$-component, the set of all nonvanishing ends of that component is nonempty, closed, and has no isolated points.
\end{enumerate}
\end{theorem}
\begin{proof}
It is enough to prove the forward implication of each part since $X$ can always be partitioned into three $E_T$-invariant Borel sets on which these forward implications can be applied.

The forward implications of \labelcref{trichotomy:smooth,trichotomy:nowhere-amenable} follow respectively from \cref{smooth=rho-finite,nonamenable_perfect_nonvanishing}. 
Towards proving the forward implications of \labelcref{trichotomy:amenable}, assume that $E_T$ is $\mu$-amenable and $\mu$-nowhere smooth. 
By \cref{smooth_iff_all_vanishing_ends}, discarding a null set, we may further assume that each $T$-component has at least one nonvanishing end. 
By decomposing $X$ into two $E_T$-invariant Borel sets, the analysis reduces to the following cases.

\begin{case}{1}{$T$ is essentially two-ended} In this case, there is a Borel set $L \subseteq X$ that meets each $T$-component in a single bi-infinite line. By \cref{pruning}, every end of $T$ other than the two spanned by $L$ in each $T$-component, is vanishing.
\end{case}

\begin{case}{2}{$T$ is $\mu$-nowhere essentially two-ended} By \cref{end_selection}, since $E_T$ is amenable, there is an $E_T$-invariant Borel selection of exactly one end from each $T$-component, and hence $T = T_f$ for the Borel map $f : X \to X$ which moves each point of $X$ one step closer to its selected end. Then we apply \cref{back_ends_converge_to_0} to conclude that in a.e.\ $T$-component, the selected end is the only nonvanishing end.\qedhere
\end{case}
\end{proof}

\begin{theorem}[Dichotomy via cocycle-infinite geodesics]\label{dichotomy:coc-finite_geodesics}
Let $T$ be an acyclic locally countable mcp Borel graph on a standard probability space $(X,\mu )$, with associated Radon--Nikodym cocycle $\rho :E_T\rightarrow \Rpos$.
Then:
\begin{enumerate}[(a)]        
    \item \label{dichotomy:amenable} $E_T$ is amenable if and only if a.e.\ component of $T$ has at most two ends with $\rho$-infinite geodesics.

    \item \label{dichotomy:nonamenable} $E_T$ is $\mu$-nowhere amenable if and only if for a.e.\ $T$-component, the set of all ends with $\rho$-infinite geodesics is nonempty, $G_{\delta}$, and has no isolated points.
\end{enumerate}
\end{theorem}
\begin{proof}
As in the proof of \cref{trichotomy}, it is enough to prove the forward implications in both \labelcref{dichotomy:amenable,dichotomy:nonamenable}.
The forward implication in part \labelcref{dichotomy:amenable} follows from parts \labelcref{trichotomy:smooth,trichotomy:amenable} of \cref{trichotomy}, in tandem with \cref{vanishing_implies_cocycle-finite_geodesics}.

For \labelcref{dichotomy:nonamenable}, by \cref{nonamenable_perfect_nonvanishing} and after discarding a null set and restricting to the $T$-convex hull of the set of all nonvanishing ends of $T$, we may assume that every end of $T$ is nonvanishing, $T$ has no leaves, and the set of ends in each $T$-component is perfect. 
By \cref{constant_weight}, we may additionally assume that \labelcref{eq:sup-weight_attained} holds for all $x \in X$ and every end $\xi$ of $T \rest{[x]_{E_T}}$.
Fix $x \in X$.

Firstly, we claim that the set $\IC_x$ of ends in $\del_T [x]_{E_T}$ with $\rho$-infinite geodesics is dense in $\del_T [x]_{E_T}$, and hence this set has no isolated points in its subspace topology since $\del_T [x]_{E_T}$ is perfect.
Indeed, we fix a directed edge $e$ of $T \rest{[x]_{E_T}}$ and we show that $\del_T \Vt(e)$ contains an end with $\rho$-infinite geodesics.
Letting $x_{-1} \defeq \orig(e)$, we inductively define a sequence $(x_n)$ of points in $\Vt(e)$ that spans a geodesic ray with $\rho^x(x_n) \le \rho^x(x_{n+1})$ for all $n \in \N$.
Let $x_0 \defeq \term(e)$. 
Given $x_0, \dots, x_n$, let $e_n$ be the unique edge with $\term(e_n) = x_n$ and $x_{-1} \in \Vo(e_n)$. Let $x_{n+1}$ be a vertex in $\Vt(e_n)$ with $\rho^x(x_{n+1}) \ge \rho^x(x_n)$, which exists by \labelcref{eq:sup-weight_attained}.
This proves the claim.

It remains to show that $\IC_x$ is $G_\delta$.
But this is immediate from the fact that an end $\xi \in \del_T [x]_{E_T}$ is in $\IC_x$ if and only if for each $w \in \N$ there is $\ell \in \N$ such that the initial segment of the geodesic $[x,\xi)_T$ of length $\ell$ has $\rho^x$-weight at least $w$.
\end{proof}

\begin{remark}\label{infinite_geodesics_not_robust}
We emphasize that \cref{dichotomy:coc-finite_geodesics} cannot be made into a trichotomy analogous to \cref{trichotomy}: by example \cref{example:least_deletion} there exists a one-ended acyclic Borel graph on a probability space $(X,\mu )$ that is $\mu$-nowhere smooth, but has no end with $\rho$-infinite geodesics.

In addition, alternative \labelcref{dichotomy:nonamenable} of \cref{dichotomy:coc-finite_geodesics} cannot be strengthened by replacing ``$G_\delta$'' with ``closed.'' 
Indeed, let $T_{2,3}$ denote the directed tree in which every vertex has two incoming edges and three outgoing edges, and let $G$ be its automorphism group.
Then, in the natural directed treeing $T$ of the cross-section equivalence relations associated to a free pmp action of $G$, every $T$-component is isomorphic to $T_{2,3}$, and the cocycle $\rho$ takes the value $3/2$ along each directed edge of $T$; it follows that both the set of ends with $\rho$-infinite geodesics and its complement are dense in each component, hence neither is closed.
\end{remark}

%%%%%%%%%%%%%%%%%%%%%%%%%%%%%%%%%%%%%%%%%%%%%%%%%%%%%%%%%%%%%%%%%%%%%%%%%%%%%%%%%%%%%%%%%%%%

%\bigskip

\def\MR#1{}
\bibliographystyle{alphaurl-arxiv}
\bibliography{refs}

\end{document}